\documentclass[10pt]{amsart}
\usepackage[margin=1.2in]{geometry}
\usepackage{color}
\usepackage{amssymb}
\newtheorem{theorem}{Theorem}[section]
\newtheorem{lemma}[theorem]{Lemma}
\newtheorem{proposition}{Proposition}[section]

\theoremstyle{definition}

\theoremstyle{remark}
\newtheorem{remark}[theorem]{Remark}

\numberwithin{equation}{section}

\usepackage{graphicx}
\DeclareGraphicsExtensions{.pdf,.png,.jpg}
\begin{document}
	\title[Relativistic BGK model for polyatomic gases near equilibrium]{On a relativistic BGK model for polyatomic gases near equilibrium}
	\author[B.-H. Hwang]{Byung-Hoon Hwang}
	\address{Department of Mathematics, Sungkyunkwan University, Suwon 440-746, Republic of Korea}
	\email{bhh0116@skku.edu}
	\author[T. Ruggeri]{Tommaso Ruggeri}
	\address{Department of Mathematics and Alma Mater, Research Center on Applied Mathematics
		AM$^2$, University of Bologna, Bologna, Italy}	
	\email{tommaso.ruggeri@unibo.it}
		\author[S.-B. Yun]{Seok-Bae Yun}
	\address{Department of Mathematics, Sungkyunkwan University, Suwon 440-746, Republic of Korea}
	
	\email{sbyun01@skku.edu}
	\keywords{relativistic kinetic theory of gases, relativistic Boltzmann equation, relativistic BGK model, nonlinear energy method}
	
	\begin{abstract}
		Recently, a novel relativistic polyatomic BGK model was suggested by Pennisi and Ruggeri [J. of Phys. Conf. Series, 1035, (2018)]  to overcome drawbacks of the Anderson-Witting model and Marle model.
        In this paper, we prove the unique existence and asymptotic behavior of classical solutions to the relativistic polyatomic BGK model when the initial data is sufficiently close to a global equilibrium.
	\end{abstract}
	\maketitle
	\section{Introduction}
	  	 In the classical kinetic theory of gases, the BGK relaxation operator \cite{BGK} has been successfully used in place of the Boltzmann collision operator, yielding satisfactory simulation of the Boltzmann flows at much lower numerical cost. 
	The relativistic generalization of the BGK approximation was first made by Marle \cite{Mar2,Mar3} and successively by Anderson and Witting \cite{AW}. The  Marle model is an extension of the classical BGK model in the Eckart frame \cite{CK,E}, and the   Anderson-Witting model obtains such extention using the Landau-Lifshitz frame \cite{CK,LL}.
	
	These models have been widely employed for various relativistic problems  \cite{ CKT,DHMNS2,FRS,FRS2,HM,JRS}, but several drawbacks were also recognized in the literature. For the Marle model, the relaxation time becomes unbounded for particles with zero rest mass \cite{CK}. The problem for the  Anderson-Witting model is that the Landau-Lifshitz frame was established on the assumption that some of high order non-equilibrium quantities are negligible near equilibrium \cite{AW,CK}, which inevitably leads to some loss of consistency.   
	
	Starting from these considerations, Pennisi and Ruggeri proposed a variant of Anderson-Witting model in the Eckart frame both for monatomic and polyatomic gas, and proved that the conservation laws of particle number and energy-momentum are satisfied and the H-theorem holds \cite{PR} (see also \cite{newbook}). In the case of polyatomic gas, the Cauchy problem for the relativistic BGK model reads:
	\begin{align}\label{PR}
	\begin{split}
		\partial_t F+\hat{p}\cdot\nabla_x F&=\frac{U_\mu p^\mu}{c \tau p^0}\left\{\left(1-p^\mu q_\mu \frac{1+\frac{\mathcal{I}}{mc^2}}{bmc^2}\right)F_E-F\right\},\cr
		F(0,x,p)&= F_0(x,p),
	\end{split}
\end{align}
	where $F\equiv F(x^\alpha , p^\beta,\mathcal{I} )$ is the momentum distribution function representing the number density of relativistic particles at the phase point $(x^\alpha,p^\beta)~(\alpha,\beta=0,1,2,3)$ with the microscopic internal energy $\mathcal{I}\in \mathbb{R}^+$ that  takes into account the energy due to the internal degrees of freedom of the particles. Here $x^\alpha=(ct,x)\in\mathbb{R}^+\times \mathbb{R}^3$ is the space-time coordinate, and $p^\beta=(\sqrt{(mc)^2+|p|},p)\in\mathbb{R}^+\times \mathbb{R}^3$ is the four-momentum. Greek indices
run from $0$ to $3$ and the repeated indices are assumed to be summed over its whole range;  $c$ is the light velocity, $\hat{p}:=cp/p^0$ is the normalized momentum, $m$ and $\tau$ are respectively the mass and the relaxation time in the rest frame where the momentum of particles is zero.   The macroscopic quantity $b$ is given in \eqref{b def}. Throughout this paper, the metric tensor $g_{\alpha\beta}$ and its inverse $g^{\alpha\beta}$  are given by 
$$
g_{\alpha\beta}=g^{\alpha\beta}=\text{diag}(1,-1,-1,-1)
$$
and we use the raising and lowering indices as
$$
g_{\alpha\mu}p^\mu=p_\alpha,\qquad g^{\alpha\mu}p_\mu=p^\alpha,
$$
which implies $p_\alpha=(p^0,-p)$. The Minkowski inner product is defined by
$$
p^\mu q_\mu=p_\mu q^\mu=p^0q^0-\sum_{i=1}^3 p^iq^i.
$$
To present the macroscopic fields of $F$, we define the particle-particle flux $V^\mu$ and energy-momentum tensor $T^{\mu\nu}$ \cite{PR,PR2} by
	\begin{equation}\label{VT}
V^\mu=mc\int_{\mathbb{R}^3}\int_0^\infty p^\mu F \phi(\mathcal{I}) \,d\mathcal{I}\,\frac{dp}{p^0},\qquad T^{\mu\nu}=\frac{1}{mc} \int_{\mathbb{R}^3}\int_0^\infty p^\mu p^\nu F\left( mc^2 + \mathcal{I} \right)   \phi(\mathcal{I})\,d\mathcal{I}\,\frac{dp}{p^0}.
\end{equation}
Here $\phi(\mathcal{I})\ge 0$ is the state density of the internal mode such that  $\phi(\mathcal{I}) \, d  \mathcal{I}$ represents the number of the internal states of a molecule having the internal energy between $\mathcal{I}$ and $\mathcal{I}+d \mathcal{I}$   which can take various forms according to the physical context.  For example, the following form of $\phi(\mathcal{I})$ is employed in \cite{PR}
	\begin{equation*}
	\phi(\mathcal{I})=\mathcal{I}^{(f^i-2)/2} 
	\end{equation*}
to get the correct classical limit of internal energy of polyatomic gas.	Here $f^i\ge 0$ is the internal degrees of freedom due to the internal motion of molecules. In this paper, instead of choosing the specific form of $\phi(\mathcal{I})$, we develop an existence theory for \eqref{PR} that is valid for a general class of $\phi(\mathcal{I})$ satisfying a certain  condition which covers all the physically relevant cases (see \eqref{phi condition}).
 
Going back to \eqref{VT}, we introduce the decomposition of $V^\mu$ and $T^{\mu\nu}$: 
\begin{align}\label{Eckart}
	\begin{split}
		V^\mu&=nmU^\mu\cr
		T^{\mu\nu}&=\sigma^{\langle \mu\nu\rangle}+(p+\Pi)h^{\mu\nu}+\frac{1}{c^2}\left(q^\mu U^\nu+q^\nu U^\mu\right)+\frac{e}{c^2}U^\mu U^\nu
\end{split}\end{align}
which is called the Eckart frame \cite{CK,E}. In \eqref{Eckart}, $nm=c^{-1}\sqrt{V^\mu V_\mu}$ denotes the number density, $U^\mu=(\sqrt{c^2+|U|^2},U)$ the Eckart four-velocity, $p$ the pressure, $\Pi$ the dynamical pressure, $h^{\mu\nu}=-g^{\mu\nu}+\frac{1}{c^2}U^\mu U^\nu$ the projection tensor, $\sigma^{\mu\nu}=T^{\alpha\beta}\left(h^\mu_\alpha h^\nu_\beta-\frac{1}{3}h^{\mu\nu}h_{\alpha\beta}\right)$ the viscous deviatoric stress, $e$ the energy, and $q^\mu=-h^\mu_\alpha U_\beta T^{\alpha\beta}$ the heat flux. We recall that only $14$ field variables in \eqref{Eckart} are independent due to the constraints: 
\begin{equation*}
	U_\alpha U^\alpha = c^2, \quad   U_\alpha q^\alpha=0, \quad    U_\alpha \sigma^{<\alpha \beta>}  =0,  \quad  g_{\alpha \beta} \,  \sigma^{<\alpha \beta>}  =0.
\end{equation*}
The macroscopic fields that appear frequently in this paper are defined as a suitable moment of $F$ in the following manner.
\begin{align}\label{macroscopic fields}\begin{split}
		n^2&=\left(\int_{\mathbb{R}^3}\int_0^\infty F\phi(\mathcal{I}) \,d\mathcal{I}\,dp\right)^2-\sum_{i=1}^3 \left(\int_{\mathbb{R}^3}\int_0^\infty p^iF \phi(\mathcal{I}) \,d\mathcal{I}\,\frac{dp}{p^0}\right)^2,\cr U^\mu&=\frac{c}{n}\int_{\mathbb{R}^3}\int_0^\infty p^\mu F\phi(\mathcal{I}) \,d\mathcal{I}\,\frac{dp}{p^0},\cr
		e&=\frac{1}{c} \int_{\mathbb{R}^3}\int_0^\infty \left(U^\mu p_\mu\right)^2 F\left(1+\frac{\mathcal{I}}{mc^2}\right) \phi(\mathcal{I}) \,d\mathcal{I}\,\frac{dp}{p^0},\cr
		q^\mu&=c\int_{\mathbb{R}^3}\int_0^\infty p^\mu\left(U^\nu p_\nu\right) F\left(1+\frac{\mathcal{I}}{mc^2}\right) \phi(\mathcal{I}) \,d\mathcal{I}\,\frac{dp}{p^0},\cr
		&-\frac{1}{ c }U^\mu \int_{\mathbb{R}^3}\int_0^\infty \left(U^\nu p_\nu\right)^2 F\left(1+\frac{\mathcal{I}}{mc^2}\right) \phi(\mathcal{I}) \,d\mathcal{I}\,\frac{dp}{p^0}.
\end{split}\end{align}
The equilibrium distribution function $F_E$ of \eqref{PR} reads \cite{PR}:
\begin{equation}\label{GJdf} 
F_E=\exp\left\{-1+\frac{m}{k_B}\frac{ g_r}{T}-\left(1+\frac{\mathcal{I}}{mc^2}\right)\frac{1}{k_B T}U^\mu p_\mu \right\} .
\end{equation}
We note that \eqref{GJdf} reduces to the well known   J\"{u}ttner distribution function \cite{Juttner} in the  monatomic case. The equilibrium temperature $T$ is determined by the following nonlinear relation:
\begin{align}\label{gamma relation}\begin{split}
\widetilde{e}(T)\equiv\frac{\int_{\mathbb{R}^3}\int_0^\infty \sqrt{1+|p|^2}e^{-\left(mc^2+ \mathcal{I} \right)\frac{1}{k_B T}\sqrt{1+|p|^2}} \left(mc^2+ \mathcal{I} \right)\phi(\mathcal{I})\, d\mathcal{I}dp}{ \int_{\mathbb{R}^3}\int_0^\infty e^{-\left(mc^2+ \mathcal{I} \right)\frac{1}{k_B T}\sqrt{1+|p|^2}}\phi(\mathcal{I}) \,d\mathcal{I}dp}=\frac{e}{n},
\end{split}\end{align}
which is required for \eqref{PR} to satisfy the conservation laws. With such $T$, the relativistic chemical potential $g_r$ is defined by the following equation
\begin{equation*} 
e^{-1+\frac{m}{k_B}\frac{g_r}{T}}=\frac{n}{(mc)^3\int_{\mathbb{R}^3}\int_0^\infty  e^{-\left(mc^2+ \mathcal{I} \right)\frac{1}{k_B T}\sqrt{1+|p|^2}}\phi(\mathcal{I})\,d \mathcal{I}dp}.
\end{equation*}
Here $k_B$ is the Boltzmann constant. The uniqueness of the equilibrium temperature $T$ will be considered later (see Proposition \ref{solvability gamma}). Then, $F_E$ satisfies 
\begin{align}\label{conservation laws}\begin{split}
&U_\mu\int_{\mathbb{R}^3}\int_0^\infty p^\mu\left\{\left(1-p^\mu q_\mu \frac{1+\frac{\mathcal{I}}{mc^2}}{bmc^2}\right)F_E-F\right\}\phi(\mathcal{I})\,d\mathcal{I}\frac{dp}{p^0}=0,\cr
 &U_\mu\int_{\mathbb{R}^3}\int_0^\infty p^\mu p^\nu\left\{\left(1-p^\mu q_\mu \frac{1+\frac{\mathcal{I}}{mc^2}}{bmc^2}\right)F_E-F\right\}  \left(1+\frac{\mathcal{I}}{mc^2}\right)\phi(\mathcal{I})\,d\mathcal{I}\frac{dp}{p^0}=0,
\end{split}\end{align}
so that the following conservation laws for $V^\mu$ and $T^{\mu\nu}$ hold true
\begin{equation}\label{conservation law}
\partial_\mu V^\mu=0,\qquad \partial_\mu T^{\mu\nu}=0.
\end{equation}
Finally, the macroscopic quantity $b$ of \eqref{PR} is defined as
\begin{align}\label{b def} 
	b=\frac{n mc^2}{\gamma^2}\left(\int_0^\infty \frac{K_2(\gamma^*)}{\gamma^*}\phi(\mathcal{I})\, d\mathcal{I}\right)^{-1}\int_0^\infty K_3(\gamma^*)\phi(\mathcal{I})\,d\mathcal{I},
\end{align}
where $\gamma$ and $\gamma^*$ denote
$$ 
\gamma=\frac{mc^2}{k_BT}, \qquad\gamma^*=\gamma\left(1+\frac{\mathcal{I}}{mc^2}\right)
$$
and $K_n$ is the modified Bessel function of the second kind:
\begin{align*}
	K_{n}(\gamma)&=\int_{0}^{\infty}\cosh(nr)e^{-\gamma \cosh(r) }dr.
\end{align*}

\noindent\newline

\subsection{Main result}
 The aim of this paper is to study the global in time existence and asymptotic behavior of classical solutions to the Pennisi-Ruggeri model \eqref{PR} when the initial data starts sufficiently close to a global equilibrium.
For this, we decompose the solution $F$ around the global equilibrium:
	\begin{equation}\label{decomposition}
	F=F_E^0+f\sqrt{F_E^0}.
	\end{equation}
	Here $f$ is the perturbation part, and $F_E^0$ is the global equilibrium defined by
	\begin{equation*}
		F_E^0\equiv F_E(g_{r0},0,T_0;p)=\exp\left\{-1+\frac{m}{k_B}\frac{ g_{r0}}{T_0}-\left(1+\frac{\mathcal{I}}{mc^2}\right)\frac{cp^0}{k_B T_0}  \right\}
	\end{equation*}
where $g_{r0}$ and $T_0$ are positive constants. Inserting \eqref{decomposition}, then \eqref{PR} can be rewritten as
	\begin{align*}\label{AWRBGK2}
		\begin{split}
			\partial_t f+\hat{p}\cdot\nabla_x f&=\frac{1}{\tau}\left(L(f)+\Gamma (f)\right),\cr
			f_0(x,p)&=f(0,x,p)
	\end{split}\end{align*}
	where $L$ is the linearized operator and $\Gamma$ is the nonlinear perturbation whose definition can be found in \eqref{Pf}, Lemma \ref{lin2} and Proposition \ref{lin3}. The initial perturbation $f_0$ is given by $F_0=F_E^0+f_0\sqrt{F_E^0}.$ 	We now define the notations to state our main result.
%
	\begin{itemize}
		\item We define the weighted $L^2$ inner product:
		\begin{align*}
			\langle f,g\rangle_{L^2_{p,\mathcal{I}}}&=\int_{\mathbb{R}^3}\int_0^\infty f({p,\mathcal{I}})g({p,\mathcal{I}})\phi(\mathcal{I})\, d\mathcal{I}dp,\cr
			\langle f,g\rangle_{L^2_{x,p,\mathcal{I}}}&=\int_{\mathbb{R}^3}\int_{\mathbb{R}^3}\int_0^\infty f(x,{p,\mathcal{I}})g(x,{p,\mathcal{I}})\phi(\mathcal{I})\, d\mathcal{I}dpdx 
		\end{align*}
and the corresponding norms:
		\[\|f\|^{2}_{L^2_{p,\mathcal{I}}}=\int_{\mathbb{R}^{3}}\int_0^\infty|f({p,\mathcal{I}})|^2\phi(\mathcal{I})\,d\mathcal{I}dp,\quad\|f\|_{L^2_{x,p,\mathcal{I}}}^{2}=\int_{\mathbb{R}^3}\int_{\mathbb{R}^3}\int_0^\infty|f(x,{p,\mathcal{I}})|^{2}\phi(\mathcal{I})\,d\mathcal{I}dpdx.
		\]
		\item We define the operator $\Lambda^s$ $(s\in \mathbb{R})$ by
		$$
		\Lambda^s f(x)=\int_{\mathbb{R}^3}|\xi|^s\hat{f}(\xi)e^{2\pi ix\cdot\xi}\,d\xi
		$$ 
		where $\hat{f}$ is the Fourier transformation of $f$.
		\item We denote $\dot{H}^s_x$ to be the homogeneous space endowed with the norm:
		$$
		\|f\|_{\dot{H}^s_x}:=\|\Lambda^sf\|_{L^2_x}=\left\| |\xi|^s\hat{f}(\xi)\right\|_{L^2_{\xi}}
		$$
		where $\|\cdot\|_{L^2_x}$ and $\|\cdot\|_{L^2_\xi}$ are the usual $L^2$-norm.
		\item We use the notation $L^2_{p,\mathcal{I}}H^s_x$ to denote
		$$
		\|f\|_{L^2_{p,\mathcal{I}}H^s_x}=\left\|\|f\|_{H^s_x} \right\|_{L^2_{p,\mathcal{I}}}
		$$
		where $\|\cdot\|_{H^s_x}$ is the usual Sobolev norm. 
		\item We define the energy functional $E$ and dissipation rate $\mathcal{D}$ by
		\begin{align*}
			E(f)(t)&=\sum_{|\alpha|+|\beta|\le N}\| \partial^\alpha_\beta f\|^{2}_{L^2_{x,{p,\mathcal{I}}}},\cr
			\mathcal{D}(f)(t)&=\|\{I-P\}f\|^2_{L^2_{x,p,\mathcal{I}}}+\sum_{1\le |\alpha|+|\beta|\le N}\|\partial^\alpha_\beta f\|^2_{L^2_{x,p,\mathcal{I}}},
		\end{align*}
	where the multi-index $\alpha=[\alpha_0,\alpha_1,\alpha_2,\alpha_3]$ and $\beta=[\beta_1,\beta_2,\beta_3]$ are used to denote  
$$
\partial^\alpha_\beta=\partial^{\alpha_0}_{t}\partial^{\alpha_1}_{x^1}\partial^{\alpha_2}_{x^2}\partial^{\alpha_3}_{x^3}\partial^{\beta_1}_{p^1}\partial^{\beta_2}_{p^2}\partial^{\beta_3}_{p^3}.
$$
		We also define the energy functional and dissipation rate for spatial derivatives by
		\begin{align*}
			E_N(f)(t)&=\sum_{0\le k\le N}\| \nabla^k_xf\|^{2}_{L^2_{x,{p,\mathcal{I}}}},\cr
			\mathcal{D}_N(f)(t)&=\|\{I-P\}f\|^2_{L^2_{x,p,\mathcal{I}}}+\sum_{1\le k\le N}\|\nabla^k_x f\|^2_{L^2_{x,p,\mathcal{I}}}.
		\end{align*}
	\end{itemize}
	Then, our main result is as follows. 
%
\begin{theorem}\label{main3}
Let $N\ge 3$ be an integer. Assume that the state density $\phi$ satisfies
\begin{equation}\label{phi condition}
\int_{\mathbb{R}^3}\int_0^\infty\mathbb{P}(p^0,\mathcal{I}) e^{-C\left(1+\frac{\mathcal{I}}{mc^2}\right)  p^0}\phi(\mathcal{I})\,d\mathcal{I}dp < \infty
\end{equation}
for any positive constant $C$ and arbitrary polynomial $\mathbb{P}$ of $p^0$ and $\mathcal{I}$. Then, there exists a positive constant $\delta$ such that if $E(f_0)<\delta $,
 \eqref{PR} admits a unique global in time solution such that the energy functional is uniformly bounded:
	$$
E_N(f)(t)+\int_0^t \mathcal{D}_N(f)(s) ds\le CE_N(f_0).
	$$
If futher $f_0\in L^2_{p,\mathcal{I}}\dot{H}^{-s}_x$ for some $s\in [0,3/2)$, then 
\begin{enumerate}
	\item The negative Sobolev norm is uniformly bounded:
	$$
	\|\Lambda^{-s}f(t)\|_{L^2_{x,p,\mathcal{I}}}\le C_0.
	$$
	\item The solution converges to the global equilibrium with algebraic decay rate:
$$
\sum_{\ell\le k\le N}\|\nabla^k f(t)\|_{L^2_{x,p,\mathcal{I}}}\le C(1+t)^{-\frac{\ell+s}{2}}\quad \text{for}\ -s<\ell\le N-1.
$$
\item The microscopic part decays faster by $1/2$: 
	$$
\left\|\nabla^k\{I-P\}f(t)\right\|_{L^2_{x,p,\mathcal{I}}}\le C(1+t)^{-\frac{\ell+1+s}{2}}\quad \text{for}\ -s<\ell\le N-2.
$$
\end{enumerate}

\end{theorem}
\begin{remark}\label{choices}
The choice of state density $\phi(\mathcal{I})$ to guarantee the correct classical limit is not unique. For example, the following choices of $ \phi(\mathcal{I})$ 
	$$
	\mathcal{I}^{(f^i-2)/2}\qquad \text{or}\qquad e^{-b\frac{\mathcal{I}}{c^2}}\left(1+\frac{\mathcal{I}}{mc^2}\right)^{r}\mathcal{I}^{(f^i-2)/2}\quad\text{with}\quad b\ge0,\ r>0
	$$ 
lead to the correct classical limit (See \cite{PR3}). 
\end{remark} 

Unlike the classical BGK models \cite{BGK,Holway}, the equilibrium temperature of \eqref{PR} is determined through the nonlinear relation \eqref{gamma relation}
due to the relativistic nature of the equilibrium distribution function $F_E$. For rigorous analysis, therefore, it must be first analyzed whether or not the relation \eqref{gamma relation} provides the unique equilibrium temperature as a moment of the solution $F$. That is, any existence problem for \eqref{PR} must be understood as the problem of solving the coupled system of \eqref{PR} and \eqref{gamma relation}. In the case of relativistic models for a monatomic gas, such solvability problem was addressed in \cite{BCNS} for the Marle model, and \cite{HY2} for the Anderson-Witting model. In \cite{BCNS,HY2}, there is clever manipulations of the modified Bessel functions of the second kind that was crucially used to show the monotonicity property of $\widetilde{e}$, and it plays an important role in proving the one-to-one correspondence between $T$ and $e/n$. However, in the case of a polyatomic gas, similar line of argument using the modified Bessel functions does not work due to the presence of the state density of the internal mode $\phi(\mathcal{I})$ which can takes various forms (See Remark \eqref{choices}). 
In view of these difficulties,  we derive the following identity 
	\begin{equation*}
\left\{\widetilde{e}\right\}^{\prime}(T) =\frac{1}{k_B nT^2}\int_{\mathbb{R}^3}\int_0^\infty \left\{cp^0\Big(1+\frac{\mathcal{I}}{mc^2}\Big)-\frac{e}{n}\right\}^2F_E(g_r,0,T)\phi(\mathcal{I})\,d\mathcal{I}dp
	\end{equation*}
to investigate the monotonicity property of $\widetilde{e}$ in a different way (See Proposition \ref{solvability gamma}). Since the number density $n$ and the energy $e$ are strictly positive for sufficiently small $E(f)(t)$, the above relation implies that  $\widetilde{e}(T)$ is strictly increasing on $T\in (0,\infty)$, which enables us to solve the solvability problem of $T$. We mention that since the relativistic BGK model \eqref{PR} does not guarantee the positivity of solutions, the smallness condition of $E(f)(t)$ was required to preserve the sign of $n$ and $e$.

\noindent\newline

\subsection{Brief history} The mathematical research on the relativistic BGK model was initiated in 2012 by Bellouquid et al \cite{BCNS} for the Marle model, where the unique determination of equilibrium variables, asymptotic limits and linearization problem were addressed. Afterward, Bellouquid et al \cite{BNU} proved the existence and asymptotic behavior of solutions for the Marle model when the initial data starts close to the global equilibrium. Recently, Hwang and Yun \cite{HY1} established the existence and uniqueness of stationary solutions to the boundary value problem for the Marle model in a finite interval. The weak solutions were covered by Calvo et al \cite{CJS}. In the case of the Anderson-Witting model, the unique determination of equilibrium variables, and the existence and asymptotic behavior of near-equilibrium solutions were addressed in \cite{HY2}. The unique existence of stationary solutions to the Anderson-Witting model in a slab was studied in \cite{HY3}. 
		
For the relativistic Boltzmann equation, much more have been established. We refer to \cite{B,D,DE} for the local existence and linearized solution, \cite{GS1,GS2,Guo Strain Momentum,Strain1,Strain Zhu} for the global existence and asymptotic behavior of near-equilibrium solutions, and \cite{Dud3,Jiang1,Jiang2} for the existence with large data. The spatially homogeneous case was addressed in \cite{LR,Strain Yun}. The regularizing effect of the collision operator has been studied in \cite{A,JY,W}. The propagation of the uniform upper bound was established in \cite{JSY}. We refer to \cite{Cal,Strain2} for the Newtonian limit and \cite{SS} for the hydrodynamic limit. For the results on the relativistic theories of continuum for rarefied gases and its connections with the kinetic theory see for example  \cite{LMR,CPR1,CPR2,PR5,PR6,RXZ}.

	
	\bigskip

	This paper is organized as follows. In Section 2, the unique determination of the equilibrium variable $T$ is discussed. In Section 3, we study the linearization of the relativistic BGK model \eqref{PR}. In Section 4, we provide estimates for the macroscopic fields and nonlinear perturbation. Section 5 is devoted to the proof of Theorem \ref{main3}.

\noindent\newline
 \section{Unique determination of the equilibrium temperature $T$}
We recall from \eqref{gamma relation} that $T$ is determined through the following relation
\begin{align*}
\widetilde{e}(T)=\frac{\int_{\mathbb{R}^3}\int_0^\infty \sqrt{1+|p|^2}e^{-\left(mc^2+ \mathcal{I} \right)\frac{1}{k_B T}\sqrt{1+|p|^2}} \left(mc^2+ \mathcal{I} \right)\phi(\mathcal{I})\, d\mathcal{I}dp}{ \int_{\mathbb{R}^3}\int_0^\infty e^{-\left(mc^2+ \mathcal{I} \right)\frac{1}{k_B T}\sqrt{1+|p|^2}}\phi(\mathcal{I}) \,d\mathcal{I}dp}=\frac{e}{n}.
\end{align*}
In this section, formal calculations are first presented to show that the relativistic BGK model \eqref{PR} satisfies the conservation laws \eqref{conservation law} if the above relation admits a unique $T$. And then we prove that when $E(f)(t)$ is small enough, $T$ indeed can be uniquely determined. The following lemma will be used later to simplify the integral of $F_E$.
\begin{lemma}\cite{Strain2}\label{rest frame}
	For $U^\mu=(\sqrt{c^2+|U|^2},U)$, define $\Lambda$ by
	\begin{align*}
		\Lambda=
		\begin{bmatrix}
			c^{-1}U^0 & -c^{-1}U^1 & -c^{-1}U^2 & -c^{-1}U^3 \cr
			-U^1&  1+(U^0-1)\frac{(U^1)^2}{|U|^2}&(U^0-1)\frac{U^1U^2}{|U|^2}  &(U^0-1)\frac{U^1U^3}{|U|^2}  \cr
			-U^2& (U^0-1)\frac{U^1U^2}{|U|^2} &  1+(U^0-1)\frac{(U^2)^2 }{|U|^2}&(U^0-1)\frac{U^2U^3}{|U|^2}  \cr
			-U^3&  (U^0-1)\frac{U^1U^3}{|U|^2}& (U^0-1)\frac{U^2U^3}{|U|^2} &  1+(U^0-1)\frac{(U^3)^2}{|U|^2}
		\end{bmatrix}.
	\end{align*}
	Then $\Lambda$ transforms $U^\mu$ into the local rest frame $(c,0,0,0).$
\end{lemma}
\begin{proof}
	The proof that $\Lambda$ is the Lorentz transformation can be found in \cite{Strain2}. The identity $\Lambda U^\mu=(c,0,0,0)$ can be verified by an explicit computation:
	\begin{align*}
		\Lambda U^\mu&=
		\begin{bmatrix}
			c^{-1}(U^0)^2-c^{-1}(U^1)^2-c^{-1}(U^2)^2-c^{-1}(U^3)^2\cr
			-U^0U^1+U^1+\frac{(U^0-1)U^1}{|U|^2}|U|^2\cr
			-U^0U^2+U^2+\frac{(U^0-1)U^2}{|U|^2}|U|^2\cr
			-U^0U^3+U^3+\frac{(U^0-1)U^3}{|U|^2}|U|^2
		\end{bmatrix}=\begin{bmatrix}
			c^{-1}\left(c^2+|U|^2-|U|^2\right)\cr
			-U^0U^1+U^1+(U^0-1)U^1 \cr
			-U^0U^2+U^2+(U^0-1)U^2 \cr
			-U^0U^3+U^3+(U^0-1)U^3
		\end{bmatrix}=\begin{bmatrix}
			c\cr
			0 \cr
			0 \cr
			0
		\end{bmatrix}.
	\end{align*}
	\end{proof}

\begin{lemma}\label{explicit GM}

Assume \eqref{gamma relation} admits a unique $T$, then $F_E$ satisfies \eqref{conservation laws}.
\end{lemma}
\begin{remark}
	The solvability problem of $T$ in \eqref{gamma relation} is addressed in Proposition \ref{solvability gamma}.
\end{remark}
\begin{proof}
	 We write \eqref{conservation laws} in terms of the macroscopic fields using the Eckart frame \eqref{Eckart}: 
	\begin{align}\label{det}\begin{split}
			\frac{1}{c}U_\mu\int_{\mathbb{R}^3}\int_0^\infty p^\mu\left(1-p^\alpha q_\alpha \frac{1+\frac{\mathcal{I}}{mc^2}}{bmc^2}\right)F_E   \phi(\mathcal{I}) \,d\mathcal{I}\frac{dp}{p^0}=\frac{1}{mc^2}U_\mu V^\mu&=n,\cr
			cU_\mu\int_{\mathbb{R}^3}\int_0^\infty p^\mu p^\nu \left(1-p^\alpha q_\alpha \frac{1+\frac{\mathcal{I}}{mc^2}}{bmc^2}\right)F_E \Big(1+\frac{\mathcal{I}}{mc^2}\Big) \phi(\mathcal{I})\,d\mathcal{I}\frac{dp}{p^0}=U_\mu T^{\mu\nu}&=q^\nu+eU^\nu.
	\end{split}\end{align}
By the change of variables $P^\mu:=\Lambda p^\mu$ using Lemma \ref{rest frame}, we have
	\begin{eqnarray}\label{F_E1}\begin{split}
			&\int_{\mathbb{R}^3}\int_0^\infty p^\mu F_E   \phi(\mathcal{I})\,d\mathcal{I}\frac{dp}{p^0}\cr
			&=e^{-1+\frac{m}{k_B}\frac{ g_r}{T} }\int_{\mathbb{R}^3}\int_0^\infty \left(\Lambda^{-1}P^\mu\right) 	e^{-\left(1+\frac{\mathcal{I}}{mc^2}\right)\frac{1}{k_B T} \left(\Lambda U^\mu\right)\left( \Lambda p_\mu\right)}  \phi(\mathcal{I}) \,d\mathcal{I}\frac{dP}{P^0}\cr
			&=e^{-1+\frac{m}{k_B}\frac{ g_r}{T} }\Lambda^{-1}\left(\int_{\mathbb{R}^3}\int_0^\infty e^{-\left(1+\frac{\mathcal{I}}{mc^2}\right)\frac{cP^0}{k_B T}}\phi(\mathcal{I}) \,d\mathcal{I}dP,0,0,0\right)\cr
			&=\frac{1}{c}e^{-1+\frac{m}{k_B}\frac{ g_r}{T} }\left\{\int_{\mathbb{R}^3}\int_0^\infty e^{-\left(1+\frac{\mathcal{I}}{mc^2}\right)\frac{cP^0}{k_B T}}\phi(\mathcal{I}) \,d\mathcal{I}dP\right\}\Lambda^{-1}\left(c,0,0,0\right) \cr
			&=\frac{1}{c}e^{-1+\frac{m}{k_B}\frac{ g_r}{T} }\left\{\int_{\mathbb{R}^3}\int_0^\infty e^{-\left(1+\frac{\mathcal{I}}{mc^2}\right)\frac{cp^0}{k_B T}}\phi(\mathcal{I}) \,d\mathcal{I}dp \right\}U^\mu
	\end{split}\end{eqnarray}
	and
	\begin{align}\label{F_E2}\begin{split}
			&U_{\mu}\int_{\mathbb{R}^3}\int_0^\infty p^\mu p^\nu F_E\Big(1+\frac{\mathcal{I}}{mc^2}\Big)   \phi(\mathcal{I})\,d\mathcal{I}\frac{dp}{p^0}\cr
			&= e^{-1+\frac{m}{k_B}\frac{ g_r}{T} }\int_{\mathbb{R}^3}\int_0^\infty \big\{(\Lambda U_{\mu}) (\Lambda p^\mu)\big\} (\Lambda^{-1}P^\nu) e^{-\left(1+\frac{\mathcal{I}}{mc^2}\right)\frac{1}{k_B T} \left(\Lambda U^\mu\right)\left( \Lambda p_\mu\right)} \Big(1+\frac{\mathcal{I}}{mc^2}\Big) \phi(\mathcal{I}) \,d\mathcal{I}\frac{dP}{P^0}\cr
			&= ce^{-1+\frac{m}{k_B}\frac{ g_r}{T} }\Lambda^{-1}\left( \int_{\mathbb{R}^3}\int_0^\infty P^0e^{-\left(1+\frac{\mathcal{I}}{mc^2}\right)\frac{cP^0}{k_B T}}\Big(1+\frac{\mathcal{I}}{mc^2}\Big)   \phi(\mathcal{I}) \,d\mathcal{I}dP,0,0,0\right)\cr
			&= e^{-1+\frac{m}{k_B}\frac{ g_r}{T} }\left\{\int_{\mathbb{R}^3}\int_0^\infty p^0 e^{-\left(1+\frac{\mathcal{I}}{mc^2}\right)\frac{cp^0}{k_B T}}\Big(1+\frac{\mathcal{I}}{mc^2}\Big) \phi(\mathcal{I})\, d\mathcal{I}dp\right\}U^\nu 
	\end{split}\end{align}
	where we used the fact that (1) Lorentz inner product and the volume element $dp/p^0$ are invariant under $\Lambda$, and (2) $P^\mu$ takes the form of $(\sqrt{(mc)^2+|P|^2},P).$ 	Then, it follows from \eqref{F_E2} and the decomposition of third moment \cite{PR}:
	\begin{align*}
		\int_{\mathbb{R}^3}\int_0^\infty p^\mu   p^\nu p^\alpha F_E  \Big(1+\frac{\mathcal{I}}{mc^2}\Big)^2  \phi(\mathcal{I}) \,d\mathcal{I}\frac{dp}{p^0}=\frac{m}{c}\left\{aU^\alpha U^\mu U^\nu+b\left(h^{\alpha\mu}U^\nu+h^{\alpha\nu}U^\mu+h^{\mu\nu}U^\alpha\right)\right\}
	\end{align*}
	that the second terms on the l.h.s of \eqref{det} are calculated respectively as follows
	\begin{align}\label{second1}\begin{split}
			&-\frac{1}{c}U_\mu\int_{\mathbb{R}^3}\int_0^\infty p^\mu p^\alpha q_\alpha \frac{1+\frac{\mathcal{I}}{mc^2}}{bmc^2}F_E   \phi(\mathcal{I}) \,d\mathcal{I}\frac{dp}{p^0} \cr
			&=-\frac{1}{bmc^3} q_\alpha \int_{\mathbb{R}^3}\int_0^\infty p^\alpha U_\mu p^\mu   \Big(1+\frac{\mathcal{I}}{mc^2}\Big) F_E   \phi(\mathcal{I}) \,d\mathcal{I}\frac{dp}{p^0} \cr
			&=-\frac{1}{bmc^3} e^{-1+\frac{m}{k_B}\frac{ g_r}{T} }\left\{\int_{\mathbb{R}^3}\int_0^\infty p^0 e^{-\left(1+\frac{\mathcal{I}}{mc^2}\right)\frac{cp^0}{k_B T}}\Big(1+\frac{\mathcal{I}}{mc^2}\Big) \phi(\mathcal{I})\, d\mathcal{I}dp\right\} q_\alpha U^\alpha\cr
			&=0
	\end{split}\end{align}
	and
	\begin{align}\label{second2}\begin{split}
			&-cU_\mu\int_{\mathbb{R}^3}\int_0^\infty p^\mu p^\nu p^\alpha q_\alpha \frac{1+\frac{\mathcal{I}}{mc^2}}{bmc^2}F_E \Big(1+\frac{\mathcal{I}}{mc^2}\Big) \phi(\mathcal{I})\,d\mathcal{I}\frac{dp}{p^0}\cr
			&=-\frac{1}{bc^2} U_\mu q_\alpha\left\{aU^\alpha U^\mu U^\nu+b\left(h^{\alpha\mu}U^\nu+h^{\alpha\nu}U^\mu+h^{\mu\nu}U^\alpha\right)\right\}\cr 
			&=-\frac{1}{c^2} U_\mu q_\alpha \Big( h^{\alpha\nu}U^\mu\Big)\cr 
			&=q^\nu.
	\end{split}\end{align}
	Here we used the fact that (1) $U_\mu U^\mu=c^2$, and (2) $h^{\alpha\mu}, h^{\mu\nu}$ and $q^\mu$ are orthogonal to $U_\mu$ in the following sense
	$$
	U_\mu h^{\mu\nu}= U_\mu\Big(-g^{\mu\nu}+\frac{1}{c^2}U^\mu U^\nu\Big)=-U^\nu+U^\mu=0,\qquad U_\mu q^\mu=-U_\mu h^\mu_\alpha U_\beta T^{\alpha\beta}=0.
	$$
Finally, we go back to \eqref{det} with \eqref{F_E1}--\eqref{second2} to obtain
	\begin{align*} 
			&e^{-1+\frac{m}{k_B}\frac{ g_r}{T}}\int_{\mathbb{R}^3}\int_0^\infty e^{-\left(1+\frac{\mathcal{I}}{mc^2}\right)\frac{cp^0}{k_B T}}\phi(\mathcal{I}) \,d\mathcal{I}dp=n,\cr
			& ce^{-1+\frac{m}{k_B}\frac{ g_r}{T}}\left\{\int_{\mathbb{R}^3}\int_0^\infty p^0e^{-\left(1+\frac{\mathcal{I}}{mc^2}\right)\frac{cp^0}{k_B T}} \Big(1+\frac{\mathcal{I}}{mc^2}\Big)\phi(\mathcal{I})\, d\mathcal{I}dp\right\}U^\nu=eU^\nu.
	\end{align*}
Using the change of variables $\frac{p}{mc}\rightarrow p$, we get
\begin{align*}
\frac{e}{n}&=\frac{\int_{\mathbb{R}^3}\int_0^\infty \sqrt{1+|p|^2}e^{-\left(mc^2+ \mathcal{I} \right)\frac{1}{k_B T}\sqrt{1+|p|^2}} \left(mc^2+ \mathcal{I} \right)\phi(\mathcal{I})\, d\mathcal{I}dp}{ \int_{\mathbb{R}^3}\int_0^\infty e^{-\left(mc^2+ \mathcal{I} \right)\frac{1}{k_B T}\sqrt{1+|p|^2}}\phi(\mathcal{I}) \,d\mathcal{I}dp},\cr
e^{-1+\frac{m}{k_B}\frac{ g_r}{T}}&=\frac{n}{(mc)^3\int_{\mathbb{R}^3}\int_0^\infty e^{-\left(mc^2+ \mathcal{I} \right)\frac{1}{k_B T}\sqrt{1+|p|^2}}\phi(\mathcal{I}) \,d\mathcal{I}dp}
\end{align*}
which gives the desired result.
\end{proof}
The following lemma provides information about the ranges of $n$ and $e/n$ when $E(f)(t)$ is small enough.
\begin{lemma}\label{n positive}
	Suppose $E(f)(t)$ is sufficiently small. Then we have
	$$ |n-1|+\left|\frac{e}{n}-\widetilde{e}(T_0)\right|\le C\sqrt{E(f)(t)}.$$
\end{lemma}
\begin{proof} 
By Lemma \ref{lem2} and the Soboelv embedding $H^2(\mathbb{R}_x^{3})\subseteq L^{\infty}(\mathbb{R}_x^{3})$, we have
	$$
	| n -1 |+ \left| \frac{e}{n} -\widetilde{e}(T_0) \right|\le C \|  f\|_{L^2_{p,\mathcal{I}}}\le C \left\|  \|f\|_{L^\infty_x}\right\|_{L^2_{p,\mathcal{I}}}\le C \left\|  \|f\|_{H^2_x}\right\|_{L^2_{p,\mathcal{I}}}\le C\sqrt{E(f)(t)}.
	$$
	Note that the above result is independent of the determination of $T$ since Lemma \ref{lem2} is established by the definition of $n$ and $e/n$ given in \eqref{macroscopic fields}.
\end{proof}
We are now ready to prove that \eqref{gamma relation} determines a unique $T$ in the near-equilibrium regime. 
\begin{proposition}\label{solvability gamma}
	Suppose $E(f)(t)$ is sufficiently small. Then $T$ can be uniquely determined by the relation \eqref{gamma relation}. Thus $T$ is written as 
	$$
	T=(\widetilde{e})^{-1}\left(\frac{e}{n}\right).
	$$
\end{proposition}
\begin{proof}
	We observe that  
	\begin{align*}
		\left\{\widetilde{e}\right\}^{\prime}(T)&=	\frac{1}{k_BT^2}\frac{\int_{\mathbb{R}^3}\int_0^\infty (1+|p|^2)e^{-\left(mc^2+ \mathcal{I} \right)\frac{1}{k_B T}\sqrt{1+|p|^2}} \left(mc^2+ \mathcal{I} \right)^2\phi(\mathcal{I})\, d\mathcal{I}dp}{ \int_{\mathbb{R}^3}\int_0^\infty e^{-\left(mc^2+ \mathcal{I} \right)\frac{1}{k_B T}\sqrt{1+|p|^2}}\phi(\mathcal{I}) \,d\mathcal{I}dp}\cr
		&- \frac{1}{k_BT^2}	\frac{\big(\int_{\mathbb{R}^3}\int_0^\infty \sqrt{1+|p|^2}e^{-\left(mc^2+ \mathcal{I} \right)\frac{1}{k_B T}\sqrt{1+|p|^2}} \left(mc^2+ \mathcal{I} \right)\phi(\mathcal{I})\, d\mathcal{I}dp\big)^2}{\big( \int_{\mathbb{R}^3}\int_0^\infty e^{-\left(mc^2+ \mathcal{I} \right)\frac{1}{k_B T}\sqrt{1+|p|^2}}\phi(\mathcal{I}) \,d\mathcal{I}dp\big)^2}\cr
		&=\frac{1}{k_B T^2}\left\{\frac{\int_{\mathbb{R}^3}\int_0^\infty (1+|p|^2)e^{-\left(mc^2+ \mathcal{I} \right)\frac{1}{k_B T}\sqrt{1+|p|^2}} \left(mc^2+ \mathcal{I} \right)^2\phi(\mathcal{I})\, d\mathcal{I}dp}{ \int_{\mathbb{R}^3}\int_0^\infty e^{-\left(mc^2+ \mathcal{I} \right)\frac{1}{k_B T}\sqrt{1+|p|^2}}\phi(\mathcal{I}) \,d\mathcal{I}dp}  -\Big(\frac{e}{n}\Big)^2\right\}.
	\end{align*}
	Using the change of variables $p\rightarrow \frac{p}{mc}$, one finds
	\begin{align}\label{solv}\begin{split}
			 \left\{\widetilde{e}\right\}^{\prime}(T)	&=\frac{1}{k_B T^2}\left\{c^2\frac{\int_{\mathbb{R}^3}\int_0^\infty (p^0)^2e^{-\left(1+ \frac{\mathcal{I}}{mc^2} \right)\frac{cp^0}{k_B T}} \left(1+ \frac{\mathcal{I}}{mc^2} \right)^2\phi(\mathcal{I})\, d\mathcal{I}dp}{ \int_{\mathbb{R}^3}\int_0^\infty e^{-\left(1+ \frac{\mathcal{I}}{mc^2} \right)\frac{cp^0}{k_B T}}\phi(\mathcal{I}) \,d\mathcal{I}dp}  -\Big(\frac{e}{n}\Big)^2\right\}\cr
			 &=\frac{1}{k_B T^2}\left\{\frac{1}{n}\int_{\mathbb{R}^3}\int_0^\infty \biggl\{cp^0\Big(1+ \frac{\mathcal{I}}{mc^2} \Big)\biggl\}^2F_E(g_r,0,T) \phi(\mathcal{I})\, d\mathcal{I}dp  -\Big(\frac{e}{n}\Big)^2\right\}
	\end{split}	\end{align}
	where $F_E(g_r,0,T)$ denotes
	\begin{align*}
F_E(g_r,0,T)&=	e^{-1+\frac{m}{k_B}\frac{ g_{r}}{T}-\left(1+\frac{\mathcal{I}}{mc^2}\right)\frac{cp^0}{k_B T}  }\cr
&=\frac{n}{\int_{\mathbb{R}^3}\int_0^\infty  e^{-\left(1+ \frac{\mathcal{I}}{mc^2} \right)\frac{cp^0}{k_B T}}\phi(\mathcal{I})\,d \mathcal{I}dp}e^{-\left(1+\frac{\mathcal{I}}{mc^2}\right)\frac{cp^0}{k_B T}  }.
	\end{align*}
	We also observe that 
	\begin{align}\label{slov 2}\begin{split}
		2\left(\frac{e}{n }\right)^2-\left(\frac{e}{n }\right)^2&= \frac{2e}{n }\frac{\int_{\mathbb{R}^3}\int_0^\infty \sqrt{1+|p|^2}e^{-\left(mc^2+ \mathcal{I} \right)\frac{1}{k_B T}\sqrt{1+|p|^2}} \left(mc^2+ \mathcal{I} \right)\phi(\mathcal{I})\, d\mathcal{I}dp}{ \int_{\mathbb{R}^3}\int_0^\infty e^{-\left(mc^2+ \mathcal{I} \right)\frac{1}{k_B T}\sqrt{1+|p|^2}}\phi(\mathcal{I}) \,d\mathcal{I}dp}\cr
		&-\left(\frac{e}{n}\right)^2 \frac{1}{n}\int_{\mathbb{R}^3}\int_0^\infty F_E(g_r,0,T)\phi(\mathcal{I})\,d\mathcal{I}dp\cr
		&=  \frac{1}{n}\int_{\mathbb{R}^3}\int_0^\infty\biggl\{ \frac{2cep^0}{n }\Big(1+ \frac{\mathcal{I}}{mc^2} \Big)-\left(\frac{e}{n}\right)^2 \biggl\}F_E(g_r,0,T) \phi(\mathcal{I})\, d\mathcal{I}dp 
\end{split}	\end{align}
where we used the change of variables $ p\rightarrow \frac{p}{mc}$ to get
\begin{align*}
 &\frac{\int_{\mathbb{R}^3}\int_0^\infty \sqrt{1+|p|^2}e^{-\left(mc^2+ \mathcal{I} \right)\frac{1}{k_B T}\sqrt{1+|p|^2}} \left(mc^2+ \mathcal{I} \right)\phi(\mathcal{I})\, d\mathcal{I}dp}{ \int_{\mathbb{R}^3}\int_0^\infty e^{-\left(mc^2+ \mathcal{I} \right)\frac{1}{k_B T}\sqrt{1+|p|^2}}\phi(\mathcal{I}) \,d\mathcal{I}dp}\cr
 &=  \frac{c\int_{\mathbb{R}^3}\int_0^\infty p^0e^{-\left(1+ \frac{\mathcal{I}}{mc^2} \right)\frac{cp^0}{k_B T}} \left(1+ \frac{\mathcal{I}}{mc^2} \right)\phi(\mathcal{I})\, d\mathcal{I}dp}{ \int_{\mathbb{R}^3}\int_0^\infty e^{-\left(1+ \frac{\mathcal{I}}{mc^2} \right)\frac{cp^0}{k_B T}}\phi(\mathcal{I}) \,d\mathcal{I}dp}\cr
 &=\frac{c}{n}\int_{\mathbb{R}^3}\int_0^\infty p^0F_E(g_r,0,T) \Big(1+ \frac{\mathcal{I}}{mc^2} \Big)\phi(\mathcal{I})\, d\mathcal{I}dp.
\end{align*}
Combining \eqref{solv} and \eqref{slov 2}, we have
\begin{align}\label{e5}
		 \left\{\widetilde{e}\right\}^{\prime}(T)&=\frac{1}{k_B nT^2}\int_{\mathbb{R}^3}\int_0^\infty \left\{cp^0\Big(1+\frac{\mathcal{I}}{mc^2}\Big)-\frac{e}{n}\right\}^2F_E(g_r,0,T)\phi(\mathcal{I})\,d\mathcal{I}dp.
	\end{align} 
Since Lemma \ref{n positive} says that $n$ is positive for sufficiently small $E(f)(t)$, \eqref{e5} implies that $\widetilde{e}(T)$ is a strictly increasing function. Furthermore, $\widetilde{e}(T)$ is continuous on $T\in (0,\infty)$ under the assumption \eqref{phi condition}.  So, there exists a positive constant $\delta_0$ such that $$
\left[\widetilde{e}(T_0)-\delta_0,~\widetilde{e}(T_0)+\delta_0\right]\subseteq \text{Range}\left(\widetilde{e}(T)\right).
$$ 
If $E(f)(t)\le \delta_0$, then we have from Lemma \ref{n positive} that 
$$
\widetilde{e}(T_0)-\delta_0\le \frac{e}{n}\le \widetilde{e}(T_0)+\delta_0
$$
which implies that the range of $e/n$ is included in the range of $\widetilde{e}(T)$ for sufficiently small $E(f)(t)$.
Therefore, there exists a one-to-one correspondence between $T$ and $e/n$ providing 
$$
T=(\widetilde{e})^{-1}\left(\frac{e}{n}\right).
$$
\end{proof}

 \section{Linearization}
In this section, the linearization of \eqref{PR} is discussed when the solution is sufficiently close to the global equilibrium. 
First, we provide computations of $F_E^0$ that will be used often later.
\begin{lemma}\label{computation F_E^0} The following identities hold:
	\begin{enumerate}
		\item  	$\displaystyle	\int_{\mathbb{R}^3}\int_0^\infty (p^i)^2 F_E^0 \left(1+\frac{\mathcal{I}}{mc^2}\right)\phi(\mathcal{I}) \,d\mathcal{I}\,\frac{dp}{p^0}= \frac{ k_BT_0}{c }. $
		\item  	$\displaystyle	\int_{\mathbb{R}^3}\int_0^\infty (p^0)^2 F_E^0 \left(1+\frac{\mathcal{I}}{mc^2}\right)\phi(\mathcal{I}) \,d\mathcal{I}\,\frac{dp}{p^0}= \frac{\widetilde{e}(T_0)}{c} .$
		\item $\displaystyle  \int_{\mathbb{R}^3}\int_0^\infty (p^i)^2   F_E^0\left(1+\frac{\mathcal{I}}{mc^2}\right)^2\phi(\mathcal{I})\,d\mathcal{I}dp=b_0m.$
		\item  $\displaystyle\frac{mc^2 b}{k_B nT}= \frac{e}{n}+ k_BT.$
\item $\displaystyle \frac{mc^2b_0}{k_BT_0 }= \widetilde{e}(T_0)+ k_BT_0 . $
	\end{enumerate}
Here  $b_0$ denotes  
	\begin{equation*}
   b_0=\frac{ mc^2}{\gamma_0^2}\left(\int_0^\infty \frac{K_2(\gamma_0^*)}{\gamma_0^*}\phi(\mathcal{I})\, d\mathcal{I}\right)^{-1}\int_0^\infty K_3(\gamma_0^*)\phi(\mathcal{I})\,d\mathcal{I}
	\end{equation*}
with
	$$
\gamma_0=\frac{mc^2}{k_BT_0},\qquad \gamma^*_0=\gamma_0\biggl( 1+\frac{\mathcal{I}}{mc^2}\biggl).
$$
\end{lemma}
\begin{proof}
	For reader's convenience, we record the definition of $F_E^0$:
$$
F_E^0=\exp\left\{-1+\frac{m}{k_B}\frac{ g_{r0}}{T_0}-\left(1+\frac{\mathcal{I}}{mc^2}\right)\frac{cp^0}{k_B T_0}  \right\}
$$
where
$$		
e^{-1+\frac{m}{k_B}\frac{g_{r0}}{T_0}}=\frac{1}{(mc)^3\int_{\mathbb{R}^3}\int_0^\infty  e^{-\left(mc^2+ \mathcal{I} \right)\frac{1}{k_B T_0}\sqrt{1+|p|^2}}\phi(\mathcal{I})\,d \mathcal{I}dp}.
$$
	%
	$\bullet$ Proof of (1): It follows from the change of variables $\frac{p}{mc}\rightarrow p$ that
	\begin{eqnarray*}
		&&\int_{\mathbb{R}^3}\int_0^\infty (p^i)^2F_E^0 \biggl(1+\frac{\mathcal{I}}{mc^2}\biggl)\phi(\mathcal{I}) \,d\mathcal{I}\,\frac{dp}{p^0}\cr
		&&=e^{-1+\frac{m}{k_B}\frac{ g_{r0}}{T_0}}	\int_{\mathbb{R}^3}\int_0^\infty (p^i)^2e^{-\left(1+\frac{\mathcal{I}}{mc^2}\right)\frac{cp^0}{k_B T_0}} \biggl(1+\frac{\mathcal{I}}{mc^2}\biggl)\phi(\mathcal{I}) \,d\mathcal{I}\,\frac{dp}{p^0}\cr
		&&=m^3c^2e^{-1+\frac{m}{k_B}\frac{ g_{r0}}{T_0}}	\int_{\mathbb{R}^3}\int_0^\infty  (p^i)^2 e^{-\left(mc^2+ \mathcal{I} \right)\frac{1}{k_B T_0}\sqrt{1+|p|^2}} \left(mc^2+ \mathcal{I} \right)\phi(\mathcal{I}) \,d\mathcal{I}\,\frac{dp}{\sqrt{1+|p|^2}}\cr
		&&=\frac{1}{3c}   \frac{	\int_{\mathbb{R}^3}\int_0^\infty  |p|^2 e^{-\left(mc^2+ \mathcal{I} \right)\frac{1}{k_B T_0}\sqrt{1+|p|^2}} \left(mc^2+ \mathcal{I} \right)\phi(\mathcal{I}) \,d\mathcal{I}\,\frac{dp}{\sqrt{1+|p|^2}}}{ \int_{\mathbb{R}^3}\int_0^\infty  e^{-\left(mc^2+ \mathcal{I} \right)\frac{1}{k_B T_0}\sqrt{1+|p|^2}}\phi(\mathcal{I})\,d \mathcal{I}dp}.
	\end{eqnarray*}
Using spherical coordinates and integration by parts, we have
	\begin{eqnarray*}
		&&\int_{\mathbb{R}^3}\int_0^\infty (p^i)^2F_E^0 \left(1+\frac{\mathcal{I}}{mc^2}\right)\phi(\mathcal{I}) \,d\mathcal{I}\,\frac{dp}{p^0}\cr
		&&=\frac{1}{3c}   \frac{	\int_{\mathbb{R}^3}\int_0^\infty  |p|^2 e^{-\left(mc^2+ \mathcal{I} \right)\frac{1}{k_B T_0}\sqrt{1+|p|^2}} \left(mc^2+ \mathcal{I} \right)\phi(\mathcal{I}) \,d\mathcal{I}\,\frac{dp}{\sqrt{1+|p|^2}}}{ \int_{\mathbb{R}^3}\int_0^\infty  e^{-\left(mc^2+ \mathcal{I} \right)\frac{1}{k_B T_0}\sqrt{1+|p|^2}}\phi(\mathcal{I})\,d \mathcal{I}dp}\cr
		&&=\frac{1}{3c}   \frac{	\int_0^\infty\int_0^\infty  \frac{r^4}{\sqrt{1+r^2}} e^{-\left(mc^2+ \mathcal{I} \right)\frac{1}{k_B T_0}\sqrt{1+r^2}} \left(mc^2+ \mathcal{I} \right)\phi(\mathcal{I}) \,d\mathcal{I}\,dr}{ \int_0^\infty\int_0^\infty  r^2e^{-\left(mc^2+ \mathcal{I} \right)\frac{1}{k_B T_0}\sqrt{1+r^2}}\phi(\mathcal{I})\,d \mathcal{I}dr}\cr
		&&=\frac{k_BT_0}{c}   
		\end{eqnarray*}
\noindent
$\bullet$ Proof of (2):	It can be obtained in a similar way to (1) as follows
	\begin{eqnarray*}
		&&\int_{\mathbb{R}^3}\int_0^\infty (p^0)^2F_E^0 \left(1+\frac{\mathcal{I}}{mc^2}\right)\phi(\mathcal{I}) \,d\mathcal{I}\,\frac{dp}{p^0}\cr
		&&=e^{-1+\frac{m}{k_B}\frac{ g_{r0}}{T_0}}	\int_{\mathbb{R}^3}\int_0^\infty p^0e^{-\left(1+\frac{\mathcal{I}}{mc^2}\right)\frac{cp^0}{k_B T_0}} \left(1+\frac{\mathcal{I}}{mc^2}\right)\phi(\mathcal{I}) \,d\mathcal{I}\,dp\cr
		&&=\frac{1}{c}\frac{\int_{\mathbb{R}^3}\int_0^\infty \sqrt{1+|p|^2}e^{-\left(mc^2+\mathcal{I}\right)\frac{1}{k_B T_0}\sqrt{1+|p|^2}} \left(mc^2+ \mathcal{I} \right)\phi(\mathcal{I}) \,d\mathcal{I}\,dp}{\int_{\mathbb{R}^3}\int_0^\infty e^{-\left(mc^2+\mathcal{I}\right)\frac{1}{k_B T_0}\sqrt{1+|p|^2}} \left(mc^2+ \mathcal{I} \right)\phi(\mathcal{I}) \,d\mathcal{I}\,dp}	\cr
		&&=\frac{\widetilde{e}(T_0)}{c}.
	\end{eqnarray*}
\noindent
$\bullet$ Proof of (3): We first introduce another representation of the modified Bessel functions of the second kind:
\begin{align*}
K_2(\gamma)=\int_0^\infty	\frac{2r^2+1}{\sqrt{1+r^2}}e^{-\gamma\sqrt{1+r^2}}\,dr,\qquad K_3(\gamma)=\int_0^\infty	(4r^2+1)e^{-\gamma\sqrt{1+r^2}}\,dr.
\end{align*}
Using this, one can rewrite $e^{-1+\frac{m}{k_B}\frac{ g_{r0}}{T_0}}$ as
	\begin{align}\label{K2}\begin{split}
	e^{-1+\frac{m}{k_B}\frac{ g_{r0}}{T_0}} &=\frac{1}{4\pi(mc)^3}\left(\int_0^\infty\int_0^\infty  r^2 e^{-\left(1+ \frac{\mathcal{I}}{mc^2} \right)\gamma_0\sqrt{1+r^2}}\phi(\mathcal{I})\,d \mathcal{I}dr\right)^{-1}\cr
	&=\frac{1}{4\pi(mc)^3}\left(\int_0^\infty\int_0^\infty  \frac{2r^2+1}{\sqrt{1+r^2}}\frac{1}{\gamma_0^*} e^{-\gamma^*_0\sqrt{1+r^2}}\phi(\mathcal{I})\,d \mathcal{I}dr\right)^{-1}\cr
	&=\frac{1}{4\pi(mc)^3 }\left(\int_0^\infty \frac{K_2(\gamma_0^*)}{\gamma_0^*}\phi(\mathcal{I})\, d\mathcal{I}\right)^{-1}.
\end{split}	\end{align}
	Using spherical coordinates and change of variables $\frac{p}{mc}\rightarrow p$, we have from \eqref{K2} that
	\begin{align}\label{(3)}
	\begin{split} 
		&\int_{\mathbb{R}^3}\int_0^\infty (p^j)^2   F_E^0\left(1+\frac{\mathcal{I}}{mc^2}\right)^2\phi(\mathcal{I})\,d\mathcal{I}dp\cr
		&=e^{-1+\frac{m}{k_B}\frac{ g_{r0}}{T_0}}\frac{1}{3 }\int_{\mathbb{R}^3}\int_0^\infty |p|^2   e^{-\left(1+\frac{\mathcal{I}}{mc^2}\right)\frac{cp^0}{k_B T_0}}\left(1+\frac{\mathcal{I}}{mc^2}\right)^2\phi(\mathcal{I})\,d\mathcal{I}dp\cr
		&=\frac{ (mc)^2}{3 }\left(\int_0^\infty \frac{K_2(\gamma_0^*)}{\gamma_0^*}\phi(\mathcal{I})\, d\mathcal{I}\right)^{-1}\int_0^\infty\int_0^\infty r^4   e^{-\left(1+ \frac{\mathcal{I}}{mc^2} \right)\gamma_0\sqrt{1+r^2}}\left(1+\frac{\mathcal{I}}{mc^2}\right)^2\phi(\mathcal{I})\,d\mathcal{I}dr.
	\end{split}
\end{align}
By integration by parts twice, \eqref{(3)} becomes 
	\begin{align*}
		&\int_{\mathbb{R}^3}\int_0^\infty (p^j)^2   F_E^0\left(1+\frac{\mathcal{I}}{mc^2}\right)^2\phi(\mathcal{I})\,d\mathcal{I}dp\cr
		&=\frac{ (mc)^2}{3 }\left(\int_0^\infty \frac{K_2(\gamma_0^*)}{\gamma_0^*}\phi(\mathcal{I})\, d\mathcal{I}\right)^{-1} \int_0^\infty\int_0^\infty r^4   e^{-\left(1+ \frac{\mathcal{I}}{mc^2} \right)\gamma_0\sqrt{1+r^2}}\left(1+\frac{\mathcal{I}}{mc^2}\right)^2\phi(\mathcal{I})\,d\mathcal{I}dr\cr
		&=\frac{ (mc)^2}{\gamma_0^2 } \left(\int_0^\infty \frac{K_2(\gamma_0^*)}{\gamma_0^*}\phi(\mathcal{I})\, d\mathcal{I}\right)^{-1}\int_0^\infty\int_0^\infty (4r^2+1)   e^{-\left(1+ \frac{\mathcal{I}}{mc^2} \right)\gamma_0\sqrt{1+r^2}} \phi(\mathcal{I})\,d\mathcal{I}dr\cr
		&=\frac{ (mc)^2}{\gamma_0^2 }\left(\int_0^\infty \frac{K_2(\gamma_0^*)}{\gamma_0^*}\phi(\mathcal{I})\, d\mathcal{I}\right)^{-1}\int_0^\infty K_3(\gamma_0^*)\phi(\mathcal{I})\,d\mathcal{I}
\end{align*} 
which gives the desired result.\noindent\newline
	\noindent
	$\bullet$ Proof of (4): Recall from \eqref{gamma relation} that
	\begin{align*}
	\frac{e}{n}&=\frac{\int_{\mathbb{R}^3}\int_0^\infty \sqrt{1+|p|^2}e^{-\left(mc^2+ \mathcal{I} \right)\frac{1}{k_B T}\sqrt{1+|p|^2}} \left(mc^2+ \mathcal{I} \right)\phi(\mathcal{I})\, d\mathcal{I}dp}{ \int_{\mathbb{R}^3}\int_0^\infty e^{-\left(mc^2+ \mathcal{I} \right)\frac{1}{k_B T}\sqrt{1+|p|^2}}\phi(\mathcal{I}) \,d\mathcal{I}dp}\cr
&=mc^2\frac{\int_0^\infty\int_0^\infty r^2\sqrt{1+r^2}e^{-\left(1+ \frac{\mathcal{I}}{mc^2} \right)\gamma\sqrt{1+r^2}} \left(1+ \frac{\mathcal{I}}{mc^2} \right)\phi(\mathcal{I})\, d\mathcal{I}dr}{ \int_0^\infty\int_0^\infty r^2e^{-\left(1+ \frac{\mathcal{I}}{mc^2} \right)\gamma\sqrt{1+r^2}}\phi(\mathcal{I}) \,d\mathcal{I}dr}.
	\end{align*}
By \eqref{K2} and integration by parts, one finds
		\begin{align*}
	\frac{e}{n}&= mc^2 \left(\int_0^\infty \frac{K_2(\gamma^*)}{\gamma^*}\phi(\mathcal{I})\, d\mathcal{I}\right)^{-1}\int_0^\infty\int_0^\infty r^2\sqrt{1+r^2}e^{-\left(1+ \frac{\mathcal{I}}{mc^2} \right)\gamma\sqrt{1+r^2}} \left(1+ \frac{\mathcal{I}}{mc^2} \right)\phi(\mathcal{I})\, d\mathcal{I}dr\cr
	&= \frac{mc^2}{\gamma} \left(\int_0^\infty \frac{K_2(\gamma^*)}{\gamma^*}\phi(\mathcal{I})\, d\mathcal{I}\right)^{-1}\int_0^\infty\int_0^\infty (3r^2+1)e^{-\left(1+ \frac{\mathcal{I}}{mc^2} \right)\gamma\sqrt{1+r^2}} \phi(\mathcal{I})\, d\mathcal{I}dr\cr
	&= \frac{mc^2}{\gamma} \left(\int_0^\infty \frac{K_2(\gamma^*)}{\gamma^*}\phi(\mathcal{I})\, d\mathcal{I}\right)^{-1}\left(\int_0^\infty K_3(\gamma^*)\phi(\mathcal{I})\,d\mathcal{I}-\int_0^\infty\int_0^\infty r^2e^{-\left(1+ \frac{\mathcal{I}}{mc^2} \right)\gamma\sqrt{1+r^2}}\phi(\mathcal{I}) \,d\mathcal{I}dr \right)\cr
	&=\frac{\gamma b}{n}- \frac{ mc^2}{\gamma  }
\end{align*}
which gives the desired result. Since (5) can be obtained in the same manner as in (4), we omit it.
 \end{proof}
\noindent

\subsection{Linearization of \eqref{PR}}  
   Define $e_i$ $(i=1,\cdots,5)$ by
 \begin{align*}
	e_1&=\sqrt{F_E^0},\qquad e_{2,3,4}=\sqrt{\frac{1}{b_0m}}\left(1+\frac{\mathcal{I}}{mc^2}\right)p\sqrt{F_E^0},\cr
	e_5&=\sqrt{\frac{1}{k_BT_0^2\left\{\widetilde{e}\right\}^{\prime}(T_{0})}}\left\{cp^0\left(1+\frac{\mathcal{I}}{mc^2}\right)-\widetilde{e}(T_0)\right\}\sqrt{F_E^0}
\end{align*}
and the projection operator $P(f)$ by
\begin{align}\label{Pf}
	P(f)= \sum_{i=1}^5\langle f,e_{i} \rangle_{L^2_{p,I}} e_{i}.
\end{align}
Then, the equilibrium distribution function $F_E$ given in \eqref{GJdf} is linearized as follows. 
 \begin{lemma}\label{lin2}
 Suppose $E(f)(t)$ is sufficiently small. We then have
 	$$
    	\left(1-p^\mu q_\mu \frac{1+\frac{\mathcal{I}}{mc^2}}{bmc^2}\right)F_E-F_E^0  =\left(P(f)+\sum_{i=1}^4\Gamma_i (f)\right)\sqrt{F_E^0}.
 	$$
Here the nonlinear perturbations $\Gamma_i(f)$ $(i=1,\cdots, 4)$ are given by
  	\begin{eqnarray*}
&& \Gamma_1(f)=\left(\frac{\Psi_1}{2}-\frac{\Psi^2}{2(2+\Psi+2\sqrt{1+\Psi})}\right)\sqrt{F_E^0},\cr
  					&&\Gamma_2 (f)=  \frac{1}{\left\{\widetilde{e}\right\}^{\prime}(T_{0})}\frac{1}{k_BT_0^2}\left\{cp^0 \left(1+\frac{\mathcal{I}}{mc^2}\right) -\widetilde{e}(T_0)\right\}\sqrt{F_E^0}\cr
 		&&\hspace{9mm}\times\biggl\{\frac{1}{c}\left( 1-\int_{\mathbb{R}^3}\int_0^\infty f\sqrt{F_E^0}  \phi(\mathcal{I}) \,d\mathcal{I}\,dp \right) \int_{\mathbb{R}^3}\int_0^\infty \left\{	2cp^0\Phi+\Phi^2\right\}F\left(1+\frac{\mathcal{I}}{mc^2}\right)\phi(\mathcal{I})\,d\mathcal{I}\frac{dp}{p^0}\cr
 		&&\hspace{12mm}-\frac{1}{c}\left(	\frac{\Psi_1}{2}	+\frac{\Psi^3-3\Psi^2}{2(2+\Psi-\Psi^2+2\sqrt{1+\Psi})}\right) \int_{\mathbb{R}^3}\int_0^\infty \left(U^\mu p_\mu\right)^2F\left(1+\frac{\mathcal{I}}{mc^2}\right)\phi(\mathcal{I})\,d\mathcal{I}\frac{dp}{p^0}\cr
 		&&\hspace{12mm}-c\int_{\mathbb{R}^3}\int_0^\infty  f\sqrt{F_E^0}\phi(\mathcal{I}) \,d\mathcal{I}dp\int_{\mathbb{R}^3}\int_0^\infty p^0f\sqrt{F_E^0}\left(1+\frac{\mathcal{I}}{mc^2}\right)\phi(\mathcal{I})\,d\mathcal{I}\,dp\biggl\}\cr
 	&&\Gamma_{3}(f)=-c\frac{1+\frac{\mathcal{I}}{mc^2}}{k_BT_0}\left\{\frac{\Psi}{2}+\frac{\Psi^3-3\Psi^2}{2(2+\Psi-\Psi^2+2\sqrt{1+\Psi})}\right\}\int_{\mathbb{R}^3}\int_0^\infty pf\sqrt{F_E^0} \phi(\mathcal{I}) \,d\mathcal{I}\,\frac{dp}{p^0}\cdot p\sqrt{F_E^0}\cr
 	&&\hspace{9mm}+\frac{1+\frac{\mathcal{I}}{mc^2}}{b_0mc^2}  \Gamma_3^*(f)\cdot p\sqrt{F_E^0}-p^0 q^0\frac{1+\frac{\mathcal{I}}{mc^2}}{b_0mc^2} \sqrt{F_E^0}\cr
 	&&\Gamma_{4} (f)=\frac{1}{\sqrt{F_E^0}}\int_0^1 (1-\theta)\left(n -1,U,\frac{e}{n} -\widetilde{e}(T_0),q^\mu\right)D^2\widetilde{F}(\theta)\left(n -1,U ,\frac{e}{n} -\widetilde{e}(T_0),q^\mu\right)^T\,d\theta,
 	\end{eqnarray*}
 where  $\Gamma_{3}^*(f)$ denotes
\begin{eqnarray*}\label{gamma3}\begin{split}
	&\Gamma_{3}^*(f)	= -c^2\sum_{i=1}^3\int_{\mathbb{R}^3}\int_0^\infty p^if\sqrt{F_E^0}\phi(\mathcal{I})\,d\mathcal{I}\,\frac{dp}{p^0}\int_{\mathbb{R}^3}\int_0^\infty p p^if\sqrt{F_E^0}\left(1+\frac{\mathcal{I}}{mc^2}\right)\phi(\mathcal{I})\,d\mathcal{I}\frac{dp}{p^0}\cr
&\hspace{9mm}+c\int_{\mathbb{R}^3}\int_0^\infty p\Phi_1 F\left(1+\frac{\mathcal{I}}{mc^2}\right)\phi(\mathcal{I})\,d\mathcal{I} \frac{dp}{p^0}-\int_{\mathbb{R}^3}\int_0^\infty pf\sqrt{F_E^0} \phi(\mathcal{I}) \,d\mathcal{I}\,\frac{dp}{p^0}\cr
&\hspace{9mm} \times\int_{\mathbb{R}^3}\int_0^\infty  \left\{c^2p^0f\sqrt{F_E^0}+\frac{1}{p^0}\left(2cp^0\Phi+\Phi^2 \right)F\right\} \left(1+\frac{\mathcal{I}}{mc^2}\right)\phi(\mathcal{I}) \,d\mathcal{I}\,dp\cr
&\hspace{9mm}+	\left\{\frac{\Psi}{2}+\frac{\Psi^3-3\Psi^2}{2(2+\Psi-\Psi^2+2\sqrt{1+\Psi})}\right\}\int_{\mathbb{R}^3}\int_0^\infty pf\sqrt{F_E^0} \phi(\mathcal{I}) \,d\mathcal{I}\,\frac{dp}{p^0}\cr
&\hspace{9mm} \times\int_{\mathbb{R}^3}\int_0^\infty  \left(U^\nu p_\nu\right)^2 F \left(1+\frac{\mathcal{I}}{mc^2}\right)\phi(\mathcal{I}) \,d\mathcal{I}\,\frac{dp}{p^0},
\end{split}\end{eqnarray*}
and $\Psi,\Psi_1,\Phi$ and $\Phi_1$ are defined as
  	\begin{eqnarray}\label{notation2}
 	\begin{split}
 	&\Psi=2\int_{\mathbb{R}^3}\int_0^\infty f\sqrt{F_E^0}  \phi(\mathcal{I}) \,d\mathcal{I}\,dp+	\left(\int_{\mathbb{R}^3}\int_0^\infty f\sqrt{F_E^0}  \phi(\mathcal{I}) \,d\mathcal{I}\,dp\right)^2\cr
 	&\hspace{3mm}-\sum_{i=1}^3 \left(\int_{\mathbb{R}^3}\int_0^\infty p^if\sqrt{F_E^0}  \phi(\mathcal{I}) \,d\mathcal{I}\,\frac{dp}{p^0}\right)^2,\cr
 	&\Psi_1= \left(\int_{\mathbb{R}^3}\int_0^\infty f\sqrt{F_E^0}  \phi(\mathcal{I}) \,d\mathcal{I}\,dp\right)^2-\sum_{i=1}^3 \left(\int_{\mathbb{R}^3}\int_0^\infty p^if\sqrt{F_E^0}  \phi(\mathcal{I}) \,d\mathcal{I}\,\frac{dp}{p^0}\right)^2,\cr
 	&\Phi= U_{\mu} p^\mu-cp^0,\cr
 	&\Phi_1= U_{\mu} p^\mu-cp^0+c\sum_{i=1}^3 p^i\int_{\mathbb{R}^3}\int_0^\infty p^i f\sqrt{F_E^0}\phi(\mathcal{I})\,d\mathcal{I}\frac{dp}{p^0}.
 	 	\end{split}\end{eqnarray}
 \end{lemma}
 \begin{proof}
We consider the transitional macroscopic fields between $F$ and $F_E^0$:
 	\begin{equation*} 
 	\left(n_\theta, U_\theta,\left(\frac{e}{n}\right)_\theta,q^\mu_\theta\right)=\theta	\Big(n, U,\frac{e}{n},q^\mu\Big)+(1-\theta)	(1,0,\widetilde{e}(T_0),0),
 	\end{equation*}
 	and define $\widetilde{F}(\theta)$ and $F(\theta)$ by
 	\begin{align*}
 	\widetilde{F}(\theta)&=\left(1-p^\mu q_{\theta\mu} \frac{1+\frac{\mathcal{I}}{mc^2}}{b_\theta mc^2}\right) e^{-1+\frac{m}{k_B}\frac{ g_{r\theta}}{T_\theta}-\left(1+\frac{\mathcal{I}}{mc^2}\right)\frac{1}{k_B T_\theta}U_\theta^\mu p_\mu }\cr
 	 F(\theta)&=e^{-1+\frac{m}{k_B}\frac{ g_{r\theta}}{T_\theta}-\left(1+\frac{\mathcal{I}}{mc^2}\right)\frac{1}{k_B T_\theta}U_\theta^\mu p_\mu }.
 	\end{align*}
Here $g_{r\theta}, T_\theta$ and $b_\theta$ are given by  
\begin{align*} 
	e^{-1+\frac{m}{k_B}\frac{g_{r\theta}}{T_\theta}}&=\frac{n_\theta}{(mc)^3\int_{\mathbb{R}^3}\int_0^\infty  e^{-\left(mc^2+ \mathcal{I} \right)\frac{1}{k_B T_\theta}\sqrt{1+|p|^2}}\phi(\mathcal{I})\,d \mathcal{I}dp},\qquad T_\theta=\{\widetilde{e}\}^{-1}\left(\frac{e}{n}\right)_\theta \cr
&b_\theta=\frac{n_\theta mc^2}{\gamma_\theta^2}\left(\int_0^\infty \frac{K_2(\gamma_\theta^*)}{\gamma_\theta^*}\phi(\mathcal{I})\, d\mathcal{I}\right)^{-1}\int_0^\infty K_3(\gamma_\theta^*)\phi(\mathcal{I})\,d\mathcal{I}
\end{align*}  
  	with
  	$$
  	\gamma_\theta =\frac{mc^2}{k_BT_\theta},\qquad \gamma_\theta^*=\gamma_\theta\left(1+\frac{\mathcal{I}}{mc^2}\right).
  	$$
By Lemma \ref{computation F_E^0}, $b_\theta$ can be also expressed in the following manner 
$$
b_\theta= \frac{k_B n_\theta  T_\theta \widetilde{e}(T_\theta) }{mc^2} + \frac{n_\theta \left(k_B T_\theta\right)^2 }{mc^2}.
$$
Thus $\widetilde{F}(\theta)$  can be regarded as a function of $n_\theta, U_\theta,(e/n)_\theta $ and $q^\mu_\theta$, so it follows from the Talyor expansion that 
 	\begin{align*} 
 	& \left(1-p^\mu q_\mu \frac{1+\frac{\mathcal{I}}{mc^2}}{b mc^2}\right)F_E-F_E^0 \cr
 	&= \widetilde{F}(1)-\widetilde{F}(0) \cr
 	&= \frac{\partial \widetilde{F}}{\partial n_\theta}\biggl|_{\theta=0}\frac{\partial n_\theta}{\partial \theta}+\nabla_{U_\theta}\widetilde{F}\Big|_{\theta=0}\cdot \frac{\partial U_\theta}{\partial \theta}+\frac{\partial \widetilde{F}}{\partial (e/n)_\theta} \biggl|_{\theta=0}\frac{\partial (e/n)_\theta}{\partial \theta}+\frac{\partial \widetilde{F}}{\partial q^0_\theta} \biggl|_{\theta=0}\frac{\partial q^0_\theta}{\partial \theta}+\nabla_{q_\theta}\widetilde{F}\biggl|_{\theta=0}\cdot \frac{\partial q_\theta}{\partial \theta} \cr
 	&+  \int_0^1 (1-\theta)\left(n -1,U,T -T_0,q^\mu\right)D^2\widetilde{F}(\theta)\left(n -1,U ,T -T_0,q^\mu\right)^T\,d\theta\cr
 	&=(n -1) F_E^0 +\frac{1}{\left\{\widetilde{e}\right\}^{\prime}(T_{0})}\frac{1}{k_BT_0^2}\left(\frac{e}{n} -\widetilde{e}(T_0)\right)\biggl\{ cp^0 \biggl(1+\frac{\mathcal{I}}{mc^2}\biggl) -\widetilde{e}(T_0)\biggl\} F_E^0 + \frac{1+\frac{\mathcal{I}}{mc^2} }{k_BT_0} U \cdot p  F_E^0 \cr
 	&-p ^\mu q_\mu \frac{1+\frac{\mathcal{I}}{mc^2}}{b_0 mc^2}  F_E^0 +  \int_0^1 (1-\theta)\left(n -1,U,\frac{e}{n} -\widetilde{e}(T_0),q^\mu\right)D^2\widetilde{F}(\theta)\left(n -1,U ,\frac{e}{n} -\widetilde{e}(T_0),q^\mu\right)^T\,d\theta.
  	\end{align*}
In the last identity, we used the simple calculations
 	\begin{align*} 
 	&\frac{\partial \widetilde{F}}{\partial n_\theta}\biggl|_{\theta=0}=F_E^0,\qquad \nabla_{U_\theta}\widetilde{F}\Big|_{\theta=0}= \frac{1+\frac{\mathcal{I}}{mc^2}}{k_BT_0} pF_E^0,\cr
 	&\frac{\partial}{\partial (e/n)_\theta}\widetilde{F}\biggl|_{\theta=0}=\frac{1}{\left\{\widetilde{e}\right\}^{\prime}(T_{0})}\frac{1}{k_BT_0^2}\left\{cp^0 \left(1+\frac{\mathcal{I}}{mc^2 }\right) -\widetilde{e}(T_0) \right\}F_E^0,\cr
 	&\frac{\partial}{\partial q_\theta^0}\widetilde{F}\biggl|_{\theta=0}=-p^0\frac{1+\frac{\mathcal{I}}{mc^2}}{b_0mc^2} F_E^0,\qquad \nabla_{q_\theta}\widetilde{F}\Big|_{\theta=0}=p\frac{1+\frac{\mathcal{I}}{mc^2}}{b_0mc^2} F_E^0.
  \end{align*}
 	Now we use the notations $I_i$ $(i=1,\cdots,4)$ to denote the first four terms on the last identity, and we decompose them into the linear part and the nonlinear part.
 	\newline
 	$\bullet$ Decomposition of $I_1$: Inserting $F=F_E^0+f\sqrt{F_E^0}$ into a definition of $n$, one finds
\begin{align}\label{n=}\begin{split}
n&=\left\{	\left(\int_{\mathbb{R}^3}\int_0^\infty F \phi(\mathcal{I}) \,d\mathcal{I}\,dp\right)^2-\sum_{i=1}^3 \left(\int_{\mathbb{R}^3}\int_0^\infty p^i F \phi(\mathcal{I}) \,d\mathcal{I}\,\frac{dp}{p^0}\right)^2\right\}^{\frac{1}{2}}\cr
&=\biggl\{1+2\int_{\mathbb{R}^3}\int_0^\infty f\sqrt{F_E^0}  \phi(\mathcal{I}) \,d\mathcal{I}\,dp+	\left(\int_{\mathbb{R}^3}\int_0^\infty f\sqrt{F_E^0}  \phi(\mathcal{I}) \,d\mathcal{I}\,dp\right)^2\cr
&-\sum_{i=1}^3 \left(\int_{\mathbb{R}^3}\int_0^\infty p^if\sqrt{F_E^0}  \phi(\mathcal{I}) \,d\mathcal{I}\,\frac{dp}{p^0}\right)^2\biggl\}^{\frac{1}{2}}\cr
&=\sqrt{1+\Psi}
\end{split}\end{align}
which, together with the following identity \cite{BCNS}:
\begin{equation}\label{route pi}
\sqrt{1+\Psi}=1+ \frac{\Psi}{2}-\frac{\Psi^2}{2(2+\Psi+2\sqrt{1+\Psi})}
\end{equation}
gives
$$
n-1= \frac{\Psi}{2}-\frac{\Psi^2}{2(2+\Psi+2\sqrt{1+\Psi})}.
$$
Therefore,
\begin{align*}
I_1&=(n-1)F_E^0\cr
&=\left( \frac{\Psi}{2}-\frac{\Psi^2}{2(2+\Psi+2\sqrt{1+\Psi})}
\right)F_E^0\cr
&= \left(\int_{\mathbb{R}^3}\int_0^\infty f\sqrt{F_E^0}  \phi(\mathcal{I}) \,d\mathcal{I}\,dp+\frac{\Psi_1}{2}-\frac{\Psi^2}{2(2+\Psi+2\sqrt{1+\Psi})}\right)F_E^0\cr
&= \left\{\int_{\mathbb{R}^3}\int_0^\infty f\sqrt{F_E^0}  \phi(\mathcal{I}) \,d\mathcal{I}\,dp\right\}F_E^0+\Gamma_1(f)\sqrt{F_E^0}.
\end{align*}
\newline
\noindent
 	$\bullet$ Decomposition of $I_2$: Considering \eqref{n=} and the following identity \cite{BCNS}:
 	\begin{align*}
 	\frac{1}{\sqrt{1+\Psi}}=1-\frac{\Psi}{2}-\frac{\Psi^3-3\Psi^2}{2(2+\Psi-\Psi^2+2\sqrt{1+\Psi})},
 	\end{align*}
one can see that
 	\begin{align}\label{1/n=}
 	\begin{split}	\frac{1}{n}&=1-\frac{\Psi}{2}-\frac{\Psi^3-3\Psi^2}{2(2+\Psi-\Psi^2+2\sqrt{1+\Psi})}\cr
 	&=	1-\int_{\mathbb{R}^3}\int_0^\infty f\sqrt{F_E^0}  \phi(\mathcal{I}) \,d\mathcal{I}\,dp-\frac{\Psi_1}{2}	-\frac{\Psi^3-3\Psi^2}{2(2+\Psi-\Psi^2+2\sqrt{1+\Psi})}.
 	\end{split}\end{align}
We have from $\eqref{notation2}_3$ that
 	\begin{equation}\label{uq2}
 	(U^\mu p_\mu)^2=(cp^0)^2+2cp^0\Phi+\Phi^2.
 	\end{equation}
By \eqref{1/n=} and \eqref{uq2}, $\displaystyle e/n -\widetilde{e}(T_0)$ is decomposed as  	\begin{align}\label{e-e3}
 	\begin{split}
 \frac{e}{n} -\widetilde{e}(T_0)	&=\frac{1}{nc} \int_{\mathbb{R}^3}\int_0^\infty  \left(U^\mu p_\mu\right)^2F\left(1+\frac{\mathcal{I}}{mc^2}\right) \phi(\mathcal{I}) \,d\mathcal{I}\,\frac{dp}{p^0}-\widetilde{e}(T_0)\cr
 	&=	\frac{1}{c}\left\{	1-\int_{\mathbb{R}^3}\int_0^\infty f\sqrt{F_E^0}  \phi(\mathcal{I}) \,d\mathcal{I}\,dp-\frac{\Psi_1}{2}	-\frac{\Psi^3-3\Psi^2}{2(2+\Psi-\Psi^2+2\sqrt{1+\Psi})}\right\}\cr
 	& \times\int_{\mathbb{R}^3}\int_0^\infty \left\{	(cp^0)^2+2cp^0\Phi+\Phi^2\right\}F\left(1+\frac{\mathcal{I}}{mc^2}\right)\phi(\mathcal{I})\,d\mathcal{I}\frac{dp}{p^0}-\widetilde{e}(T_0)\cr
 	&\equiv  \left( 1-\int_{\mathbb{R}^3}\int_0^\infty f\sqrt{F_E^0}  \phi(\mathcal{I}) \,d\mathcal{I}\,dp \right)\int_{\mathbb{R}^3}\int_0^\infty c(p^0)^2 F\left(1+\frac{\mathcal{I}}{mc^2}\right)\phi(\mathcal{I})\,d\mathcal{I}\frac{dp}{p^0}-\widetilde{e}(T_0)\cr & +R_{I_2},
 	\end{split}\end{align}
 where $R_{I_2}$ denotes
 \begin{align*}
 	R_{I_2}&= \frac{1}{c}\left( 1-\int_{\mathbb{R}^3}\int_0^\infty f\sqrt{F_E^0}  \phi(\mathcal{I}) \,d\mathcal{I}\,dp \right) \int_{\mathbb{R}^3}\int_0^\infty \left\{	2cp^0\Phi+\Phi^2\right\}F\left(1+\frac{\mathcal{I}}{mc^2}\right)\phi(\mathcal{I})\,d\mathcal{I}\frac{dp}{p^0} \cr
 	&  -\frac{1}{c}\left(	\frac{\Psi_1}{2}	+\frac{\Psi^3-3\Psi^2}{2(2+\Psi-\Psi^2+2\sqrt{1+\Psi})}\right) \int_{\mathbb{R}^3}\int_0^\infty \left(U^\mu p_\mu\right)^2F\left(1+\frac{\mathcal{I}}{mc^2}\right)\phi(\mathcal{I})\,d\mathcal{I}\frac{dp}{p^0}.
 \end{align*}
Here $R_{I_2}$ is nonlinear with respect to $f$ that can be shown by the following identity
\begin{align*}
\Phi&=\sqrt{c^2+|U|^2}p^0-U\cdot p-cp^0 \cr
&=\left(\frac{|U|^2}{2c}-\frac{|U|^4}{2c(2c^2+|U|^2+2c\sqrt{1+|U|^2})}\right)p^0-U\cdot p.
\end{align*}
Using Lemma \ref{computation F_E^0} (2), the first two terms on the last identity of \eqref{e-e3} reduce to 
 	\begin{align}\label{e-e4}
 	\begin{split}
 	& 	\left(	1-\int_{\mathbb{R}^3}\int_0^\infty f\sqrt{F_E^0}  \phi(\mathcal{I}) \,d\mathcal{I}\,dp\right)\int_{\mathbb{R}^3}\int_0^\infty cp^0\left( F_E^0 +f\sqrt{F_E^0}\right)\left(1+\frac{\mathcal{I}}{mc^2}\right)\phi(\mathcal{I})\,d\mathcal{I}\,dp-\widetilde{e}(T_0)\cr
 	& = 	\left(	1-\int_{\mathbb{R}^3}\int_0^\infty f\sqrt{F_E^0}  \phi(\mathcal{I}) \,d\mathcal{I}\,dp\right)\left(\widetilde{e}(T_0)+\int_{\mathbb{R}^3}\int_0^\infty cp^0  f\sqrt{F_E^0}\left(1+\frac{\mathcal{I}}{mc^2}\right)\phi(\mathcal{I})\,d\mathcal{I}\,dp\right)-\widetilde{e}(T_0)\cr
 	& =\int_{\mathbb{R}^3}\int_0^\infty \left\{cp^0\left(1+\frac{\mathcal{I}}{mc^2}\right)-\widetilde{e}(T_0)\right\} f\sqrt{F_E^0} \phi(\mathcal{I})\,d\mathcal{I}dp\cr
 	&-c\int_{\mathbb{R}^3}\int_0^\infty  f\sqrt{F_E^0}\phi(\mathcal{I}) \,d\mathcal{I}dp\int_{\mathbb{R}^3}\int_0^\infty p^0f\sqrt{F_E^0}\left(1+\frac{\mathcal{I}}{mc^2}\right)\phi(\mathcal{I})\,d\mathcal{I}\,dp.
 	\end{split}\end{align}
 	Now we go back to \eqref{e-e3} with \eqref{e-e4} to get
 	\begin{align}\label{e-e0}		
 \begin{split}	
 \frac{e}{n}-\widetilde{e}(T_0)&= \int_{\mathbb{R}^3}\int_0^\infty \left\{cp^0\left(1+\frac{\mathcal{I}}{mc^2}\right)-\widetilde{e}(T_0)\right\} f\sqrt{F_E^0} \phi(\mathcal{I})\,d\mathcal{I}dp\cr
 	&-c\int_{\mathbb{R}^3}\int_0^\infty  f\sqrt{F_E^0}\phi(\mathcal{I}) \,d\mathcal{I}dp\int_{\mathbb{R}^3}\int_0^\infty p^0f\sqrt{F_E^0}\left(1+\frac{\mathcal{I}}{mc^2}\right)\phi(\mathcal{I})\,d\mathcal{I}\,dp+R_{I_2} ,
\end{split} 	\end{align}
which leads to 
\begin{align*}
 	I_2&=\frac{1}{\left\{\widetilde{e}\right\}^{\prime}(T_{0})}\frac{1}{k_BT_0^2}\left(\frac{e}{n} -\widetilde{e}(T_0)\right)\left(cp^0 \Big(1+\frac{\mathcal{I}}{mc^2}\Big) -\widetilde{e}(T_0)\right)F_E^0\cr
 	&=\frac{1}{\left\{\widetilde{e}\right\}^{\prime}(T_{0})}\frac{1}{k_BT_0^2}\int_{\mathbb{R}^3}\int_0^\infty \left(cp^0\Big(1+\frac{\mathcal{I}}{mc^2}\Big)-\widetilde{e}(T_0)\right) f\sqrt{F_E^0} \phi(\mathcal{I})\,d\mathcal{I}\,dp\cr
 	& \times\left(cp^0 \Big(1+\frac{\mathcal{I}}{mc^2}\Big) -\widetilde{e}(T_0)\right) F_E^0 +\Gamma_2 (f)\sqrt{F_E^0}.
 	\end{align*}
 	\newline
 	\noindent	
 	$\bullet$ Decomposition of $I_3$: By \eqref{1/n=}, $U$ is written as
 	\begin{align}\label{u_F}
 	\begin{split}
 	U&=\frac{c}{n}\int_{\mathbb{R}^3}\int_0^\infty pF\phi(\mathcal{I}) \,d\mathcal{I}\,\frac{dp}{p^0}\cr
 	&=c\left\{1-\frac{\Psi}{2}-\frac{\Psi^3-3\Psi^2}{2(2+\Psi-\Psi^2+2\sqrt{1+\Psi})}\right\}\int_{\mathbb{R}^3}\int_0^\infty pf\sqrt{F_E^0} \phi(\mathcal{I}) \,d\mathcal{I}\,\frac{dp}{p^0}.
 	\end{split}\end{align}
This directly leads to the linearization of $I_3$ as follows
 	\begin{align}\label{u}\begin{split}
 	I_3&=\frac{1+\frac{\mathcal{I}}{mc^2} }{k_BT_0} U \cdot p  F_E^0\cr
 	&=	c\frac{1+\frac{\mathcal{I}}{mc^2}}{k_BT_0}\int_{\mathbb{R}^3}\int_0^\infty pf\sqrt{F_E^0} \phi(\mathcal{I}) \,d\mathcal{I}\,\frac{dp}{p^0}\cdot p F_E^0 \cr
 	&-c\frac{1+\frac{\mathcal{I}}{mc^2}}{k_BT_0}\left\{\frac{\Psi}{2}+\frac{\Psi^3-3\Psi^2}{2(2+\Psi-\Psi^2+2\sqrt{1+\Psi})}\right\}\int_{\mathbb{R}^3}\int_0^\infty pf\sqrt{F_E^0} \phi(\mathcal{I}) \,d\mathcal{I}\,\frac{dp}{p^0}\cdot p F_E^0
 	\end{split}\end{align}
which will be combined with the result of $I_4$.
 	\noindent\newline
 	$\bullet$ Decomposition of $I_4$: Recall from \eqref{macroscopic fields} that the heat flux $q$ is defined by
 	\begin{align}\label{heat}\begin{split}
 	q &=c\int_{\mathbb{R}^3}\int_0^\infty p  \left(U^\nu p_\nu\right)F \left(1+\frac{\mathcal{I}}{mc^2}\right)\phi(\mathcal{I}) \,d\mathcal{I}\,\frac{dp}{p^0}\cr
 	&-\frac{1}{ c }U \int_{\mathbb{R}^3}\int_0^\infty   \left(U^\nu p_\nu\right)^2 F\left(1+\frac{\mathcal{I}}{mc^2}\right)\phi(\mathcal{I}) \,d\mathcal{I}\,\frac{dp}{p^0}.
\end{split} 	\end{align}
Inserting $\eqref{notation2}_4$ into the first term of \eqref{heat}, one finds 
 		 		\begin{eqnarray}\label{flux1}\begin{split}
 		 		& 	c\int_{\mathbb{R}^3}\int_0^\infty p  (U^\nu p_\nu)F\left(1+\frac{\mathcal{I}}{mc^2}\right)\phi(\mathcal{I})\,d\mathcal{I}\frac{dp}{p^0}	\cr
 		 		&= c^2\int_{\mathbb{R}^3}\int_0^\infty p   \biggl(F_E^0+f\sqrt{F_E^0}\biggl) \biggl(1+\frac{\mathcal{I}}{mc^2}\biggl)\phi(\mathcal{I})\,d\mathcal{I} dp \cr
 		 		&- c^2\sum_{i=1}^3\int_{\mathbb{R}^3}\int_0^\infty p^i f\sqrt{F_E^0}\phi(\mathcal{I})\,d\mathcal{I}\frac{dp}{p^0}\int_{\mathbb{R}^3}\int_0^\infty p  p^i \biggl(F_E^0+f\sqrt{F_E^0}\biggl) \biggl(1+\frac{\mathcal{I}}{mc^2}\biggl)\phi(\mathcal{I})\,d\mathcal{I}\frac{dp}{p^0}\cr
 		 		&+ c\int_{\mathbb{R}^3}\int_0^\infty p \Phi_1 F \biggl(1+\frac{\mathcal{I}}{mc^2}\biggl)\phi(\mathcal{I})\,d\mathcal{I}\frac{dp}{p^0}\cr
	 		 	&=c^2\int_{\mathbb{R}^3}\int_0^\infty p   f\sqrt{F_E^0}\left(1+\frac{\mathcal{I}}{mc^2}\right)\phi(\mathcal{I})\,d\mathcal{I}dp-  ck_BT_0  \int_{\mathbb{R}^3}\int_0^\infty p f\sqrt{F_E^0}\phi(\mathcal{I})\,d\mathcal{I}\frac{dp}{p^0}\cr
 		 		& -c^2\sum_{i=1}^3\int_{\mathbb{R}^3}\int_0^\infty p^if\sqrt{F_E^0}\phi(\mathcal{I})\,d\mathcal{I}\,\frac{dp}{p^0}\int_{\mathbb{R}^3}\int_0^\infty p p^if\sqrt{F_E^0}\left(1+\frac{\mathcal{I}}{mc^2}\right)\phi(\mathcal{I})\,d\mathcal{I}\frac{dp}{p^0}\cr
 		 		&+c\int_{\mathbb{R}^3}\int_0^\infty p \Phi_1 F\left(1+\frac{\mathcal{I}}{mc^2}\right)\phi(\mathcal{I})\,d\mathcal{I} \frac{dp}{p^0}.
\end{split} 		 	\end{eqnarray}
In the last identity, we used the spherical symmetry of $F_E^0$ and Lemma \ref{computation F_E^0} (1) so that 
 		 	\begin{eqnarray*}
 			&&\sum_{i=1}^3\int_{\mathbb{R}^3}\int_0^\infty p^i f\sqrt{F_E^0}\phi(\mathcal{I})\,d\mathcal{I}\frac{dp}{p^0}\int_{\mathbb{R}^3}\int_0^\infty pp^i F_E^0\left(1+\frac{\mathcal{I}}{mc^2}\right)\phi(\mathcal{I})\,d\mathcal{I}\frac{dp}{p^0}\cr
 			&&=\frac{1}{3}\int_{\mathbb{R}^3}\int_0^\infty p f\sqrt{F_E^0}\phi(\mathcal{I})\,d\mathcal{I}\frac{dp}{p^0}\int_{\mathbb{R}^3}\int_0^\infty |p|^2 F_E^0\left(1+\frac{\mathcal{I}}{mc^2}\right)\phi(\mathcal{I})\,d\mathcal{I}\frac{dp}{p^0} \cr
 			&&=\frac{k_BT_0}{c}\int_{\mathbb{R}^3}\int_0^\infty p f\sqrt{F_E^0}\phi(\mathcal{I})\,d\mathcal{I}\frac{dp}{p^0}.
 		\end{eqnarray*}
To deal with the second term of \eqref{heat}, we use Lemma \ref{computation F_E^0} (2)  and \eqref{uq2}  to obtain
 	\begin{eqnarray*}
 		&&\frac{1}{c}U\int_{\mathbb{R}^3}\int_0^\infty   \left(U^\nu p_\nu\right)^2 F \left(1+\frac{\mathcal{I}}{mc^2}\right)\phi(\mathcal{I}) \,d\mathcal{I}\,\frac{dp}{p^0}\cr
&&=\frac{1}{c}U\int_{\mathbb{R}^3}\int_0^\infty  \left\{(cp^0)^2+2cp^0\Phi+\Phi^2 \right\}\left(F_E^0+f\sqrt{F_E^0}\right) \left(1+\frac{\mathcal{I}}{mc^2}\right)\phi(\mathcal{I}) \,d\mathcal{I}\,\frac{dp}{p^0}\cr
	&&= \widetilde{e}(T_0)U + U\int_{\mathbb{R}^3}\int_0^\infty   \left\{cp^0f\sqrt{F_E^0}+\frac{1}{cp^0} \left(  2cp^0\Phi+\Phi^2 \right)F\right\}  \left(1+\frac{\mathcal{I}}{mc^2}\right)\phi(\mathcal{I}) \,d\mathcal{I}\,dp ,
 	\end{eqnarray*}
which, together with \eqref{u_F} gives
 	 	\begin{eqnarray}\label{flux3}
 	\begin{split}
 	& \frac{1}{c}U\int_{\mathbb{R}^3}\int_0^\infty   \left(U^\nu p_\nu\right)^2 F \left(1+\frac{\mathcal{I}}{mc^2}\right)\phi(\mathcal{I}) \,d\mathcal{I}\,\frac{dp}{p^0}\cr
  	&=c\widetilde{e}(T_0)\int_{\mathbb{R}^3}\int_0^\infty pf\sqrt{F_E^0} \phi(\mathcal{I}) \,d\mathcal{I}\,\frac{dp}{p^0}+\int_{\mathbb{R}^3}\int_0^\infty pf\sqrt{F_E^0} \phi(\mathcal{I}) \,d\mathcal{I}\,\frac{dp}{p^0}\cr
  	&\times \int_{\mathbb{R}^3}\int_0^\infty  \left\{c^2p^0f\sqrt{F_E^0}+\frac{1}{p^0}\left(2cp^0\Phi+\Phi^2 \right)F \right\}\left(1+\frac{\mathcal{I}}{mc^2}\right)\phi(\mathcal{I}) \,d\mathcal{I}\,dp \cr
 	&-	\left\{\frac{\Psi}{2}+\frac{\Psi^3-3\Psi^2}{2(2+\Psi-\Psi^2+2\sqrt{1+\Psi})}\right\}\int_{\mathbb{R}^3}\int_0^\infty pf\sqrt{F_E^0} \phi(\mathcal{I}) \,d\mathcal{I}\,\frac{dp}{p^0}\cr
 	& \times\int_{\mathbb{R}^3}\int_0^\infty   \left(U^\nu p_\nu\right)^2 F\left(1+\frac{\mathcal{I}}{mc^2}\right)\phi(\mathcal{I}) \,d\mathcal{I}\,\frac{dp}{p^0}.
 	\end{split}
 	\end{eqnarray}
 We go back to \eqref{heat} with \eqref{flux1} and \eqref{flux3} to get
 	\begin{eqnarray}\label{flux}\begin{split}
 	&p\cdot q\frac{1+\frac{\mathcal{I}}{mc^2}}{b_0mc^2} F_E^0\cr
 	&=  \frac{1+\frac{\mathcal{I}}{mc^2}}{b_0m }\int_{\mathbb{R}^3}\int_0^\infty p   f\sqrt{F_E^0}\left(1+\frac{\mathcal{I}}{mc^2}\right)\phi(\mathcal{I})\,d\mathcal{I}dp\cdot pF_E^0\cr
 	&- \frac{1+\frac{\mathcal{I}}{mc }}{b_0mc^2}\left(k_BT_0+\widetilde{e}(T_0)\right) \int_{\mathbb{R}^3}\int_0^\infty p f\sqrt{F_E^0}\phi(\mathcal{I})\,d\mathcal{I}\frac{dp}{p^0}\cdot pF_E^0+ \frac{1+\frac{\mathcal{I}}{mc^2}}{b_0mc^2}\Gamma_3^*(f)\cdot pF_E^0\cr
 	&=\frac{1+\frac{\mathcal{I}}{mc^2}}{b_0m }\int_{\mathbb{R}^3}\int_0^\infty p   f\sqrt{F_E^0}\left(1+\frac{\mathcal{I}}{mc^2}\right)\phi(\mathcal{I})\,d\mathcal{I}dp\cdot p F_E^0\cr
 	&-c\frac{1+\frac{\mathcal{I}}{mc^2}}{k_BT_0}\int_{\mathbb{R}^3}\int_0^\infty p f\sqrt{F_E^0}\phi(\mathcal{I})\,d\mathcal{I}\frac{dp}{p^0}\cdot pF_E^0+\frac{1+\frac{\mathcal{I}}{mc^2}}{b_0mc^2}\Gamma_3^*(f)\cdot pF_E^0
\end{split} 	\end{eqnarray}
where we used Lemma \ref{computation F_E^0} (5). We combine \eqref{u} and \eqref{flux} to conclude
 	\begin{eqnarray*}
 &&I_3+I_4	=\frac{1+\frac{\mathcal{I}}{mc^2} }{k_BT_0} U \cdot p  F_E^0+p \cdot q \frac{1+\frac{\mathcal{I}}{mc^2}}{b_0 mc^2} F_E^0-p ^0 q^0 \frac{1+\frac{\mathcal{I}}{mc^2}}{b_0 mc^2} F_E^0\cr
 &&\hspace{11.7mm}=c\frac{1+\frac{\mathcal{I}}{mc^2}}{k_BT_0}\int_{\mathbb{R}^3}\int_0^\infty pf\sqrt{F_E^0} \phi(\mathcal{I}) \,d\mathcal{I}\,\frac{dp}{p^0}\cdot p F_E^0 +p\cdot q\frac{1+\frac{\mathcal{I}}{mc^2}}{b_0mc^2} F_E^0 -p^0 q^0\frac{1+\frac{\mathcal{I}}{mc^2}}{b_0mc^2} F_E^0\cr
 &&\hspace{11.7mm}-c\frac{1+\frac{\mathcal{I}}{mc^2}}{k_BT_0}\left\{\frac{\Psi}{2}+\frac{\Psi^3-3\Psi^2}{2(2+\Psi-\Psi^2+2\sqrt{1+\Psi})}\right\}\int_{\mathbb{R}^3}\int_0^\infty pf\sqrt{F_E^0} \phi(\mathcal{I}) \,d\mathcal{I}\,\frac{dp}{p^0}\cdot p F_E^0 \cr
 	&&\hspace{11.7mm}=  \frac{1+\frac{\mathcal{I}}{mc^2}}{b_0m}  \int_{\mathbb{R}^3}\int_0^\infty p   f\sqrt{F_E^0}\left(1+\frac{\mathcal{I}}{mc^2}\right)\phi(\mathcal{I})\,d\mathcal{I}dp \cdot pF_E^0+\Gamma_3(f)\sqrt{F_E^0}.
 	\end{eqnarray*}
In the last identity, $p^0 q^0\frac{1+\frac{\mathcal{I}}{mc^2}}{b_0mc^2} F_E^0$ was absorbed into $\Gamma_3(f)\sqrt{F_E^0}$  since $q^0$ is nonlinear with respect to $f$. Finally,
letting
 	$$
 	\Gamma_4(f)=\frac{1}{\sqrt{F_E^0}}\int_0^1 (1-\theta)\left(n -1,U,\frac{e}{n} -\widetilde{e}(T_0),q^\mu\right)D^2\widetilde{F}(\theta)\left(n -1,U ,\frac{e}{n} -\widetilde{e}(T_0),q^\mu\right)^T\,d\theta
 	$$
 	completes the proof.
 	
 \end{proof}

 In the following proposition, we present the linearization of \eqref{PR}.
 \begin{proposition}\label{lin3}
 Suppose $E(f)(t)$ is sufficiently small. Then, \eqref{PR} can be linearized with respect to the perturbation $f$ as follows
 	\begin{align}\label{LAW}\begin{split}
 	\partial_t f+\hat{p}\cdot\nabla_x f&=\frac{1}{\tau}\left(L(f)+\Gamma (f)\right),\cr
 	f_0(x,p)&=f(0,x,p),
 	\end{split}\end{align}
 	where the linearized operator $L(f)$ and the nonlinear perturbation $\Gamma (f)$ are defined as
 	\begin{align*}
 	L(f)&= P(f)-f , \cr
 	\Gamma (f)&=\frac{U_\mu p^\mu}{c  p^0} \sum_{i=1}^4\Gamma_i (f)+\frac{P(f)-f}{c p^0}  \Phi 
 	\end{align*}
respectively.
 \end{proposition}
 
 \begin{proof}
 	Inserting $F=F_E^0+f\sqrt{F_E^0}$ into \eqref{PR}, one finds
 	\begin{align*}
 	\partial_t f+\hat{p}\cdot\nabla_x f&=\frac{U_\mu p^\mu}{c\tau p^0}\frac{1}{\sqrt{F_E^0}}\left\{\left(1-p^\mu q_\mu\frac{1+\frac{\mathcal{I}}{mc^2}}{bmc^2}\right)F_E-F_E^0-f\sqrt{F_E^0}\right\}.
 	\end{align*}
Then, it follows from Lemma \ref{lin2} that
  	 \begin{align*}
	\partial_t f+\hat{p}\cdot\nabla_x f=\frac{U_\mu p^\mu}{c\tau p^0}\Big(P(f)-f+\sum_{i=1}^4\Gamma_i (f)\Big)
  \end{align*}
which combined with $\eqref{notation2}_3$ gives the desired result.	
 \end{proof}

\subsection{Analysis of the linearized operator $L$}
 Let $N$ be a five dimensional space spanned by $$\left\{\sqrt{F_E^0},\Big(1+\frac{\mathcal{I}}{mc^2}\Big)p^\mu \sqrt{F_E^0}\right\}.$$
 \begin{lemma}\label{ortho}
 	$P$ is an orthonormal projection from $L_{p,I}^2(\mathbb{R}^3)$ onto $N$.
 \end{lemma}
 \begin{proof}
 	It is enough to show that $\{e_i\}$ $(i=1,\cdots,5)$ forms an orthonormal basis with respect to the inner product $\langle \cdot,\cdot \rangle_{L^2_{p,\mathcal{I}}}$.
 	\newline
 	$\bullet$ $\|e_1\|_{L^2_{p,I}}=1$: By a definition of $F_E^0$, it is straightforward  that
 	\begin{align*}
 	\langle e_1,e_1\rangle_{L^2_{p,\mathcal{I}}} &= \int_{\mathbb{R}^3}\int_0^\infty   F_E^0  \phi(\mathcal{I}) \,d\mathcal{I}\,dp\cr
 	&=\frac{\int_{\mathbb{R}^3}\int_0^\infty   e^{ -\left(1+\frac{\mathcal{I}}{mc^2}\right)\frac{cp^0}{k_B T_0}  }  \phi(\mathcal{I}) \,d\mathcal{I}\,dp}{(mc)^3\int_{\mathbb{R}^3}\int_0^\infty  e^{-\left(mc^2+ \mathcal{I} \right)\frac{1}{k_B T_0}\sqrt{1+|p|^2}}\phi(\mathcal{I})\,d \mathcal{I}dp} \cr
 	&=1.
 	\end{align*}
 	$\bullet$ $\|e_{i+1}\|_{L^2_{p,I}}=1$ $(i=1,2,3)$: It follows from Lemma \ref{computation F_E^0} (3) that
 	\begin{align*}
  	\langle e_{i+1},e_{i+1}\rangle_{L^2_{p,\mathcal{I}}}&=\frac{1 }{b_0m}\int_{\mathbb{R}^3}\int_0^\infty (p^i)^2   F_E^0\left(1+\frac{\mathcal{I}}{mc^2}\right)^2\phi(\mathcal{I})\,d\mathcal{I}dp\cr
  &=1.
 \end{align*}
\noindent\newline
 	$\bullet$ $\|e_5\|_{L^2_{p,I}}=1$: Inserting $(n,0,T)=(1,0,T_0)$ into \eqref{e5}, one finds
 	 \begin{align*}
 	  \left\{\widetilde{e}\right\}^{\prime}(T_0)&= \frac{1}{k_BT^2_0}\int_{\mathbb{R}^3}\int_0^\infty \left\{cp^0 \left(1+\frac{\mathcal{I}}{mc^2}\right)- \widetilde{e}(T_0) \right\}^2F_E^0\phi(\mathcal{I})\,d\mathcal{I}dp 
 	\end{align*} 
 where we used 
 $$
 \frac{e}{n}\Big|_{T=T_0}= \widetilde{e}(T)\big|_{T=T_0}=\widetilde{e}(T_0).
 $$
Using this, we have
 	\begin{align*}
 	\langle e_5,e_5\rangle_{L^2_{p,\mathcal{I}}}&= \frac{1}{k_BT_0^2\left\{\widetilde{e}\right\}^{\prime}(T_{0})}\int_{\mathbb{R}^3}\int_0^\infty\left\{cp^0\left(1+\frac{\mathcal{I}}{mc^2}\right)-\widetilde{e}(T_0)\right\}^2 F_E^0\phi(\mathcal{I})\,d\mathcal{I}dp\cr
 	&=1.	
  	\end{align*}
 	\noindent\newline 	
 	$\bullet$ $\langle e_i,e_j\rangle_{L^2_{p,\mathcal{I}}}=0$ $(i\neq j)$: Since the orthogonality can be proved in the same manner, we omit it.
 	
 \end{proof}
 \begin{proposition}\label{pro}
 	The linearized operator $L$ satisfies the following properties:
 	\begin{enumerate}
 		\item $Ker(L)=N$.
 		\item $L$ is dissipative in the following sense:
 		$$
 		\langle L(f),f\rangle_{L^2_{p,\mathcal{I}}}= -\|\{I-P\}(f)\|^2_{L^2_{p,\mathcal{I}}} \le 0.
 		$$
 	\end{enumerate}
 \end{proposition}
\begin{proof}
Since it is straightforward by the definition of $L$, we omit the proof. 
\end{proof}
 \noindent\newline
\section{Estimates for macroscopic fields and nonlinear perturbations}

To deal with the macroscopic fields, we first estimate $\Psi, \Psi_1,\Phi$ and $\Phi_1$ whose definitions are given in \eqref{notation2}. For brevity, we set $\tau=1$ and use the generic constant $C$ which can change line by line but does not affect the proof of Theorem \ref{main3}.
\begin{lemma}\label{lem22}
	Suppose $E(f)(t)$ is sufficiently small. Then $\Psi, \Psi_1,\Phi$ and $\Phi_1$ satisfy
	\begin{enumerate}
		\item $\displaystyle |\partial^\alpha \Psi|+|\partial^\alpha \Phi|\le Cp^0 \sum_{|\alpha_1|\le|\alpha|}\|\partial^{\alpha_1} f\|_{L^2_{p,\mathcal{I}}},$
		\item $\displaystyle |\partial^\alpha \Psi_1|+|\partial^\alpha \Phi_1|\le C p^0 \sqrt{E(f)(t)}\sum_{|{\alpha_1}|\le|\alpha|}\|\partial^{\alpha_1} f\|_{L^2_{p,\mathcal{I}}} .$
	\end{enumerate}
\end{lemma}
\begin{proof}
	$\bullet$ Estimates of $\Psi$ and $\Psi_1$: Recalling \eqref{notation2}, we see that
	\begin{align*}
		\partial^\alpha\Psi&=2\int_{\mathbb{R}^3}\int_0^\infty \partial^\alpha f\sqrt{F_E^0}  \phi(\mathcal{I}) \,d\mathcal{I}\,dp+	\partial^\alpha\left\{\int_{\mathbb{R}^3}\int_0^\infty f\sqrt{F_E^0}  \phi(\mathcal{I}) \,d\mathcal{I}\,dp\right\}^2\cr
		&-\sum_{i=1}^3 \partial^\alpha\left\{\int_{\mathbb{R}^3}\int_0^\infty p^if\sqrt{F_E^0}  \phi(\mathcal{I}) \,d\mathcal{I}\,\frac{dp}{p^0}\right\}^2,
	\end{align*}
and it follows from H\"{o}lder's inequality that
	\begin{align}\label{Psi}
	\begin{split}
		|\partial^\alpha \Psi|	&\le C\biggl(\|\partial^\alpha f\|_{L^2_{p,\mathcal{I}}}+  \sum_{|{\alpha_1}|\le |\alpha|}\|\partial^{\alpha_1} f\|_{L^2_{p,\mathcal{I}}}\|\partial^{\alpha-\alpha_1}f\|_{L^2_{p,\mathcal{I}}}\biggl).
\end{split}	\end{align}
	Applying the Sobolev embedding $H_x^2\subseteq L^{\infty}_x$ to lower order terms of \eqref{Psi}, we get
	\begin{align}\label{H}
		\begin{split}
			|\partial^\alpha \Psi| &\le C\biggl(\|\partial^\alpha f\|_{L^2_{p,\mathcal{I}}}+ \sqrt{E(f)(t)}\sum_{|{\alpha_1}|\le |\alpha|}\|\partial^{\alpha_1} f\|_{L^2_{p,\mathcal{I}}}\biggl)\le C\sum_{|{\alpha_1}|\le |\alpha|}\|\partial^{\alpha_1} f\|_{L^2_{p,\mathcal{I}}}
	\end{split}\end{align}
for sufficiently small $E(f)(t)$. In the same manner, one can have
	\begin{align}\label{Psi1}
		\begin{split}
			|\partial^\alpha \Psi_1|	&\le C\sum_{|{\alpha_1}|\le|\alpha|}\|\partial^{\alpha_1} f\|_{L^2_{p,\mathcal{I}}}\|\partial^{\alpha-{\alpha_1}} f\|_{L^2_{p,\mathcal{I}}}\cr
			&\le C\sqrt{E(f)(t)}\sum_{|{\alpha_1}|\le |\alpha|}\|\partial^{\alpha_1} f\|_{L^2_{p,\mathcal{I}}}.
	\end{split} \end{align}
	$\bullet$ Estimates of $\Phi$: Observe from \eqref{macroscopic fields} and $\eqref{notation2}_{3}$ that
	\begin{align}\label{phi 1}
		\begin{split}
			\Phi=U_\mu p^\mu-cp^0=\frac{c}{n}p^\mu\int_{\mathbb{R}^3}\int_0^\infty p_\mu F\phi(\mathcal{I}) \,d\mathcal{I}\,\frac{dp}{p^0}  -cp^0.
	\end{split}\end{align} 
Inserting $F=F_E^0+f\sqrt{F_E^0}$, then \eqref{phi 1} becomes 
\begin{align*}
	\Phi&=\frac{c}{n}p^\mu\int_{\mathbb{R}^3}\int_0^\infty p_\mu F_E^0\phi(\mathcal{I}) \,d\mathcal{I}\,\frac{dp}{p^0}+\frac{c}{n}p^\mu\int_{\mathbb{R}^3}\int_0^\infty p_\mu f\sqrt{F_E^0}\phi(\mathcal{I}) \,d\mathcal{I}\,\frac{dp}{p^0}  -cp^0\cr
	&=\frac{c}{n}\left( p^0+p^\mu\int_{\mathbb{R}^3}\int_0^\infty p_\mu f\sqrt{F_E^0}\phi(\mathcal{I}) \,d\mathcal{I}\,\frac{dp}{p^0}\right)-cp^0	
\end{align*}
where we used
$$
\int_{\mathbb{R}^3}\int_0^\infty  F_E^0\phi(\mathcal{I}) \,d\mathcal{I}\,dp=1,\qquad \int_{\mathbb{R}^3}\int_0^\infty  p F_E^0\phi(\mathcal{I}) \,d\mathcal{I}\,\frac{dp}{p^0}=0.
$$
	Now we use \eqref{1/n=} to expand $1/n$ as
	\begin{align}\label{Phi form}\begin{split}
		\Phi&= \left(	1-\int_{\mathbb{R}^3}\int_0^\infty f\sqrt{F_E^0}  \phi(\mathcal{I}) \,d\mathcal{I}\,dp-\frac{\Psi_1}{2}	-\frac{\Psi^3-3\Psi^2}{2(2+\Psi-\Psi^2+2\sqrt{1+\Psi})}\right)\cr
		&\times \left(cp^0+	cp^\mu\int_{\mathbb{R}^3}\int_0^\infty p_\mu f\sqrt{F_E^0}\phi(\mathcal{I}) \,d\mathcal{I}\,\frac{dp}{p^0}\right)-cp^0 \cr
		&=cp^\mu\int_{\mathbb{R}^3}\int_0^\infty p_\mu f\sqrt{F_E^0}\phi(\mathcal{I}) \,d\mathcal{I}\,\frac{dp}{p^0}+\left(cp^0+	cp^\mu\int_{\mathbb{R}^3}\int_0^\infty p_\mu f\sqrt{F_E^0}\phi(\mathcal{I}) \,d\mathcal{I}\,\frac{dp}{p^0}\right)\cr
		&\times \left( -\int_{\mathbb{R}^3}\int_0^\infty f\sqrt{F_E^0}  \phi(\mathcal{I}) \,d\mathcal{I}\,dp-\frac{\Psi_1}{2}	-\frac{\Psi^3-3\Psi^2}{2(2+\Psi-\Psi^2+2\sqrt{1+\Psi})} \right).
\end{split}	\end{align}
Since it follows from \eqref{H}, \eqref{Psi1} and the Sobolev embedding $H_x^2\subseteq L^{\infty}_x$ that 
\begin{equation}\label{phi psi}
|\Psi|\le C\|f\|_{L^2_{p,\mathcal{I}}}< C\sqrt{E(f)(t)}, \qquad | \Psi_1|\le C \|  f\|_{L^2_{p,\mathcal{I}}}
\end{equation}
for sufficiently small $E(f)(t)$,  the last identity of \eqref{Phi form} leads to
	\begin{align*}
		|\Phi|	&\le Cp^0\|f\|_{L^2_{p,\mathcal{I}}}.
		\end{align*}	
For $\alpha\neq 0$, we have from \eqref{Phi form} that
\begin{align*}
		\partial^\alpha \Phi&=cp^\mu\int_{\mathbb{R}^3}\int_0^\infty p_\mu \partial^\alpha f\sqrt{F_E^0}\phi(\mathcal{I}) \,d\mathcal{I}\,\frac{dp}{p^0}+\sum_{|\alpha_1|\le |\alpha|}C_{\alpha_1}\partial^{\alpha_1}\biggl\{ cp^0+	cp^\mu\int_{\mathbb{R}^3}\int_0^\infty p_\mu f\sqrt{F_E^0}\phi(\mathcal{I}) \,d\mathcal{I}\,\frac{dp}{p^0}\biggl\}\cr
	&\times \partial^{\alpha-\alpha_1}\left\{  -\int_{\mathbb{R}^3}\int_0^\infty f\sqrt{F_E^0}  \phi(\mathcal{I}) \,d\mathcal{I}\,dp-\frac{\Psi_1}{2}	-\frac{\Psi^3-3\Psi^2}{2(2+\Psi-\Psi^2+2\sqrt{1+\Psi})} \right\}.
\end{align*}
Using H\"{o}lder's inequality and \eqref{Psi1}, we have
	\begin{align}\label{partial phi}
		\begin{split}
 |\partial^\alpha\Phi|&\le  Cp^0\| \partial^\alpha f\|_{L^2_{p,\mathcal{I}}}+Cp^0\sum_{|\alpha_1|\le |\alpha|}   \left(  1+\| \partial^{\alpha_1} f\|_{L^2_{p,\mathcal{I}}} \right)\biggl(\| \partial^{\alpha-\alpha_1} f\|_{L^2_{p,\mathcal{I}}}\cr
 &  +\sqrt{E(f)(t)}\sum_{|\alpha_2|\le|\alpha-\alpha_1|}\| \partial^{\alpha_2} f\|_{L^2_{p,\mathcal{I}}}+\left|\frac{\mathbb{P}\left(\Psi,\cdots,\partial^{\alpha-\alpha_1} \Psi,\sqrt{1+\Psi}\right)}{\mathbb{M}\left(2+\Psi-\Psi^2+2\sqrt{1+\Psi},\sqrt{1+\Psi} \right)}\right|\biggl)
	\end{split}\end{align}
 where $\mathbb{P}$ and $\mathbb{M}$ denote the homogeneous generic polynomial and monomial respectively. Using \eqref{Psi} and the Sobolev embedding $H_x^2\subseteq L^{\infty}_x$, the last term on the r.h.s of \eqref{partial phi} can be estimated explicitly as 
 	\begin{equation}\label{1/1+pi}
\left|\frac{\mathbb{P}(\Psi,\cdots,\partial^{\alpha-\alpha_1} \Psi,\sqrt{1+\Psi},\cdots,\partial^{\alpha-\alpha_1}\sqrt{1+\Psi})}{(2+\Psi-\Psi^2+2\sqrt{1+\Psi})^{2^{|\alpha-\alpha_1|}}}\right| \le C\sum_{|\alpha_2|\le |{\alpha-\alpha_1}|}\|\partial^{\alpha_2} f\|_{L^2_{p,\mathcal{I}}}.
	\end{equation}
Inserting \eqref{1/1+pi} into \eqref{partial phi}, one finds
	\begin{align}\label{leads to}
		|\partial^\alpha \Phi|\le Cp^0\sum_{|{\alpha_1}|\le |\alpha|}\|\partial^{\alpha_1} f\|_{L^2_{p,\mathcal{I}}}.
	\end{align}
	\newline
	$\bullet$ Estimates of $\Phi_1$: In the same manner as in the case of $\Phi$, one can have
	\begin{align}\label{obtain}
		|\partial^\alpha \Phi_1|\le Cp^0\sqrt{E(f)(t)}\sum_{|{\alpha_1}|\le |\alpha|}\|\partial^{\alpha_1} f\|_{L^2_{p,\mathcal{I}}}.
	\end{align}
	As a result, (1) is obtained by  \eqref{Psi} and \eqref{leads to}, and combining \eqref{Psi1} and \eqref{obtain} gives (2).	
\end{proof}

\begin{lemma}\label{lem2} 	Suppose $E(f)(t)$ is sufficiently small. Then we have
	\begin{enumerate}
		\item $	\displaystyle\ |\partial^\alpha \{n -1\}|+|\partial^\alpha U |+\left|\partial^\alpha \Big\{\frac{e}{n} -\widetilde{e}(T_0)\Big\}\right|\le C\sum_{|{\alpha_1}|\le|\alpha|}\|\partial^{\alpha_1} f\|_{L^2_{p,\mathcal{I}}}.$
		\item $\displaystyle|\partial^{\alpha} \{U ^0-c\}|\le C\sqrt{E(f)(t)}\sum_{|{\alpha_1}|\le|\alpha|}\|\partial^{\alpha_1} f\|_{L^2_{p,\mathcal{I}}}.$
	\end{enumerate}
\end{lemma}
\begin{proof}
	$\bullet$ $\partial^\alpha \{n -1\}$ : It follows from \eqref{n=}, \eqref{route pi} and \eqref{phi psi} that
	\begin{align}\label{n-1}
		\left|n -1\right| &= \left|\frac{\Psi}{2}-\frac{\Psi^2}{2(2+\Psi+2\sqrt{1+\Psi})}\right|		  \le C\|f\|_{L^2_{p,\mathcal{I}}}
	\end{align}
for sufficiently small $E(f)(t)$. 	For $\alpha\neq 0$, it takes the form of
	\begin{align*}
		\begin{split}
			\partial^\alpha \{n -1\}&=\partial^\alpha\left\{\sqrt{1+\Psi}\right\}=\frac{\mathbb{P}( \Psi,\cdots \partial^\alpha \Psi,\sqrt{1+\Psi})}{\mathbb{M}\left(\sqrt{1+\Psi}\right) }
		\end{split}
	\end{align*}
which can be estimated explicitly by using \eqref{phi psi} as follows 
	\begin{align*}
		\left|	\partial^\alpha \{n -1\}\right| 	&\le C\sum_{|{\alpha_1}|\le |\alpha|}\|\partial^{\alpha_1}f\|_{L^2_{p,\mathcal{I}}}
	\end{align*}
for sufficiently small $E(f)(t)$.
	\noindent\newline
	$\bullet$ $\partial^\alpha U $ : Observe that	\begin{align*}
		 |U|  &=\left| \frac{1}{n}\int_{\mathbb{R}^3}\int_0^\infty p F\phi(\mathcal{I})\,d\mathcal{I}\frac{dp}{p^0}\right| =\left| \frac{1}{n}\int_{\mathbb{R}^3}\int_0^\infty p f\sqrt{F_E^0}\phi(\mathcal{I})\,d\mathcal{I}\frac{dp}{p^0}\right|\le  \frac{C}{n}\|f\|_{L^2_{p,\mathcal{I}}} . 
	\end{align*}
Using \eqref{n-1} and the Sobolev embedding $H_x^2\subseteq L^{\infty}_x$, this leads to
\begin{align*}
	|U | \leq  \frac{C}{ 1-C\sqrt{E(f)(t)}  }\|  f\|_{L^2_{p,\mathcal{I}}}\leq C\|  f\|_{L^2_{p,\mathcal{I}}}
\end{align*}
for sufficiently small $E(f)(t)$. 	For $\alpha\neq 0$, it follows from \eqref{1/n=}, \eqref{Psi1} and \eqref{phi psi} that
	\begin{align*}
		|\partial^\alpha U |&= \left| \sum_{|{\alpha_1}|\le|\alpha|}C_{\alpha_1}\partial^{\alpha_1}\left\{\frac{1}{n} \right\}\int_{\mathbb{R}^3}\int_0^\infty p\partial^{\alpha-{\alpha_1}} f\sqrt{F_E^0} \phi(\mathcal{I})\,d\mathcal{I}\frac{dp}{p^0}  \right|\cr
		&\le C  \left(1+\| f\|_{L^2_{p,\mathcal{I}}} \right) \|\partial^{\alpha}f\|_{L^2_{p,\mathcal{I}}}+C \sum_{0<|{\alpha_1}|\le |\alpha|}\|\partial^{\alpha-{\alpha_1}}f\|_{L^2_{p,\mathcal{I}}}\cr
		&\times\left(  \|\partial^{ \alpha_1 }f\|_{L^2_{p,\mathcal{I}}}+\sum_{|\alpha_2|\le |\alpha_1|}	\|\partial^{ \alpha_2}f\|_{L^2_{p,\mathcal{I}}}+\left|\frac{\mathbb{P}\left(\Psi,\cdots,\partial^{ \alpha_1} \Psi,\sqrt{1+\Psi}\right)}{\mathbb{M}\left(2+\Psi-\Psi^2+2\sqrt{1+\Psi},\sqrt{1+\Psi} \right)}\right| \right) 
	\end{align*}
for sufficiently small $E(f)(t)$. Using \eqref{1/1+pi} and the Sobolev embedding $H^2(\mathbb{R}_x^{3})\subseteq L^{\infty}(\mathbb{R}_x^{3})$, we have
	\begin{align*}\
		|\partial^\alpha U |&\le C\sum_{|{\alpha_1}|\le|\alpha|}\|\partial^{\alpha_1}f\|_{L^2_{p,\mathcal{I}}}.
	\end{align*}
	\newline
	$\bullet$ $\displaystyle\partial^\alpha \left\{\frac{e}{n} -\widetilde{e}(T_0)\right\}$ : Recall from \eqref{e-e0} that
 	\begin{eqnarray*}
	&&\frac{e}{n} -\widetilde{e}(T_0) =  \int_{\mathbb{R}^3}\int_0^\infty \left\{cp^0\left(1+\frac{\mathcal{I}}{mc^2}\right)-\widetilde{e}(T_0)\right\} f\sqrt{F_E^0} \phi(\mathcal{I})\,d\mathcal{I}dp   \cr
	&&\hspace{12mm} -c\int_{\mathbb{R}^3}\int_0^\infty  f\sqrt{F_E^0}\phi(\mathcal{I}) \,d\mathcal{I}dp\int_{\mathbb{R}^3}\int_0^\infty p^0f\sqrt{F_E^0}\left(1+\frac{\mathcal{I}}{mc^2}\right)\phi(\mathcal{I})\,d\mathcal{I}\,dp\cr
	&&\hspace{12mm}-\frac{1}{c}\left(	\frac{\Psi_1}{2}	+\frac{\Psi^3-3\Psi^2}{2(2+\Psi-\Psi^2+2\sqrt{1+\Psi})}\right) \int_{\mathbb{R}^3}\int_0^\infty \left(U^\mu p_\mu\right)^2F\left(1+\frac{\mathcal{I}}{mc^2}\right)\phi(\mathcal{I})\,d\mathcal{I}\frac{dp}{p^0}\cr
	&&\hspace{12mm}+\frac{1}{c}\left( 1-\int_{\mathbb{R}^3}\int_0^\infty f\sqrt{F_E^0}  \phi(\mathcal{I}) \,d\mathcal{I}\,dp \right) \int_{\mathbb{R}^3}\int_0^\infty \left\{	2cp^0\Phi+\Phi^2\right\}F\left(1+\frac{\mathcal{I}}{mc^2}\right)\phi(\mathcal{I})\,d\mathcal{I}\frac{dp}{p^0}.
\end{eqnarray*}
This can be handled in the same manner as in the previous cases but it requires tedious computations, so we omit the proof for brevity.

	\noindent\newline
	$\bullet$ $\partial^\alpha \{U^0 -c\}$ : By \eqref{route pi}, $U^0-c$ can be rewritten as
	\begin{align*}
		U ^0-c&=\sqrt{c^2+|U |^2}-c\cr
		&=\frac{|U|^2}{2c^2}-\frac{|U|^4}{2(2c^4+c^2|U|^2+2c^3\sqrt{c^2+|U|^2})}.
	\end{align*}
	Then it follows from the previous result of $U$ that
	\begin{align*}
		|U ^0-c|&=\left|\frac{|U|^2}{2c^2}-\frac{|U|^4}{2(2c^4+c^2|U|^2+2c^3\sqrt{c^2+|U|^2})}\right|\cr
		&\le C\|f\|^2_{L^2_{p,\mathcal{I}}}+C\|f\|_{L^2_{p,\mathcal{I}}}^4\cr
		&\le C\sqrt{E(f)(t)}\|f\|_{L^2_{p,\mathcal{I}}}
	\end{align*}
for sufficiently small $E(f)(t)$. Similarly, one can have
	\begin{align*}
		\left|	\partial^\alpha \{U^0 -c\}\right|&=\left|\partial^\alpha \left\{\frac{|U|^2}{2c^2}-\frac{|U|^4}{2(2c^4+c^2|U|^2+2c^3\sqrt{c^2+|U|^2})} \right\}\right|\cr
		&\le C\sum_{|{\alpha_1}|\le|\alpha|}|\partial^{\alpha_1}U| |\partial^{\alpha-{\alpha_1}}U| + \left|\frac{\mathbb{P}( U , \cdots, \partial^\alpha U,\sqrt{c^2+|U ^2|} )}{\mathbb{M}(2c^4+c^2|U |^2+2c^3\sqrt{c^2+|U |^2}, \sqrt{c^2+|U| ^2 }) }\right|\cr
		&\le C\sqrt{E(f)(t)}\sum_{|{\alpha_1}|\le|\alpha|}\|\partial^{\alpha_1} f\|_{L^2_{p,\mathcal{I}}}.
	\end{align*}
\end{proof}
In the following lemma, we present the Hessian matrix $D^2\widetilde{F}(\theta)$ with respect to $n_{\theta},U_\theta,(e/n)_{\theta},q^\mu_{\theta}$ to clarify the explicit form of the nonlinear perturbation $\Gamma_4$ given in Lemma \ref{lin2}.
\begin{lemma}\label{QQ} We have
	$$
	D_{n_{\theta},U_\theta,(e/n)_{\theta},q^\mu_{\theta}}^{2}\tilde{F}(\theta)=\mathcal{Q} F(\theta),
	$$
	where $\mathcal{Q}$ is a $9\times9$ symmetric matrix whose elements are given by
	\begin{eqnarray*}
	&&\mathcal{Q}_{1,1}=0,\qquad \mathcal{Q}_{1,1+i}= -\frac{ 1+\frac{\mathcal{I}}{mc^2} }{k_Bn_\theta T_\theta}  \left(\frac{U_\theta^i}{U_\theta^0}p^0-p^i\right),\cr
	&&\mathcal{Q}_{1,5}=\frac{1}{k_Bn_\theta T_\theta^2 }\frac{1}{\{\widetilde{e}\}^{\prime}(T_\theta)}\left\{\frac{k_BT_\theta^2}{mc^2}\widetilde{e}(T_\theta) +\left(1+\frac{\mathcal{I}}{mc^2}\right) U_{\theta\mu}p^\mu \right\},\qquad \mathcal{Q}_{1,6}=0,\qquad\mathcal{Q}_{1,6+i}=0,\cr
	&&\mathcal{Q}_{1+i,1+j}=\biggl(1-p^\mu q_{\theta\mu}\frac{1+\frac{\mathcal{I}}{mc^2}}{b_\theta mc^2}\biggl)\frac{ 1+\frac{\mathcal{I}}{mc^2} }{k_BT_\theta}\left\{\frac{ 1+\frac{\mathcal{I}}{mc^2}  }{k_BT_\theta}  \biggl(\frac{U_\theta^i}{U_\theta^0}p^0-p^i\biggl)\biggl(\frac{U_\theta^j}{U_\theta^0}p^0-p^j\biggl)+ \frac{U^i_\theta U^j_\theta}{(U_\theta^0)^3}p^0\right\} \cr
	&&\mathcal{Q}_{1+i,5}=  \left(\frac{U^i_\theta}{U^0_\theta}p^0-p^i\right)\frac{(1+\frac{\mathcal{I}}{mc^2} )}{k_BT_\theta^2\{\widetilde{e}\}^\prime (T_\theta)}\biggl\{ \biggl(1-p^\mu q_{\theta\mu}\frac{1+\frac{\mathcal{I}}{mc^2}}{b_\theta mc^2}\biggl)\cr
	&&\hspace{11mm}\times\biggl(1-  \frac{T_\theta\widetilde{e}(T_\theta)}{mc^2}-  \frac{1+\frac{\mathcal{I}}{mc^2}}{k_BT_\theta}U_{\theta\mu}p^\mu\biggl) -  p^\mu q_{\theta\mu}T_\theta\frac{1+\frac{\mathcal{I}}{mc^2}}{mc^2b_\theta^2  } \partial_{T_\theta}b_\theta   \biggl\},\cr
	&&\mathcal{Q}_{1+i,6}=  \frac{  (1+\frac{\mathcal{I}}{mc^2} )^2}{b_\theta mc^2k_BT_\theta}  \left(\frac{U_\theta^i}{U_\theta^0}p^0-p^i\right)p^0,\qquad \mathcal{Q}_{1+i,6+j}=  - \frac{  (1+\frac{\mathcal{I}}{mc^2} )^2}{b_\theta mc^2k_BT_\theta}  \left(\frac{U_\theta^i}{U_\theta^0}p^0-p^i\right)p^j, \cr
		&& \mathcal{Q}_{5,5}=\frac{1-p^\mu q_{\theta\mu}\frac{1+\frac{\mathcal{I}}{mc^2}}{b_\theta mc^2}}{\left(\{\widetilde{e}\}^\prime(T_\theta)\right)^2}\left\{- \frac{mc^2\widetilde{e}(T_\theta)}{k_B^2T_\theta^4} +mc^2\biggl(1+\frac{\mathcal{I}}{mc^2}\biggl) \frac{U_{\theta\mu}p^\mu}{k_B^2T_\theta^4}  +\frac{2}{ T_\theta}+ \frac{\{\widetilde{e}\}^{\prime\prime}(T_\theta)}{\{\widetilde{e}\}^\prime(T_\theta)}-\frac{ \{\widetilde{e}\}^\prime(T_\theta)}{ k_BT_\theta^2}\right\}\cr
		&&\hspace{8mm}+\frac{p^\mu q_{\theta\mu}\left(1+\frac{\mathcal{I}}{mc^2}\right)}{b_\theta^2 \left(\{\widetilde{e}\}^\prime(T_\theta)\right)^2} \frac{\partial_{T_\theta}b_\theta}{k_BT_\theta^2}  \biggl\{-\frac{2\widetilde{e}(T_\theta)}{mc^2} +2\biggl(1+\frac{\mathcal{I}}{mc^2}\biggl)\frac{U_{\theta\mu}p^\mu}{mc^2}+2\frac{k_B^2T_\theta^3}{m^2c^4}+\frac{k_B^2T_\theta^4}{m^2c^4}\frac{\{\widetilde{e}\}^{\prime\prime}(T_\theta)}{\{\widetilde{e}\}^\prime(T_\theta)} \cr
		&&\hspace{8mm}-\frac{2k_BT_\theta^2}{mc^2}\frac{\partial_{T_{\theta}}b_\theta}{b_\theta}+\frac{k_BT_\theta^2}{mc^2}\frac{\partial^2_{T_{\theta}}b_\theta}{\partial_{T_\theta}b_\theta}+\frac{2k_BT_\theta}{mc^2}\biggl\}\cr
&& \mathcal{Q}_{5,6}=-  \frac{1+\frac{\mathcal{I}}{mc^2}}{b_\theta k_BT_\theta^2\{\widetilde{e}\}^\prime(T_\theta) }p^0\left\{-\frac{\widetilde{e}(T_\theta)}{mc^2}+\left(1+\frac{\mathcal{I}}{mc^2}\right)\frac{U_{\theta\mu}p^\mu}{mc^2}-\frac{k_BT_\theta^2}{mc^2}\frac{\partial_{T_\theta}b_\theta}{b_\theta}  \right\}, \cr
		&& \mathcal{Q}_{5,6+i}=\frac{1+\frac{\mathcal{I}}{mc^2}}{b_\theta k_BT_\theta^2\{\widetilde{e}\}^\prime(T_\theta)}p^i  \left\{-\frac{\widetilde{e}(T_\theta)}{mc^2}+\left(1+\frac{\mathcal{I}}{mc^2}\right)\frac{U_{\theta\mu}p^\mu}{mc^2}-\frac{k_BT_\theta^2}{mc^2}\frac{\partial_{T_\theta}b_\theta}{b_\theta}  \right\}, \cr
	&&\mathcal{Q}_{6,6}=0,\qquad \mathcal{Q}_{6,6+i}=0,\qquad\mathcal{Q}_{6+i,6+j}=0
	\end{eqnarray*}
	for $i,j=1,2,3$. 

\end{lemma}
\begin{proof}
	The proof is straightforward. We omit it.
\end{proof}
We are now ready to deal with the nonlinear perturbations. 
\begin{lemma}\label{nonlinear1}	Suppose $E(f)(t)$ is sufficiently small. Then we have
\begin{enumerate}
	\item $\displaystyle 
	\left|\int_{\mathbb{R}^3}\int_0^\infty\partial^\alpha_\beta\Gamma (f) g(p,\mathcal{I}) \phi(\mathcal{I})\,d\mathcal{I}dp\right|\le C\sum_{|{\alpha_1}|\le |\alpha|}\|\partial^{\alpha_1} f\|_{L^2_{p,\mathcal{I}}}\|\partial^{\alpha-{\alpha_1}} f\|_{L^2_{p,\mathcal{I}}}\|g\|_{L^2_{p,\mathcal{I}}},
	$ 
	\item $\displaystyle 
	\left\|\partial^\alpha_\beta\Gamma (f)  \right\|_{L^2_{p,\mathcal{I}}}\le C\sum_{|{\alpha_1}|\le |\alpha|}\|\partial^{\alpha_1} f\|_{L^2_{p,\mathcal{I}}}\|\partial^{\alpha-{\alpha_1}} f\|_{L^2_{p,\mathcal{I}}}.
	$
\end{enumerate}
\end{lemma}
\begin{proof}
$\bullet$ Proof of (1):	Recall from Proposition \ref{lin3} that
	$$
	\Gamma (f)=\frac{U_\mu p^\mu}{c  p^0} \sum_{i=1}^4\Gamma_i (f)+\frac{P(f)-f}{c p^0}\Phi .
	$$
To avoid the repetition, we only prove the most complicated term:
	$$
\frac{U_\mu p^\mu}{c  p^0} \Gamma_{4} (f)=\frac{U_\mu p^\mu}{c p^0}\frac{1}{\sqrt{F_E^0}}\int_0^1 (1-\theta)\Big(n -1,U,\frac{e}{n} -\widetilde{e}(T_0),q^\mu\Big)D^2\widetilde{F}(\theta)\Big(n -1,U ,\frac{e}{n} -\widetilde{e}(T_0),q^\mu\Big)^Td\theta
	$$
since the other terms can be handled in the same manner.
For this, we use the following notation  
 $$(y_{1},\cdots,y_9):=\left(n-1,U,\frac{e}{n}-\widetilde{e}(T_0),q^\mu\right)$$ 
to denote
	$$
	\left(n-1,U,\frac{e}{n}-\widetilde{e}(T_0),q^\mu\right)D_{n_{\theta},U_{\theta},(e/n)_{\theta},q^\mu_\theta}^{2}\widetilde{F}(\theta)\left(n-1,U,\frac{e}{n}-\widetilde{e}(T_0),q^\mu\right)^{T}=\sum_{i,j=1}^{9}y_{i}y_{j}\mathcal{Q}_{ij}F(\theta)
	$$
where $\mathcal{Q}_{ij}$ $(i,j=1,\cdots, 9)$ was given in Lemma \ref{QQ}. We observe that 
	\begin{align}\label{long}\begin{split}
		 & \partial^{\alpha}_{\beta}\biggl\{\frac{1}{\sqrt{F_E^0}}\int_0^1 (1-\theta)\left(n -1,U,\frac{e}{n} -\widetilde{e}(T_0),q^\mu\right)D^2\widetilde{F}(\theta)\left(n -1,U ,\frac{e}{n} -\widetilde{e}(T_0),q^\mu\right)^T\,d\theta\biggl\}\cr
		 &=\sum_{i,j=1}^9\sum_{\substack{|\alpha_{1}|+\cdots+|\alpha_{4}|\cr=|\alpha|}}\!\!\!\!\partial^{\alpha_{1}}y_{i}\partial^{\alpha_{2}}y_{j}
		\biggl(\sum_{\substack{|\beta_{1}|+\cdots+|\beta_{3}|\cr=|\beta|}}
		\int^1_0(1-\theta)\partial^{\alpha_{3}}_{\beta_{1}}\mathcal{Q}_{ij}\partial^{\alpha_{4}}_{\beta_{2}}F(\theta)d\theta\biggl) \partial_{\beta_{3}}\biggl\{\frac{1}{\sqrt{F_E^0}}\biggl\}.
\end{split}	\end{align}
Here $\partial^{\alpha_1}y_i \partial^{\alpha_2}y_j$ is bounded from above by Lemma \ref{lem2}:
	\begin{align}\label{yiyj0}\begin{split}
		 \big|\partial^{\alpha_1} y_i \partial^{\alpha_2}y_j\big|&\le C	\sum_{|\alpha'_1|\le|\alpha_1|}\|\partial^{\alpha'_1}f\|_{L^2_{p,\mathcal{I}}}\sum_{|\alpha'_2| \le|\alpha_2|}\|\partial'^{\alpha_2}f\|_{L^2_{p,\mathcal{I}}}.
\end{split}\end{align}
Note from Lemma \ref{lem22} and Lemma \ref{lem2} that derivatives of the macroscopic fields $n,U,e/n$ and $q^\mu$ are dominated by the $L^2_{p,\mathcal{I}}$ norm of the derivative of $f$, and with the aid of the Sobolev embedding $H^2(\mathbb{R}_x^{3})\subseteq L^{\infty}(\mathbb{R}_x^{3})$, one can see that for $|\alpha|\le N-2$,
$$
 \partial^{\alpha}\left(n-1,U,\frac{e}{n}-\widetilde{e}(T_0),q^\mu\right)\approx(0,0,0,0)
$$
 when $E(f)(t)$ is small enough. From this observation, we see that $\partial^{\alpha_3}_{\beta_1} \mathcal{Q}_{ij}$ is well-defined and can be estimated as
	\begin{align}\label{yiyj}\begin{split}
	 	|\partial^{\alpha_3}_{\beta_1} \mathcal{Q}_{ij}|&\le C(p^0)^3\biggl( 1+\frac{\mathcal{I}}{mc^2}\biggl)\left(1+ \| f\|_{L^2_{p,\mathcal{I}}}\right),\qquad \qquad\text{if }\qquad |\alpha_3|=0\cr
	 	|\partial^{\alpha_3}_{\beta_1} \mathcal{Q}_{ij}|&\le C(p^0)^3\biggl( 1+\frac{\mathcal{I}}{mc^2}\biggl) \sum_{|\alpha_3'|\le|\alpha_3|}\|\partial^{\alpha_3'}f\|_{L^2_{p,\mathcal{I}}}, \hspace{6.5mm}\text{otherwise }
		\end{split}\end{align}
for sufficiently small $E(f)(t)$. For the same reason, one can have
\begin{align}\label{partial F}\begin{split}
	\left| \partial^{\alpha_4}_{\beta_2} F(\theta)\right| \le C(p^0)^{|\alpha_4|} \biggl(1+\frac{\mathcal{I}}{mc^2} \biggl)^{|\alpha_4|+|\beta_2|} F(\theta).
\end{split}\end{align}
By a definition of $F_E^0$, one finds
\begin{equation*}
	\biggl|\partial_{\beta_3}\biggl\{\frac{1}{\sqrt{F_E^0}}\biggl\}\biggl|\le C\biggl(1+\frac{\mathcal{I}}{mc^2}\biggl)^{|\beta_3|}\frac{1}{\sqrt{F_E^0}}
\end{equation*}
which, together with \eqref{partial F} gives
\begin{align}\label{F theta}\begin{split}
		\biggl|\partial^{\alpha_4}_{\beta_2} F(\theta)\partial_{\beta_3}\biggl\{\frac{1}{\sqrt{F_E^0}}\biggl\}\biggl|  &\le  C(p^0)^{|\alpha_4|} \biggl(1+\frac{\mathcal{I}}{mc^2} \biggl)^{|\alpha_4|+|\beta_2|+|\beta_3|}\frac{F(\theta)}{\sqrt{F_E^0}}\cr
		&\le \left|\mathbb{P}\big(p^0,\mathcal{I}\big)\right|e^{-C^{\prime}\left(1+\frac{\mathcal{I}}{mc^2}\right)p^0}
\end{split}\end{align}
for sufficiently small $E(f)(t)$. Here  $C^\prime$ is the positive constant defined as
 \begin{align*}
	 e^{-\left(1+\frac{\mathcal{I}}{mc^2}\right)\frac{1}{k_B T_\theta}U_\theta^\mu p_\mu}e^{\frac{1}{2}\left(1+\frac{\mathcal{I}}{mc^2}\right)\frac{cp^0}{k_BT_0}}&\le   e^{-\left(1+\frac{\mathcal{I}}{mc^2}\right)\frac{1}{k_B T_\theta}\left(\sqrt{c^2+|U_\theta|^2}-|U_\theta|\right) p^0}e^{\frac{1}{2}\left(1+\frac{\mathcal{I}}{mc^2}\right)\frac{cp^0}{k_BT_0}}\cr
 	&\le  e^{-\min\left\{\frac{1}{k_B T_\theta}\left(\sqrt{c^2+|U_\theta|^2}-|U_\theta|\right)-\frac{c}{2k_BT_0} \right\}\left(1+\frac{\mathcal{I}}{mc^2}\right) }\cr
 	&\equiv e^{-C^\prime \left(1+\frac{\mathcal{I}}{mc^2}\right)p^0 }
\end{align*}
where we assumed that $E(f)(t)$ is small enough to satisfy
\begin{align*}
	\min \frac{\sqrt{c^2+|U_\theta|^2}-|U_\theta|}{  T_\theta}   > \frac{c}{2 T_0}
\end{align*}
to ensure the positivity of $C^\prime$. 
Combining \eqref{long}--\eqref{yiyj} and \eqref{F theta}, and applying the Sobolev embedding $H^2(\mathbb{R}_x^{3})\subseteq L^{\infty}(\mathbb{R}_x^{3})$ to the terms having lower derivative order, we get
	\begin{align*}
				 \partial^{\alpha}_{\beta}\Gamma_4(f)&=\partial^{\alpha}_{\beta}\biggl\{\frac{1}{\sqrt{F_E^0}}\int_0^1 (1-\theta)\left(n -1,U,\frac{e}{n} -\widetilde{e}(T_0),q^\mu\right)D^2\widetilde{F}(\theta)\left(n -1,U ,\frac{e}{n} -\widetilde{e}(T_0),q^\mu\right)^T\,d\theta\biggl\}\cr
				 &\le \left| \mathbb{P}(p^0,\mathcal{I})\right| \sum_{|\alpha'|\le|\alpha|}\|\partial^{\alpha'}f\|_{L^2_{p,\mathcal{I}}}\|\partial^{\alpha-\alpha'}f\|_{L^2_{p,\mathcal{I}}}e^{-C^{\prime}\left(1+\frac{\mathcal{I}}{mc^2}\right)p^0}
	\end{align*}
which, together with \eqref{phi condition} and Lemma \ref{lem2} gives the desired result as follows
	\begin{align*}
		&\left|\int_{\mathbb{R}^3}\int_0^\infty \partial^{\alpha}_{\beta}\left\{\frac{U_\mu p^\mu}{c p^0} \Gamma_{4} (f)\right\}g(p,\mathcal{I})\phi(\mathcal{I})\, d\mathcal{I}dp\right|\cr
		&\le   C\sum_{\substack{|{\alpha_1}|\le |\alpha|\cr |\beta_1|\le|\beta|}}\int_{\mathbb{R}^3}\int_0^\infty \left|\biggl(\partial^{\alpha-{\alpha_1}} \{U^0\}-\partial^{\alpha-{\alpha_1}}\{U\}\cdot \partial_{\beta-\beta_1}\{\hat{p}\}\biggl) \partial^{\alpha_1}_{\beta_1}\Gamma_{4} (f)g(p,\mathcal{I})\phi(\mathcal{I})\right|\,d\mathcal{I} dp\cr
		&\le C \sum_{|{\alpha_1}|\le|\alpha|}\|\partial^{{\alpha_1}}f\|_{L^2_{p,\mathcal{I}}}\|\partial^{\alpha-{\alpha_1}}f\|_{L^2_{p,\mathcal{I}}}\int_{\mathbb{R}^3}\int_0^\infty \left|\mathbb{P}(p^0,\mathcal{I})\right|e^{-C^{\prime}\left(1+\frac{\mathcal{I}}{mc^2}\right)p^0}\left|g(p,\mathcal{I})\right|\phi(\mathcal{I})\,d\mathcal{I}dp\cr
		&\le C \sum_{|{\alpha_1}|\le|\alpha|}\|\partial^{{\alpha_1}}f\|_{L^2_{p,\mathcal{I}}}\|\partial^{\alpha-{\alpha_1}}f\|_{L^2_{p,\mathcal{I}}}\|g\|_{L^2_{p,\mathcal{I}}}.
	\end{align*}
$\bullet$ Proof of (2): Since the proof is the same as (1), we omit it.
\end{proof}
The following lemma is necessary to prove the uniqueness of solutions.
\begin{lemma}\label{uniqueness}
	Assume $\bar{F}:=F_E^0+\bar{f}\sqrt{F_E^0}$ is another solution of \eqref{PR}. For sufficiently small $E(f)(t)$ and $E(\bar{f})(t)$, we then have
	$$
	\left|	\int_{\mathbb{R}^3}\int_{\mathbb{R}^3}\int_{0}^\infty \left\{\Gamma(f)-\Gamma(\bar{f})\right\}(f-\bar{f})\phi(\mathcal{I}) \,d\mathcal{I}dpdx\right|\le
	C \|f-\bar{f}\|^{2}_{L^2_{x,p,\mathcal{I}}}.$$
\end{lemma}
\begin{proof}
	Since it can be proved in the same manner as in Lemma \ref{nonlinear1}, we omit it. 
\end{proof}
%
%
%
%
%
%
%
\begin{lemma}\label{nonlinear3}
Suppose $E(f)(t)$ is sufficiently small. Then we have
	$$
	\int_{\mathbb{R}^3}\int_{\mathbb{R}^3}\int_0^\infty \partial^{\alpha}\Gamma(f)\partial^{\alpha} P(f) \phi(\mathcal{I})\,d\mathcal{I}dpdx=0.
	$$
\end{lemma}
\begin{proof}
Recall from Proposition \ref{lin3} that
	\begin{align*}
\left\{ L(f)+\Gamma(f) \right\}\sqrt{F_E^0}&=\frac{U_\mu p^\mu}{c p^0}\left\{\biggl(1-p^\mu q_\mu \frac{1+\frac{\mathcal{I}}{mc^2}}{bmc^2}\biggl)F_E-F\right\}\equiv \frac{1}{p^0}Q.
\end{align*}
We then have from Proposition \ref{pro} (1) that
	\begin{eqnarray}\label{gamma p}\begin{split}
		&		\int_{\mathbb{R}^3}\int_{\mathbb{R}^3}\int_0^\infty\partial^{\alpha} \Gamma(f)\partial^{\alpha} P(f) \phi(\mathcal{I})\,d\mathcal{I}\,dpdx\cr
		&=	\int_{\mathbb{R}^3}\int_{\mathbb{R}^3}\int_0^\infty \partial^{\alpha}\left\{	 \frac{1}{p^0\sqrt{F_E^0}}Q-L(f)  \right\}\partial^{\alpha}P(f) \phi(\mathcal{I})\,d\mathcal{I}\,dpdx\cr
&=			\int_{\mathbb{R}^3}\int_{\mathbb{R}^3}\int_0^\infty  	 \frac{1}{p^0\sqrt{F_E^0}}\partial^{\alpha}Q P(\partial^{\alpha}f) \phi(\mathcal{I})\,d\mathcal{I}dpdx-\langle L(\partial^{\alpha}f),P(\partial^{\alpha}f) \rangle_{L^2_{x,p,\mathcal{I}}}\cr
	&=	\int_{\mathbb{R}^3}\int_{\mathbb{R}^3}\int_0^\infty \partial^{\alpha}Q\frac{P(\partial^{\alpha}f)}{\sqrt{F_E^0}} \phi(\mathcal{I})\,d\mathcal{I}\,\frac{dp}{p^0}dx.
	\end{split}\end{eqnarray}
Here $P(f)$ is given in \eqref{Pf} as 
\begin{equation}\label{Pf ex}
 P(f) = a(t,x)\sqrt{F_E^0}+b(t,x)\cdot\left(1+\frac{\mathcal{I}}{mc^2}\right) p\sqrt{F_E^0}+c(t,x)\left\{cp^0\left(1+\frac{\mathcal{I}}{mc^2}\right)-\widetilde{e}(T_0)\right\}\sqrt{F_E^0} 
\end{equation}
where 
\begin{align*}
	a(t,x)&= \int_{\mathbb{R}^3}\int_0^\infty f\sqrt{F_E^0}  \phi(\mathcal{I}) \,d\mathcal{I}\,dp,\qquad b(t,x)=\frac{1}{b_0m}\int_{\mathbb{R}^3}\int_0^\infty p   f\sqrt{F_E^0}\left(1+\frac{\mathcal{I}}{mc^2}\right)\phi(\mathcal{I})\,d\mathcal{I}dp,\cr
	&c(t,x)= \frac{1}{k_BT_0^2\left\{\widetilde{e}\right\}^{\prime}(T_0)}\int_{\mathbb{R}^3}\int_0^\infty \left\{cp^0\left(1+\frac{\mathcal{I}}{mc^2}\right)-\widetilde{e}(T_0)\right\} f\sqrt{F_E^0} \phi(\mathcal{I})\,d\mathcal{I}dp.
\end{align*}
Inserting \eqref{Pf ex} into \eqref{gamma p}, one finds 
\begin{eqnarray*}
	&&	\int_{\mathbb{R}^3}\int_{\mathbb{R}^3}\int_0^\infty \partial^{\alpha}Q\frac{P(\partial^{\alpha}f)}{\sqrt{F_E^0}} \phi(\mathcal{I})\,d\mathcal{I}\,\frac{dp}{p^0}dx\cr
	&&=\int_{\mathbb{R}^3}  \big\{   \partial^\alpha a(t,x)-\widetilde{e}(T_0)\partial^\alpha c(t,x) \big\} \partial^{\alpha}\left\{\int_{\mathbb{R}^3}\int_0^\infty Q  \phi(\mathcal{I})\,d\mathcal{I}\frac{dp}{p^0}\right\}dx\cr
&&+\int_{\mathbb{R}^3}      \partial^\alpha b(t,x)  \cdot  \partial^{\alpha}\left\{\int_{\mathbb{R}^3}\int_0^\infty pQ    \left(1+\frac{\mathcal{I}}{mc^2}\right) 
 \phi(\mathcal{I})\,d\mathcal{I}\frac{dp}{p^0}\right\}dx\cr
&&+c\int_{\mathbb{R}^3}      \partial^\alpha c(t,x)    \partial^{\alpha}\left\{\int_{\mathbb{R}^3}\int_0^\infty p^0Q    \left(1+\frac{\mathcal{I}}{mc^2}\right) 
 \phi(\mathcal{I})\,d\mathcal{I}\frac{dp}{p^0}\right\}dx
\end{eqnarray*}
which, combined with \eqref{conservation laws} gives the desired result.
\end{proof}

\section{Proof of theorem \ref{main3}}
%
%
%
%
\subsection{Local in time existence}
Using Lemma \ref{nonlinear1} and Lemma \ref{uniqueness}, the local in time existence and uniqueness of solutions to \eqref{PR} can be obtained by the standard argument \cite{Guo whole,Guo VMB}:
\begin{proposition}
	Let $N\geq 3$ and $F_{0}=F_E^0+\sqrt{F_E^0}f_{0}$ be positve. 
	Then there exist $M_{0}>0$ and $T_{*}>0$ such that if $T_{*}\le\frac{M_{0}}{2}$ and $E(f_{0})\le \frac{M_{0}}{2}$, there is a unique solution $F(x,p,t)$ to \eqref{PR} such that the energy functional is continuous in $[0,T_{*})$ and uniformly bounded:
	$$
	\sup_{0\le t\le T_{*}}E(f)(t)\le M_{0}.
	$$
\end{proposition}
\subsection{Global in time existence}
Recall from Lemma \ref{ortho} that $P(f)$ is an orthonormal projection defined by
\begin{equation}\label{P(f)}
P(f)=\biggl\{a(t,x)+b(t,x)\cdot\Big(1+\frac{\mathcal{I}}{mc^2}\Big) p+c(t,x)\biggl(cp^0\Big(1+\frac{\mathcal{I}}{mc^2}\Big)-\widetilde{e}(T_0)\biggl)\biggl\}\sqrt{F_E^0}
\end{equation}
where 
\begin{align*}
a(t,x)&= \int_{\mathbb{R}^3}\int_0^\infty f\sqrt{F_E^0}  \phi(\mathcal{I}) \,d\mathcal{I}\,dp,\qquad b(t,x)=\frac{1}{b_0m}\int_{\mathbb{R}^3}\int_0^\infty p   f\sqrt{F_E^0}\left(1+\frac{\mathcal{I}}{mc^2}\right)\phi(\mathcal{I})\,d\mathcal{I}dp,\cr
&c(t,x)= \frac{1}{k_BT_0^2\left\{\widetilde{e}\right\}^{\prime}(T_0)}\int_{\mathbb{R}^3}\int_0^\infty \left\{cp^0\left(1+\frac{\mathcal{I}}{mc^2}\right)-\widetilde{e}(T_0)\right\} f\sqrt{F_E^0} \phi(\mathcal{I})\,d\mathcal{I}dp.
\end{align*}
\noindent\newline
Decompose $f$ as
\begin{equation}\label{decom}
f=P(f)+\{I-P\}(f)
\end{equation}
and insert \eqref{decom} into \eqref{LAW} to see that
\begin{align}\label{macro-micro1}\begin{split}
\{\partial_{t}+\hat{p}\cdot\nabla_{x}\}P(f)&=\{-\partial_{t}-\hat{p}\cdot\nabla_{x}+L\}\{I-P\}(f)+\Gamma (f)\cr
&\equiv l\{I-P\}(f)+h(f)
\end{split}\end{align}
where we used $L[P(f)]=0$ (see Prosposition \ref{pro}). Observe that 
\begin{eqnarray*}
&&\{\partial_{t}+\hat{p}\cdot\nabla_{x}\}P(f)\cr
&&=\{\partial_{t}+\hat{p}\cdot\nabla_{x}\}\left\{a(t,x)+b(t,x)\cdot\left(1+\frac{\mathcal{I}}{mc^2}\right) p+c(t,x)\left(cp^0\left(1+\frac{\mathcal{I}}{mc^2}\right)-\widetilde{e}(T_0)\right)\right\}\sqrt{F_E^0}\cr
&&=\biggl\{\partial_{t}\big\{a(t,x)-\widetilde{e}(T_0)c(t,x)\big\}+c\sum_{i=1}^{3} \partial_{x_{i}}\big\{a(t,x)-\widetilde{e}(T_0)c(t,x)\big\}\frac{ p^{i}}{p^{0}}
+\sum_{i=1}^{3} \left(\partial_{t}b_{i}(t,x)
+c^2\partial_{x_{i}} c(t,x)\right)
\cr
&&\times \left(1+\frac{\mathcal{I}}{mc^2}\right)p^i+c\sum_{j=1}^{3}\sum_{i=1}^{3}\partial_{x_{i}}b_{j}(t,x)\left(1+\frac{\mathcal{I}}{mc^2}\right)\frac{ p^{i}p^{j}}{p^{0}}  +c\partial_{t}c(t,x)\left(1+\frac{\mathcal{I}}{mc^2}\right)p^{0}\biggl\}\sqrt{F_E^0}\cr
&&:= \partial_{t}\tilde{a}(t,x) e_{a_0}+c\sum_{i=1}^{3} \partial_{x_{i}}\tilde{a}(t,x) e_{a_i}+\sum_{i=1}^{3} (\partial_{t}b_{i}(t,x)+c\partial_{x_{i}}\tilde{c}(t,x))e_{bc_i}
+c\sum_{j=1}^{3}\sum_{i=1}^{3} \partial_{x_{i}}b_{j}(t,x) e_{ij}\cr
&&+ \partial_{t}\tilde{c}(t,x) e_c
\end{eqnarray*}
where 
$$\tilde{a}(t,x):=a(t,x)-\widetilde{e}(T_0)c(t,x),\qquad \tilde{c}(t,x):=cc(t,x)$$ 
and $e_{a_0},\cdots,e_c$ denote 
\begin{align}\label{basis}\begin{split}
\{e_{a_0}, e_{a_i}, e_{bc_i}, e_{ij}, e_{c}\}&=\biggl\{1,\ \frac{p^{i}}{p^{0}} ,\ \Big(1+\frac{\mathcal{I}}{mc^2}\Big)p^{i} ,\ \Big(1+\frac{\mathcal{I}}{mc^2}\Big)\frac{p^{i}p^{j}}{p^{0}} ,\ \Big(1+\frac{\mathcal{I}}{mc^2}\Big)p^{0}\biggl\}\sqrt{F_E^0} 
\end{split}	\end{align}
for $(1\le i,j\le 3)$. In terms of the basis \eqref{basis}, \eqref{macro-micro1} leads to the following relation, that is often called a micro-macro system:
\begin{lemma}\label{M-M}
	Let $l_{a_0}\cdots l_{c}$ and $h_{a_0},\cdot,h_{c}$ denote the inner product of $l\{I-P\}(f)$ and $h(f)$ with the corresponding basis (\ref{basis}). Then we have
	\begin{enumerate}
		\item $\partial_{t}\tilde{a}(t,x)=l_{a_0}+h_{a_0}.$
		\item $\partial_{t}\tilde{c}(t,x)=l_{c}+h_{c}.$
		\item $\partial_{t}b_{i}(t,x)+c\partial_{x_{i}}\tilde{c}(t,x)=l_{bc_i}+h_{bc_i}.$
		\item $c\partial_{x_{i}}\tilde{a}(t,x)=l_{ai}+h_{ai}.$
		\item $c(1-\delta_{ij})\partial_{x_{i}}b_{j}(t,x)+c\partial_{x_{j}}b_{i}(t,x)=l_{ij}+h_{ij}.$
	\end{enumerate}
\end{lemma} 
\begin{proof}
	Since the proof is straightforward, we omit it.
\end{proof}
Using the same line of the argument as in \cite[Theorem 5.4]{Guo whole}, one finds
\begin{equation}\label{coercivitywhole}
\sum_{0<|\alpha|\le N}\left\langle L(\partial^{\alpha}f),\partial^{\alpha}f\right\rangle_{L^2_{x,p,\mathcal{I}}}\le-\delta\sum_{0<|\alpha|\le N}\|\partial^{\alpha}f\|^{2}_{L^2_{x,p,\mathcal{I}}}-C\frac{d}{dt}\int_{\mathbb{R}^3}\left(\nabla_{x}\cdot b(t,x)\right)c(t,x) \,dx.
\end{equation}
To extend the local in time solution to the global one, we take inner product of \eqref{LAW} with $f$ and use Proposition $\ref{pro}$ (2), Lemma \ref{nonlinear1} and Lemma \ref{nonlinear3} to obtain
\begin{align*} 
\frac{d}{dt}\|f\|^{2}_{L^2_{x,p,\mathcal{I}}}+\delta_1\|\{I-P\}f\|^2_{L^2_{x,p,\mathcal{I}}}\le C\sqrt{E(f)(t)}\mathcal{D}(t).
 \end{align*}
Applying $\partial^\alpha_\beta$ to \eqref{LAW}, taking $L^2_{x,p,\mathcal{I}}$ inner product with $\partial^{\alpha}_\beta f$, and employing Lemma \ref{nonlinear1} and \eqref{coercivitywhole}, we have
\begin{equation*} 
\frac{d}{dt}\|\partial^{\alpha}f\|^{2}_{L^2_{x,p,\mathcal{I}}}-C\frac{d}{dt}\int_{\mathbb{R}^3}(\nabla_{x}\cdot b)c\,dx+\delta_2\sum_{0<|\alpha|\le N}\|\partial^{\alpha}f\|^{2}_{L^2_{x,p,\mathcal{I}}}\le C\sqrt{E(f)(t)}\mathcal{D}(t)
\end{equation*}
for $\alpha\neq 0$, $\beta=0$, and
\begin{equation*} 
\frac{d}{dt}\|\partial^{\alpha}_{\beta}f\|^{2}_{L^2_{x,p,\mathcal{I}}}+\delta_3\|\partial^{\alpha}_{\beta}f\|^{2}_{L^2_{x,p,\mathcal{I}}}\le C\sum_{|\beta_1|<|\beta|}\biggl(\|\partial^{\alpha}_{\beta_1}f\|^{2}_{L^2_{x,p,\mathcal{I}}}+\|\nabla_x\partial^{\alpha}_{\beta_1}f\|^{2}_{L^2_{x,p,\mathcal{I}}}\biggl) +C\sqrt{E(f)(t)}\mathcal{D}(t)
\end{equation*}
for $\alpha,\beta\neq0$. Combining above estimates, we obtain the following energy estimate \cite{Guo whole}:
 \begin{align*}
 &\frac{d}{dt}\biggl\{C_{1}\|f\|^{2}_{L^2_{x,p,\mathcal{I}}}+\sum_{0<|\alpha|+|\beta|\le N}C_{|\beta|}\|\partial^{\alpha}_{\beta}f\|^{2}_{L^2_{x,p,\mathcal{I}}}-C_{2}\int_{\mathbb{R}^3}(\nabla_{x}\cdot b(t,x))c(t,x)\,dx\biggl\}+\delta_N\mathcal{D}(t)\cr 
 &\le C\sqrt{E(f)(t)} \mathcal{D}(t)
 \end{align*}
 for some positive constants $C_1$, $C_{|\beta|}$, $C_2$ and $\delta_{N}$.
Then, the standard continuity argument \cite{Guo whole} gives the global in time existence of solutions satisfying
	$$
E_N(f)(t)+\int_0^t \mathcal{D}_N(f)(s) ds\le CE_N(f_0).
$$

%

\subsection{Proof of the asymptotic behaviors (Theorem \ref{main3} (1)--(3))} 
We start with the derivation of the local conservation laws for the linearized relativistic BGK model \eqref{LAW}.
 \begin{lemma}\label{balance} The following relations hold
 	\begin{align*}
 		& \partial_t a(t,x)+k_BT_0\nabla_x \cdot b(t,x)=\Big\langle -\hat{p}\cdot\nabla_x\left\{I-P\right\}(f)+\frac{1}{\tau}\Gamma(f),\sqrt{F_E^0}  \Big\rangle_{L^2_{p,\mathcal{I}}},  \cr
 	& \partial_t b(t,x)+ \frac{k_BT_0}{b_0m }\nabla_x\biggl\{a(t,x)+k_BT_0c(t,x)\biggl\} \cr
 	&=\frac{1}{b_0m}\biggl\langle -\hat{p}\cdot\nabla_x\left\{I-P\right\}(f)+\frac{1}{\tau}\Gamma(f),\biggl(1+\frac{\mathcal{I}}{mc^2}\biggl)p\sqrt{F_E^0}  \biggl\rangle_{L^2_{p,\mathcal{I}}},\cr
 	&\partial_t c(t,x) +\frac{ k_B }{ \{\widetilde{e}\}^\prime(T_0)}\nabla_x\cdot b(t,x) \cr
 	&=\frac{1}{k_BT_0^2\{\widetilde{e}\}^\prime(T_0)} \biggl\langle -\hat{p}\cdot\nabla_x\left\{I-P\right\}(f)+\frac{1}{\tau}\Gamma(f),\biggl\{cp^0\biggl(1+\frac{\mathcal{I}}{mc^2}\biggl)-\widetilde{e}(T_0)\biggl\}\sqrt{F_E^0}   \biggl\rangle_{L^2_{p,\mathcal{I}}}.
 	\end{align*}

 \end{lemma}
 \begin{proof}
We rewrite \eqref{LAW} as
\begin{equation}\label{LAW 2}
	\partial_t f+\hat{p}\cdot\nabla_x P(f)=-\hat{p}\cdot\nabla_x\{I-P\}(f)+\frac{1}{\tau}\left(L(f)+ \Gamma(f) \right)
\end{equation}
	 Multiplying \eqref{LAW 2} by
$$
\sqrt{F_E^0}\phi(\mathcal{I}),\quad\frac{1}{b_0m}\left(1+\frac{\mathcal{I}}{mc^2}\right)p\sqrt{F_E^0}\phi(\mathcal{I}),\quad \frac{1}{k_BT_0^2\{\widetilde{e}\}^\prime(T_0)}\left\{cp^0\left(1+\frac{\mathcal{I}}{mc^2}\right)-\widetilde{e}(T_0)\right\}\sqrt{F_E^0}\phi(\mathcal{I}),
$$
and integrating over $p,\mathcal{I}\in \mathbb{R}^3\times \mathbb{R}^+$, one finds 
\begin{eqnarray*}
&&\partial_t a(t,x)+\biggl\langle \hat{p}\cdot\nabla_x P(f),\sqrt{F_E^0}\biggl\rangle_{L^2_{p,\mathcal{I}}}=\biggl\langle -\hat{p}\cdot\nabla_x\left\{I-P\right\}(f)+\frac{1}{\tau}\Gamma(f),\sqrt{F_E^0}  \biggl\rangle_{L^2_{p,\mathcal{I}}}  ,\cr
&&\partial_t b(t,x)+ \frac{1}{b_0m}\left\langle \hat{p}\cdot\nabla_x P(f),\left(1+\frac{\mathcal{I}}{mc^2}\right)p\sqrt{F_E^0}   \right\rangle_{L^2_{p,\mathcal{I}}}\cr
&&=\frac{1}{b_0m}\left\langle -\hat{p}\cdot\nabla_x\left\{I-P\right\}(f)+\frac{1}{\tau}\Gamma(f),\left(1+\frac{\mathcal{I}}{mc^2}\right)p\sqrt{F_E^0}  \right\rangle_{L^2_{p,\mathcal{I}}}  ,\cr
&&\partial_t c(t,x) +\frac{1}{k_BT_0^2\{\widetilde{e}\}^\prime(T_0)}\biggl\langle \hat{p}\cdot\nabla_x P(f),\left\{cp^0\left(1+\frac{\mathcal{I}}{mc^2}\right)-\widetilde{e}(T_0)\right\}\sqrt{F_E^0}\biggl\rangle_{L^2_{p,\mathcal{I}}}\cr
&& =\frac{1}{k_BT_0^2\{\widetilde{e}\}^\prime(T_0)} \biggl\langle -\hat{p}\cdot\nabla_x\left\{I-P\right\}(f)+\frac{1}{\tau}\Gamma(f),\left\{cp^0\left(1+\frac{\mathcal{I}}{mc^2}\right)-\widetilde{e}(T_0)\right\}\sqrt{F_E^0}   \biggl\rangle_{L^2_{p,\mathcal{I}}}.
\end{eqnarray*}
Claim that
 \begin{enumerate}
	\item $\displaystyle \left\langle \hat{p}\cdot\nabla_x P(f),\sqrt{F_E^0}\right\rangle_{L^2_{p,\mathcal{I}}}=k_BT_0\nabla_x \cdot b(t,x),$
	\item $\displaystyle \left\langle \hat{p}\cdot\nabla_x P(f),\left(1+\frac{\mathcal{I}}{mc^2}\right)p\sqrt{F_E^0}   \right\rangle_{L^2_{p,\mathcal{I}}}=k_BT_0\nabla_x\left\{a(t,x)+k_BT_0c(t,x)\right\},$
	\item $\displaystyle \left\langle \hat{p}\cdot\nabla_x P(f),\left\{cp^0\left(1+\frac{\mathcal{I}}{mc^2}\right)-\widetilde{e}(T_0)\right\}\sqrt{F_E^0}\right\rangle_{L^2_{p,\mathcal{I}}}= (k_BT_0)^2\nabla_x\cdot b(t,x),$
\end{enumerate}
which comples the proof.\noindent\newline
$\bullet$ Proof of (1): Observe from \eqref{P(f)} that
 	\begin{align}\label{local 1}\begin{split}
 		&\left\langle \hat{p}\cdot\nabla_x P(f),\sqrt{F_E^0}\right\rangle_{L^2_{p,\mathcal{I}}}\cr
 		&=\sum_{i=1}^3 \int_{\mathbb{R}^3}\int_0^\infty \frac{cp^i}{p^0}\partial_{x^i}P(f) \sqrt{F_E^0}\phi(\mathcal{I})\, d\mathcal{I}dp\cr
 		&= c\sum_{i=1}^3 \int_{\mathbb{R}^3}\int_0^\infty \frac{p^i}{p^0}\partial_{x^i}\biggl\{ a(t,x)+b(t,x)\cdot\biggl(1+\frac{\mathcal{I}}{mc^2}\biggl) p+c(t,x)\biggl(cp^0\biggl(1+\frac{\mathcal{I}}{mc^2}\biggl)-\widetilde{e}(T_0)\biggl)\biggl\} \cr
 		&\times F_E^0 \phi(\mathcal{I})\, d\mathcal{I}dp.
\end{split} 	\end{align}
By the spherical symmetry of $F_E^0$, \eqref{local 1} becomes
  	\begin{equation*}
 	\left\langle \hat{p}\cdot\nabla_x P(f),\sqrt{F_E^0}\right\rangle_{L^2_{p,\mathcal{I}}} = c\sum_{i=1}^3 \partial_{x^i}b_i(t,x)\int_{\mathbb{R}^3}\int_0^\infty \frac{(p^i)^2}{p^0}F_E^0\left(1+\frac{\mathcal{I}}{mc^2}\right)\phi(\mathcal{I})\, d\mathcal{I}dp
 \end{equation*} 
which, combined with Lemma \ref{computation F_E^0} (1) gives the proof of (1).
 	\noindent\newline
 $\bullet$	Proof of (2): In the same way as the proof of (2), one finds 
 	\begin{eqnarray*}
 			&&\left\langle \hat{p}\cdot\nabla_x P(f),\left(1+\frac{\mathcal{I}}{mc^2}\right)p^j\sqrt{F_E^0}\right\rangle_{L^2_{p,\mathcal{I}}}\cr
 			&&= c\int_{\mathbb{R}^3}\int_0^\infty \frac{ (p^j)^2}{p^0}\partial_{x^j}\left\{a(t,x)+c(t,x)\left(cp^0\left(1+\frac{\mathcal{I}}{mc^2}\right)-\widetilde{e}(T_0)\right)\right\}F_E^0 \left(1+\frac{\mathcal{I}}{mc^2}\right) \phi(\mathcal{I})\, d\mathcal{I}dp\cr
 			&&= c\partial_{x^j}\left\{a(t,x)-\widetilde{e}(T_0)c(t,x)\right\}\int_{\mathbb{R}^3}\int_0^\infty \frac{ (p^j)^2}{p^0}F_E^0\left(1+\frac{\mathcal{I}}{mc^2}\right)  \phi(\mathcal{I})\, d\mathcal{I}dp\cr
 			&&+ c^2\partial_{x^j}c(t,x)\int_{\mathbb{R}^3}\int_0^\infty  (p^j)^2F_E^0 \left(1+\frac{\mathcal{I}}{mc^2}\right)^2  \phi(\mathcal{I})\, d\mathcal{I}dp,
 	\end{eqnarray*}
for $j=1,2,3$. Using the  Lemma \ref{computation F_E^0} (2) and (5), we then have
 	\begin{align*}
\left\langle \hat{p}\cdot\nabla_x P(f),\left(1+\frac{\mathcal{I}}{mc^2}\right)p^j\sqrt{F_E^0}\right\rangle_{L^2_{p,\mathcal{I}}}&= k_BT_0\partial_{x^j}\left\{a(t,x)-\widetilde{e}(T_0)c(t,x)\right\}+ b_0mc^2\partial_{x^j}c(t,x)\cr
&=k_BT_0\partial_{x^j}\left\{a(t,x)+k_BT_0 c(t,x) \right\}
\end{align*}
which completes the proof of (2).
 	\noindent\newline
 	$\bullet$ Proof of (3): In the same manner, we have 
 	\begin{eqnarray*}
 	&&\left\langle \hat{p}\cdot\nabla_x P(f),\left\{cp^0\left(1+\frac{\mathcal{I}}{mc^2}\right)-\widetilde{e}(T_0)\right\}\sqrt{F_E^0}\right\rangle_{L^2_{p,\mathcal{I}}}\cr
 	&&=c^2\sum_{i=1}^3 \partial_{x^i}b_i(x,t)\int_{\mathbb{R}^3}\int_0^\infty (p^i)^2 F_E^0\left(1+\frac{\mathcal{I}}{mc^2}\right)^2 \phi(\mathcal{I})\, d\mathcal{I}dp\cr
 	&&-c\widetilde{e}(T_0)\sum_{i=1}^3 \partial_{x^i}b_i(x,t)\int_{\mathbb{R}^3}\int_0^\infty \frac{(p^i)^2}{p^0} F_E^0\left(1+\frac{\mathcal{I}}{mc^2}\right) \phi(\mathcal{I})\, d\mathcal{I}dp\cr	
 	&&=(k_BT_0)^2\nabla_x\cdot b(x,t).
 	\end{eqnarray*}
 \end{proof}	  
Now we prove the key estimate of this subsection which is a relativistic generalization of Lemma 6.1 of \cite{Guo 1}:
\begin{lemma}\label{I-P}
	Let	$N\ge 3$. Then for $k=0,\cdots,N-1$, we have
	$$
	\frac{d}{dt}G_k+\|\nabla^{k+1}Pf\|^2_{L^2_{x,p,\mathcal{I}}}\le C \biggl(\left\| \nabla^k\{I-P\}f\right\|^2_{L^2_{x,p,\mathcal{I}}}+\sum_{|{\alpha_1}|\le k}\left\|\|\nabla^{|\alpha_1|} f\|_{L^2_{p,\mathcal{I}}}\|\nabla^{k-|\alpha_1|} f\|_{L^2_{p,\mathcal{I}}}\right\|_{L^2_{x}}^2 \biggl)
	$$
	where $G_k(t)$ denotes
	\begin{align*}
		&\sum_{|\alpha|=k}\int_{\mathbb{R}^3}\langle\{I-P\}\partial^\alpha f,\epsilon_a(p,\mathcal{I}) \rangle_{L^2_{p,\mathcal{I}}}\cdot \nabla_x \partial^\alpha a(t,x)+ \langle \{I-P\}\partial^\alpha f,\epsilon_c(p,\mathcal{I}) \rangle_{L^2_{p,\mathcal{I}}}\cdot \nabla_x\partial^\alpha  c(t,x) \,dx\cr
		&+\sum_{|\alpha|=k}\int_{\mathbb{R}^3}\left\langle \{I-P\}  \partial^\alpha f , \epsilon_b(p,\mathcal{I}) \right\rangle_{L^2_{p,\mathcal{I}}} \nabla_x\cdot \partial^\alpha b(t,x)+ \partial^\alpha b(t,x)\cdot \nabla_x\partial^\alpha\tilde{c}(t,x)\,dx
	\end{align*}
and $\epsilon_a,\epsilon_b$, and $\epsilon_c$ are linear combinations of the basis \eqref{basis}.
\end{lemma}
\begin{proof}

	For $|\alpha|\le N-1$, we have from Lemma $\ref{M-M}_{(5)}$ that
	$$c\Delta \partial^{\alpha}b_{i}(t,x)=\partial^{\alpha}\sum_{j\neq i}\big[-\partial_{i}(l_{jj}+h_{jj})+\partial_{j}(l_{ij}+h_{ij})\big]+\partial_{i}\partial^{\alpha}(l_{ii}+h_{ii}).
	$$
Here $l_{a_0},\cdots, l_{c}$ take the following form:
	\begin{align*}
		\left\langle -\partial_{t}\{I-P\}\partial^{\alpha}f,\epsilon \left(p,\mathcal{I}\right) \right\rangle_{L^2_{p,\mathcal{I}}}-\left\langle \big\{\hat{p}\cdot\nabla_{x}-L\big\}\{I-P\}\partial^{\alpha}f,\epsilon (p,\mathcal{I}) \right\rangle_{L^2_{p,\mathcal{I}}}
	\end{align*}
where $\epsilon(p,\mathcal{I})$ is a suitable linear combination of the basis \eqref{basis}.
We then have
	\begin{align}\label{b1}\begin{split}
		\|\nabla\partial^{\alpha}b\|^{2}_{L^2_x}&\le \int_{\mathbb{R}^3}\left\langle \partial_t\{I-P\}\nabla_x \partial^\alpha f, \epsilon_b(p,\mathcal{I}) \right\rangle_{L^2_{p,\mathcal{I}}} \cdot \partial^\alpha b(t,x)\,dx\cr
		&  +C\sum_{|{\alpha_1}|\le |\alpha|}\left\|\|\partial^{\alpha_1} f\|_{L^2_{p,\mathcal{I}}}\|\partial^{\alpha-{\alpha_1}} f\|_{L^2_{p,\mathcal{I}}}\right\|_{L^2_{x}}\|\nabla_x\partial^{\alpha}b\|_{L^2_x}\cr
		&+\left(\|\{I-P\} \nabla_x\partial^\alpha f \|_{L^2_{x,p,\mathcal{I}}}+\|\{I-P\} \partial^\alpha f \|_{L^2_{x,p,\mathcal{I}}}\right)\|\nabla_x\partial^{\alpha}b\|_{L^2_x}\cr
\end{split}	\end{align}
where we used Lemma \ref{nonlinear1} to obtain
$$
\|\partial^{\alpha} h_{i,j}\|_{L^2_x}=\|\left\langle \partial^{\alpha} \Gamma(f),e_{ij} \right\rangle_{L^2_{p,\mathcal{I}}} \|_{L^2_x}\le C\sum_{|{\alpha_1}|\le |\alpha|}\left\|\|\partial^{\alpha_1} f\|_{L^2_{p,\mathcal{I}}}\|\partial^{\alpha-{\alpha_1}} f\|_{L^2_{p,\mathcal{I}}}\right\|_{L^2_{x}}.
$$
Here the first term on the r.h.s of \eqref{b1} can be written as
	\begin{eqnarray*}
		&&\int_{\mathbb{R}^3}\left\langle \partial_t\{I-P\} \nabla_x \partial^\alpha f , \epsilon_b(p,\mathcal{I}) \right\rangle_{L^2_{p,\mathcal{I}}} \cdot \partial^\alpha b(t,x)\,dx\cr
		&&=\frac{d}{dt} \int_{\mathbb{R}^3}\left\langle \{I-P\} \nabla_x \partial^\alpha f , \epsilon_b(p,\mathcal{I}) \right\rangle_{L^2_{p,\mathcal{I}}} \cdot \partial^\alpha b(t,x)\,dx\cr
		&&-\int_{\mathbb{R}^3}\left\langle \{I-P\} \nabla_x \partial^\alpha f , \epsilon_b(p,\mathcal{I}) \right\rangle_{L^2_{p,\mathcal{I}}} \cdot \partial_t\partial^\alpha b(t,x)\,dx \cr
		&&=-\frac{d}{dt} \int_{\mathbb{R}^3}\left\langle \{I-P\}  \partial^\alpha f , \epsilon_b(p,\mathcal{I}) \right\rangle_{L^2_{p,\mathcal{I}}}  \nabla_x \cdot\partial^\alpha b(t,x)\,dx\cr
		&&-\int_{\mathbb{R}^3}\left\langle \{I-P\} \nabla_x \partial^\alpha f , \epsilon_b(p,\mathcal{I}) \right\rangle_{L^2_{p,\mathcal{I}}} \cdot \partial_t\partial^\alpha b(t,x)\,dx ,
	\end{eqnarray*}
which, combined with Lemma $\ref{balance}_{(2)}$ leads to
\begin{align}\label{b3}\begin{split}
		&\int_{\mathbb{R}^3}\left\langle \partial_t\{I-P\} \nabla_x \partial^\alpha f , \epsilon_b(p,\mathcal{I}) \right\rangle_{L^2_{p,\mathcal{I}}} \cdot \partial^\alpha b(t,x)\,dx\cr
		&\le -\frac{d}{dt} \int_{\mathbb{R}^3}\left\langle \{I-P\}  \partial^\alpha f , \epsilon_b(p,\mathcal{I}) \right\rangle_{L^2_{p,\mathcal{I}}} \nabla_x\cdot \partial^\alpha b(t,x)\,dx +C_\varepsilon\|\{I-P\} \nabla_x\partial^\alpha f \|_{L^2_{x,p,\mathcal{I}}}^2\cr
		&+\varepsilon\|\nabla_x \partial^\alpha a +\nabla_x \partial^\alpha c \|_{L^2_{x }}^2+\varepsilon\sum_{|{\alpha_1}|\le |\alpha|}\left\|\|\partial^{\alpha_1} f\|_{L^2_{p,\mathcal{I}}}\|\partial^{\alpha-{\alpha_1}} f\|_{L^2_{p,\mathcal{I}}}\right\|_{L^2_{x}}^2 .
\end{split}	\end{align}
Go back to \eqref{b1} with \eqref{b3} to get
	\begin{align}\label{bx}\begin{split}
		\|\nabla\partial^{\alpha}b \|^{2}_{L^2_x}&\le -\frac{d}{dt}\int_{\mathbb{R}^3}\left\langle \{I-P\}  \partial^\alpha f , \epsilon_b(p,\mathcal{I}) \right\rangle_{L^2_{p,\mathcal{I}}} \nabla_x\cdot \partial^\alpha b(t,x)\,dx\cr
		&+C\left(\|\{I-P\} \nabla_x\partial^\alpha f \|_{L^2_{x,p,\mathcal{I}}}^2+\|\{I-P\} \partial^\alpha f \|^2_{L^2_{x,p,\mathcal{I}}}\right)\cr
		&+\varepsilon_1\biggl(\|\nabla_x \partial^\alpha a \|_{L^2_{x }}^2+\|\nabla_x \partial^\alpha c \|_{L^2_{x }}^2 +\sum_{|{\alpha_1}|\le |\alpha|}\left\|\|\partial^{\alpha_1} f\|_{L^2_{p,\mathcal{I}}}\|\partial^{\alpha-{\alpha_1}} f\|_{L^2_{p,\mathcal{I}}}\right\|_{L^2_{x}}^2 \biggl). \end{split}\end{align}	
In a similar way, it follows from Lemma $\ref{M-M}_{(4)}$ and Lemma $\ref{balance}_{(1)}$ that
 		\begin{align}\label{ax}\begin{split}
		\|\nabla_x\partial^{\alpha}a \|^{2}_{L^2_x}&\le -\frac{d}{dt}\int_{\mathbb{R}^3}\left\langle \{I-P\}  \partial^\alpha f , \epsilon_a(p,\mathcal{I}) \right\rangle_{L^2_{p,\mathcal{I}}} \cdot\nabla_x  \partial^\alpha a(t,x)\,dx\cr
		&+C\left(\|\{I-P\} \nabla_x\partial^\alpha f \|_{L^2_{x,p,\mathcal{I}}}^2+\|\{I-P\} \partial^\alpha f \|^2_{L^2_{x,p,\mathcal{I}}}\right)\cr
		&+\varepsilon_2\biggl(\|\nabla_x  \partial^\alpha b \|_{L^2_{x }}^2+\sum_{|{\alpha_1}|\le |\alpha|}\left\|\|\partial^{\alpha_1} f\|_{L^2_{p,\mathcal{I}}}\|\partial^{\alpha-{\alpha_1}} f\|_{L^2_{p,\mathcal{I}}}\right\|_{L^2_{x}}^2 \biggl). 
\end{split}	\end{align}
Also, we have from Lemma $\ref{M-M}_{(3)}$ that
\begin{align}\begin{split}\label{c1}
\|\nabla\partial^{\alpha}\tilde{c} \|_{L^2_x}^{2}
&\le -\int_{\mathbb{R}^3}\partial_{t}\partial^{\alpha}b(t,x)\cdot\nabla_x\partial^\alpha\tilde{c}(t,x)+\langle \partial_t\{I-P\}\nabla_x\partial^\alpha f,\epsilon_c(p,\mathcal{I})\rangle _{L^2_{p,\mathcal{I}}}\partial^\alpha\tilde{c}(t,x)\,dx\cr
&+C\left(\|\{I-P\}\nabla_x\partial^\alpha f\|_{L^2_{x,p,\mathcal{I}}}^2+\|\{I-P\}\partial^\alpha f\|_{L^2_{x,p,\mathcal{I}}}^2\right)\cr
&+C\sum_{|{\alpha_1}|\le |\alpha|}\left\|\|\partial^{\alpha_1} f\|_{L^2_{p,\mathcal{I}}}\|\partial^{\alpha-{\alpha_1}} f\|_{L^2_{p,\mathcal{I}}}\right\|_{L^2_{x}}^2.
\end{split}\end{align}
Using integration by parts and Lemma $\ref{balance}_{(3)}$, the first term on the r.h.s of \eqref{c1} can be estimated as
\begin{align*} 
&-\int_{\mathbb{R}^3} \partial_{t}\partial^{\alpha}b(t,x)\cdot\nabla_x \partial^\alpha\tilde{c}(t,x)+\langle \partial_t\{I-P\}\nabla_x\partial^\alpha f,\epsilon_c(p,\mathcal{I})\rangle _{L^2_{p,\mathcal{I}}}\partial^\alpha\tilde{c}(t,x)\,dx\cr
&=-\frac{d}{dt}\biggl\{ \int_{\mathbb{R}^3} \partial^{\alpha}b(t,x)\cdot\nabla_x \partial^\alpha\tilde{c}(t,x)-\langle \{I-P\} \nabla_x\partial^\alpha f,\epsilon_c(p,\mathcal{I})\rangle _{L^2_{p,\mathcal{I}}} \partial^\alpha\tilde{c}(t,x)\,dx\biggl\} \cr
&-\int_{\mathbb{R}^3}\nabla_x\cdot\partial^{\alpha}b(t,x) \partial_{t}\partial^\alpha\tilde{c}(t,x)-\langle \{I-P\}\nabla_x\partial^\alpha f,\epsilon_c(p,\mathcal{I})\rangle _{L^2_{p,\mathcal{I}}}\partial_t\partial^\alpha\tilde{c}(t,x)\,dx \cr
&\le -\frac{d}{dt}\biggl\{ \int_{\mathbb{R}^3} \partial^{\alpha}b(t,x)\cdot\nabla_x \partial^\alpha\tilde{c}(t,x)+\langle \{I-P\} \partial^\alpha f,\epsilon_c(p,\mathcal{I})\rangle _{L^2_{p,\mathcal{I}}}\nabla_x\partial^\alpha\tilde{c}(t,x)\,dx\biggl\} \cr
&+C^\prime\|\nabla_x\partial^\alpha b\|^2_{L^2_x} +C\biggl(\|\{I-P\}\nabla_x\partial^\alpha f\|^2_{L^2_{x,p,\mathcal{I}}}+\sum_{|{\alpha_1}|\le |\alpha|}\left\|\|\partial^{\alpha_1} f\|_{L^2_{p,\mathcal{I}}}\|\partial^{\alpha-{\alpha_1}} f\|_{L^2_{p,\mathcal{I}}}\right\|_{L^2_{x}}^2\biggl).
\end{align*}
which, together with \eqref{c1} gives
\begin{align}\begin{split}\label{c3}
&\|\nabla\partial^{\alpha}\tilde{c}\|_{L^2_x}^{2}\cr
&\le -\frac{d}{dt}\biggl\{ \int_{\mathbb{R}^3} \partial^{\alpha}b(t,x)\cdot\nabla_x \partial^\alpha\tilde{c}(t,x)+\langle \{I-P\} \partial^\alpha f,\epsilon_n(p,\mathcal{I})\rangle _{L^2_{p,\mathcal{I}}}\nabla_x\partial^\alpha\tilde{c}(t,x)\,dx\biggl\}+C^\prime\|\nabla_x\partial^\alpha b\|^2_{L^2_x} \cr
&+C\biggl(\|\{I-P\}\nabla_x\partial^\alpha f\|_{L^2_{x,p,\mathcal{I}}}^2+\|\{I-P\}\partial^\alpha f\|_{L^2_{x,p,\mathcal{I}}}^2+\sum_{|{\alpha_1}|\le |\alpha|}\left\|\|\partial^{\alpha_1} f\|_{L^2_{p,\mathcal{I}}}\|\partial^{\alpha-{\alpha_1}} f\|_{L^2_{p,\mathcal{I}}}\right\|_{L^2_{x}}^2\biggl).
\end{split}\end{align}
For sufficiently small $\varepsilon_1$ and $\varepsilon_2$ satisfying $C^\prime  \varepsilon_1 \ll1$, combining \eqref{bx}, \eqref{ax} and \eqref{c3} gives the desired result for spatial derivative.  We now employ the case of temporal derivative $d/dt$. Recall from Lemma $\ref{balance}_{(1)}$ that
	\begin{align*}
\partial_t\partial^\alpha a(t,x)&=-k_BT_0\nabla_x \cdot \partial^\alpha b(t,x)+\left\langle -\hat{p}\cdot\nabla_x\left\{I-P\right\}(\partial^\alpha f)+\partial^\alpha h,\sqrt{F_E^0}  \right\rangle_{L^2_{p,\mathcal{I}}}\cr
&\le -k_BT_0\nabla_x \cdot \partial^\alpha b(t,x)+C\|\nabla_x\left\{I-P\right\}\partial^\alpha f\|_{L^2_{p,\mathcal{I}}}\cr
&+C\sum_{|{\alpha_1}|\le |\alpha|}\left\|\|\partial^{\alpha_1} f\|_{L^2_{p,\mathcal{I}}}\|\partial^{\alpha-{\alpha_1}} f\|_{L^2_{p,\mathcal{I}}}\right\|_{L^2_{x}}.
\end{align*}
Taking inner product with $\partial_t\partial^\alpha a(t,x)$, the above estimate leads to
\begin{align*}
\|\partial_t\partial^\alpha a\|^2_{L^2_x}&\le C_\varepsilon\biggl(\|\nabla_x \partial^\alpha b\|^2_{L^2_x}+ \|\nabla_x\left\{I-P\right\}\partial^\alpha f\|_{L^2_{x,p,\mathcal{I}}}^2+C\sum_{|{\alpha_1}|\le |\alpha|}\left\|\|\partial^{\alpha_1} f\|_{L^2_{p,\mathcal{I}}}\|\partial^{\alpha-{\alpha_1}} f\|_{L^2_{p,\mathcal{I}}}\right\|_{L^2_{x}}^2\biggl)\cr
&+\varepsilon \|\partial_t\partial^\alpha a\|^2_{L^2_x}, 
\end{align*}
yielding
\begin{equation}\label{t-a}
\|\partial_t\partial^\alpha a\|^2_{L^2_x}
\le C\biggl(\|\nabla_x \partial^\alpha b\|^2_{L^2_x}+ \|\nabla_x\left\{I-P\right\}\partial^\alpha f\|^2_{L^2_{x,p,\mathcal{I}}}+C\sum_{|{\alpha_1}|\le |\alpha|}\left\|\|\partial^{\alpha_1} f\|_{L^2_{p,\mathcal{I}}}\|\partial^{\alpha-{\alpha_1}} f\|_{L^2_{p,\mathcal{I}}}\right\|_{L^2_{x}}^2\biggl)
\end{equation}
for sufficiently $\varepsilon$. In the same manner, one can have from Lemma $\ref{balance}_{(2),(3)}$ that
\begin{align}\label{t-b}\begin{split}
\|\partial_t\partial^\alpha b\|^2_{L^2_x}
&\le C\left(\|\nabla_x \partial^\alpha a\|^2_{L^2_x}+\|\nabla_x \partial^\alpha c\|^2_{L^2_x}+ \|\nabla_x\left\{I-P\right\}\partial^\alpha f\|_{L^2_{x,p,\mathcal{I}}}^2\right),\cr
&+C\sum_{|{\alpha_1}|\le |\alpha|}\left\|\|\partial^{\alpha_1} f\|_{L^2_{p,\mathcal{I}}}\|\partial^{\alpha-{\alpha_1}} f\|_{L^2_{p,\mathcal{I}}}\right\|_{L^2_{x}}^2
\end{split}\end{align}
and
\begin{align}\label{t-b-2}\begin{split}
\|\partial_t\partial^\alpha c\|^2_{L^2_x}
&\le C\left(\|\nabla_x \partial^\alpha b\|^2_{L^2_x}+ \|\nabla_x\left\{I-P\right\}\partial^\alpha f\|_{L^2_{x,p,\mathcal{I}}}^2\right)\cr
&+C\sum_{|{\alpha_1}|\le |\alpha|}\left\|\|\partial^{\alpha_1} f\|_{L^2_{p,\mathcal{I}}}\|\partial^{\alpha-{\alpha_1}} f\|_{L^2_{p,\mathcal{I}}}\right\|_{L^2_{x}}^2.
\end{split}\end{align}
Combining \eqref{t-a}--\eqref{t-b-2} and results for the spatial derivative  completes the proof.

\end{proof}


Now we are ready to prove the rest of Theorem \ref{main3}. Since the argument is similar as in \cite[Section 4]{Guo-Wang}, we only present a sketch of the proof of Theorem \ref{main3} (2)--(3) for brevity.\noindent\newline
$\bullet$ Proof of Theorem \ref{main3} (2): Let $0\le \ell\le N-1$. Applying $\nabla^k$ to \eqref{LAW}, taking the $L^2_{x,p,\mathcal{I}}$ inner product with $\nabla^k f$, and using Proposition \ref{pro}, Lemmas \ref{nonlinear1} and \ref{nonlinear3}, we have
\begin{align}\label{nablak}\begin{split}
		&\frac{1}{2}\frac{d}{dt}\sum_{\ell\le k\le N-1} \|\nabla^k f\|^2_{L^2_{x,p,\mathcal{I}}}+ \sum_{\ell\le k\le N-1}\|\nabla^k \{I-P\}f\|^2_{L^2_{x,p,\mathcal{I}}}\cr
		&\le C\sum_{\ell\le k\le N-1}\sum_{|{\alpha_1}|\le k}\left\|\|\nabla^{|\alpha_1|} f\|_{L^2_{p,\mathcal{I}}}\|\nabla^{k-|\alpha_1|} f\|_{L^2_{p,\mathcal{I}}}\right\|_{L^2_x}^2.
\end{split}\end{align}
Using the Sobolev type inequalities and Minkowski's inequality, the right-hand side of \eqref{nablak}  can be bounded from above as
\begin{align}\label{rhs1}\begin{split}
		\sum_{|{\alpha_1}|\le k}\left\|\|\nabla^{|\alpha_1|} f\|_{L^2_{p,\mathcal{I}}}\|\nabla^{k-|\alpha_1|} f\|_{L^2_{p,\mathcal{I}}}\right\|_{L^2_x}^2\le C\delta^2\biggl(\|\nabla^{k+1}f\|^2_{L^2_{x,p,\mathcal{I}}}+ \|\nabla^k \{I-P\}f\|^2_{L^2_{x,p,\mathcal{I}}}\biggl)
\end{split}\end{align}
for $k=0,\cdots,N-1$, and
\begin{align}\label{rhs2}\begin{split}
		\sum_{|{\alpha_1}|\le k}\left\|\|\nabla^{|\alpha_1|} f\|_{L^2_{p,\mathcal{I}}}\|\nabla^{k-|\alpha_1|} f\|_{L^2_{p,\mathcal{I}}}\right\|_{L^2_x}^2\le		C\delta^2\|\nabla^{N}f\|^2_{L^2_{x,p,\mathcal{I}}} 
\end{split}\end{align}
for $k=N$ respectively.  Combining \eqref{nablak}--\eqref{rhs2}, one can see that 
\begin{align}\label{nablaN}\begin{split}
		&\frac{d}{dt}\sum_{\ell\le k\le N} \|\nabla^k f\|^2_{L^2_{x,p,\mathcal{I}}}+C \sum_{\ell\le k\le N}\|\nabla^k \{I-P\}f\|^2_{L^2_{x,p,\mathcal{I}}}\cr
		&\le C\delta^2\Biggl(\sum_{\ell+1\le k\le N}\|\nabla^{k}f\|^2_{L^2_{x,p,\mathcal{I}}}+\sum_{\ell\le k\le N}\|\nabla^k \{I-P\}f\|^2_{L^2_{x,p,\mathcal{I}}}\Biggl).
\end{split}\end{align}
On the other hand, it follows from  Lemma \ref{I-P}, Sobolev interpolation and Minkowski's inequality that
\begin{align*} 
	& \frac{d}{dt}\eta\sum_{\ell\le k\le N-1}G_k+\eta\sum_{\ell+1\le k\le N} \|\nabla^{k}Pf\|^2_{L^2_{x,p,\mathcal{I}}}\cr&\le C\eta \sum_{\ell\le k\le N-1} \left\| \nabla^k\{I-P\}f\right\|^2_{L^2_{x,p,\mathcal{I}}}  +C\eta \sum_{\ell\le k\le N-1}\sum_{|{\alpha_1}|\le k}\left\|\|\nabla^{|\alpha_1|} f\|_{L^2_{p,\mathcal{I}}}\|\nabla^{k-|\alpha_1|} f\|_{L^2_{p,\mathcal{I}}}\right\|_{L^2_{x}}^2\cr
	&\le C\eta\sum_{\ell\le k\le N}  \left\| \nabla^k\{I-P\}f\right\|^2_{L^2_{x,p,\mathcal{I}}}.
\end{align*}
This, together with \eqref{nablaN} yields that for sufficiently small $\eta$ and $\delta$,
\begin{align}\label{mathcal energy} 
	&\frac{d}{dt}\mathcal{E}_\ell(t)+\|\nabla^\ell \{I-P\}f \|^2_{L^2_{x,p,\mathcal{I}}}+\sum_{\ell+1\le k\le N}\|\nabla^{k}f\|^2_{L^2_{x,p,\mathcal{I}}}\le  0,
\end{align}
where $\mathcal{E}_\ell(t)$ denotes
$$
\mathcal{E}_\ell(t)=\sum_{\ell\le k\le N}\|\nabla^k f \|^2_{L^2_{x,p,\mathcal{I}}}+\eta\sum_{\ell\le k\le N-1}G_k.
$$
Using the following Sobolev interpolation (for details, please see \cite[Lemma A.4]{Guo-Wang}):
\begin{equation*}\label{interpolation}
\|\nabla^\ell f\|_{L^2}\le C\|\nabla^{\ell+1}f \|^\theta_{L^2}\|\Lambda^{-s}f \|^\theta_{L^2},\qquad\text{where}\qquad \theta=\frac{1}{\ell+1+s},\ s,\ell \ge0,
\end{equation*}
we have from \eqref{mathcal energy} and Theorem \ref{main3} (1) that
 \begin{equation*}
 	\frac{d}{dt}\mathcal{E}_\ell(t)+ C_0\biggl(\sum_{\ell+1\le k\le N}\|\nabla^{k-1}f\|^2_{L^2_{x,p,\mathcal{I}}}\biggl)^{1+\frac{1}{\ell+s}} \le  0.
 \end{equation*}
Since $\mathcal{E}_\ell(t)$ is equivalent to $\sum_{\ell\le k\le N}\|\nabla^k f \|^2_{L^2_{x,p,\mathcal{I}}}$ for sufficiently small $\eta$, this gives the desired result.\noindent\newline
$\bullet$ Proof of Theorem \ref{main3} (3): Applying $\{I-P\}$ to \eqref{LAW}, and using Proposition \eqref{pro} and Lemma \ref{nonlinear3}, one finds
\begin{equation*}
\partial_t \{I-P\}f+\hat{p}\cdot\nabla_x \{I-P\}f+L\{I-P\}f=\Gamma(f)-\hat{p}\cdot\nabla_x Pf+P(\hat{p}\cdot\nabla_xf).
\end{equation*}
Applying $\nabla^k$ $(k=0,\cdots,N-2)$ and taking the $L^2_{x,p,\mathcal{I}}$ inner product with $\nabla^k\{I-P\}f$, this leads to
\begin{align*}
	&\frac{1}{2}\frac{d}{dt}\|\nabla^k \{I-P\}f\|^2_{L^2_{x,p,\mathcal{I}}}+\|\nabla^k \{I-P\}f\|^2_{L^2_{x,p,\mathcal{I}}}\cr
	&\le \left\langle \nabla^k\Gamma(f),\nabla^k\{I-P\}f\right\rangle_{L^2_{x,p,\mathcal{I}}}-\left\langle \hat{p}\cdot\nabla_x P\nabla^kf-P(\hat{p}\cdot\nabla_x\nabla^kf),\nabla^k\{I-P\}f\right\rangle_{L^2_{x,p,\mathcal{I}}},
\end{align*}
which, combined with Lemma \ref{nonlinear1} gives that for small $\varepsilon$,
\begin{align}\label{I-P 2}\begin{split}
	&\frac{1}{2}\frac{d}{dt}\|\nabla^k \{I-P\}f\|^2_{L^2_{x,p,\mathcal{I}}}+\|\nabla^k \{I-P\}f\|^2_{L^2_{x,p,\mathcal{I}}}\cr
	&\le C_\varepsilon \Biggl(\sum_{|{\alpha_1}|\le k}\left\|\|\nabla^{|\alpha_1|} f\|_{L^2_{p,\mathcal{I}}}\|\nabla^{k-|\alpha_1|} f\|_{L^2_{p,\mathcal{I}}}\right\|_{L^2_x}^2+\|\nabla^{k+1}f\|^2_{L^2_{x,p,\mathcal{I}}}\Biggl)+\varepsilon\|\nabla^k \{I-P\}f\|^2_{L^2_{x,p,\mathcal{I}}}. 
\end{split}\end{align}
On the other hand, note from \cite[Lemma 4.4]{Guo-Wang} that using the Sobolev type inequalities and Minkowski's inequality, one can have
\begin{align*} 
		\sum_{|{\alpha_1}|\le k}\left\|\|\nabla^{|\alpha_1|} f\|_{L^2_{p,\mathcal{I}}}\|\nabla^{k-|\alpha_1|} f\|_{L^2_{p,\mathcal{I}}}\right\|_{L^2_x}^2\le C\delta^2\biggl(\|\nabla^{k+1}f\|^2_{L^2_{x,p,\mathcal{I}}}+ \|\nabla^k \{I-P\}f\|^2_{L^2_{x,p,\mathcal{I}}}\biggl).
 \end{align*}
This, together with \eqref{I-P 2} gives that for sufficiently small $\varepsilon$ and $\delta$,
\begin{align}\label{target}
	& \frac{d}{dt}\|\nabla^k \{I-P\}f\|^2_{L^2_{x,p,\mathcal{I}}}+\|\nabla^k \{I-P\}f\|^2_{L^2_{x,p,\mathcal{I}}}\le C \|\nabla^{k+1}f\|^2_{L^2_{x,p,\mathcal{I}}}.
\end{align}
For $k=1,\cdots,N-2$, applying the Gronwall inequality to \eqref{target} and using Theorem \ref{main3} (2) with $\ell=k+1$, one finds
\begin{align*}
\|\nabla^k \{I-P\}f\|^2_{L^2_{x,p,\mathcal{I}}}&\le e^{-t}\|\nabla^k \{I-P\}f_0\|^2_{L^2_{x,p,\mathcal{I}}}+C\int_0^t e^{-(t-s)}\|\nabla^{k+1}f(s)\|^2_{L^2_{x,p,\mathcal{I}}}\,ds\cr
&\le C_0(1+t)^{-(k+1+s)}
\end{align*}
which, together with the interpolation gives the desired result for $-s<k\le N-2$.

\smallskip

\noindent{\bf Acknowledgement}
Byung-Hoon Hwang was supported by Basic Science Research Program through the National Research Foundation of Korea(NRF) funded by the Ministry of Education(No. NRF-2019R1A6A1A10073079). Seok-Bae Yun was supported by Samsung Science and Technology Foundation under Project Number SSTF-BA1801-02.
\bibliographystyle{amsplain}

\begin{thebibliography}{10}
\bibitem{AW} Anderson, J. L., Witting, H. R.: A relativistic relaxation-time model for the Boltzmann equation. Physica. {\bf{74}} (1974) 466--488.
	\bibitem{A} Andreasson, H.: Regularity of the gain term and strong $L^1$-convergence to equilibrium for the relativistic Boltzmann equation, SIAM J. Math. Anal. {\bf{27}} (1996) 1386--1405.
	\bibitem{BCNS} Bellouquid, A., Calvo, J., Nieto, J., Soler, J.: On the relativistic BGK-Boltzmann model: asymptotics and hydrodynamics. J. Stat. Phys. {\bf 149} (2012) 284--316.
	\bibitem{BNU} Bellouquid, A., Nieto, J., Urrutia, L.: Global existence and asymptotic stability near equilibrium for the relativistic BGK model. Nonlinear Anal. {\bf 114} (2015) 87--104.
	\bibitem{BGK} Bhatnagar, P. L., Gross, E. P. and Krook, M. L.: A model for collision processes in gases. I. Small amplitude processes in charged
	and neutral one-component systems, Phys. Rev. {\bf 94} (1954) 511-525.
	\bibitem{B} Bichteler, K.: On the Cauchy problem of the relativistic Boltzmann equation. Comm. Math. Phys. {\bf 4} (1967) 352--364.
	\bibitem{Cal} Calogero, S.: The Newtonian limit of the relativistic Boltzmann equation. J. Math. Phys. {\bf45} (2004) 4042--4052.
	\bibitem{CJS} Calvo, J., Jabin, P.-E., Soler, J.: Global weak solutions to the relativistic BGK equation. Comm. Partial Differential Equations {\bf{45}}(3) (2020) 191--229.
	\bibitem{PR3} Carrisi, M. C., Pennisi, S., Ruggeri, T.: Integrability properties for relativistic extended thermodynamics of polyatomic gas. Ricerche. mat. {\bf{68}}(1) (2019) 57--73.
	\bibitem{CPR1} Carrisi, M. C., Pennisi, S., Ruggeri T.: Monatomic limit of relativistic extended thermodynamics of polyatomic gas. Continuum Mech. Thermodyn. {\bf{31}}(2) (2019) 401--412.
	\bibitem{CPR2} Carrisi, M. C., Pennisi, S., Ruggeri T.: Production terms in relativistic extended thermodynamics of gas with internal structure via a new BGK model. Ann. Phys. {\bf{405}} (2019) 298--307.
	\bibitem{CK} Cercignani, C., Kremer, G. M.: The Relativistic Boltzmann Equation: Theory and Applications. Birkh\"{a}user, Basel (2002).
	\bibitem{CKT} Chen, Y., Kuang, Y., Tang, H.: Second-order accurate genuine BGK schemes for the ultra-relativistic flow simulations. J. Comput. Phys. {\bf{349}} (2017) 300--327.
	\bibitem{DHMNS2} Denicol, G. S., Heinz, U., Martinez, M., Noronha, J., Strickland, M.: Studying the validity of relativistic hydrodynamics with a new exact solution of the Boltzmann equation. Phys. Rev. D. {\bf{90}} (2014) 125026.
	
	\bibitem{D} Dudy\'{n}ski, M.: On the linearized relativistic Boltzmann equation. II. Existence of hydrodynamics.
	J. Stat. Phys. {\bf 57} (1--2) (1989) 199--245.  	
	\bibitem{Dud3} Dudy\'{n}ski, M., Ekiel-Je\.{z}ewska, M. L.: Global existence proof for relativistic Boltzmann equation, J. Stat. Phys. {\bf 66} (3) (1992) 991--1001.
	\bibitem{DE} Dudy\'{n}ski, M., Ekiel-Je\.{z}ewska, M. L.: On the linearized relativistic Boltzmann equation. I. Existence of
	solutions. Comm. Math. Phys. {\bf 115} (4) (1985) 607--629.
	\bibitem{E} Eckart, C., Phys. Rev. {\bf 58} (1940) 919.
	\bibitem{FRS} Florkowski, W., Ryblewski, R., Strickland, M.: Anisotropic hydrodynamics for rapidly expanding systems. Nucl. Phys. A {\bf{916}} (2013) 249--259.
	\bibitem{FRS2} Florkowski, W., Ryblewski, R., Strickland, M.: Testing viscous and anisotropic hydrodynamics in an exactly solvable case. Phys. Rev. C. {\bf{88}} (2013) 024903.
	\bibitem{GS1} Glassey, R. T., Strauss, W. A.: Asymptotic stability of the relativistic Maxwellian. Publ. Res. Inst. Math. Sci. {\bf 29} (2) (1993) 301--347.
	\bibitem{GS2} Glassey, R. T., Strauss, W. A.: Asymptotic stability of the relativistic Maxwellian via fourteen moments. Trans. Th. Stat. Phys. {\bf 24} (4--5) (1995) 657--678.
	
	\bibitem{Guo 1}  Guo, Y.: Boltzmann diffusive limit beyond the Navier-Stokes approximation. Comm. Pure Appl. Math. {\bf{59}} (2006) 626--687.   
	\bibitem{Guo whole} Guo, Y.: The Boltzmann equation in the whole space. Indiana Univ. Math. J. {\bf 53} (2004). no.4, 1081-1094.
	\bibitem{Guo VMB} Guo, Y.: The Vlasov-Maxwell-Boltzmann system near Maxwellians. Invent. Math. {\bf 153} (2003) no.3, 593-630.
	\bibitem{Guo Strain Momentum} Guo, Y., Strain, Robert M.: Momentum regularity and stability of the relativistic Vlasov-Maxwell-Boltzmann system. Comm. Math. Phys. {\bf 310} (3) (2012) 649--673.
	
	\bibitem{Guo-Wang} Guo, Y., Wang, Y.: Decay of dissipative equations and negative Sobolev spaces. Comm. Partial Differential Equations {\bf 37} (2012) no. 12, 2165–2208.
	
	\bibitem{HM} Hakim, R., Mornas, L.: Collective effects on transport coefficients of relativistic nuclear matter. Phys. Rev. C. {\bf{47}} (1993) 2846.
	
	\bibitem{Holway} Holway, L.H.: Kinetic theory of schock structure using and ellipsoidal distribution function. Rarefied Gas Dynamics, Vol. I
	(Proc. Fourth Internat. Sympos., Univ. Toronto, 1964), Academic Press, New York, (1966), pp. 193-215.
	
	\bibitem{HY2} Hwang, B.-H., Yun, S.-B.: Anderson--Witting model of the relativistic Boltzmann equation near equilibrium. J. Stat. Phys. {\bf{176}}(4) (2019) 1009--1045.
	\bibitem{HY3} Hwang, B.-H., Yun, S.-B.: Stationary solutions to the Anderson--Witting model of the relativistic Boltzmann equation in a bounded interval. SIAM J. Math. Anal. {\bf 53} (1) (2021) 730--753.
	\bibitem{HY1} Hwang, B.-H., Yun, S.-B.: Stationary solutions to the boundary value problem for the relativistic BGK model in a slab. Kinet. Relat. Models {\bf{12}}(4) (2019) 749--764.
	\bibitem{JRS} Jaiswal, A., Ryblewski, R., Strickland, M.: Transport coefficients for bulk viscous evolution in the relaxation time approximation. Phys. Rev. C. {\bf{90}} (2014) 044908.
	\bibitem{JSY} Jang, J. W., Strain, Robert M., Yun, S.-B.: Propagation of uniform upper bounds for the spatially homogeneous relativistic Boltzmann equation. Preprint. Available at https://arxiv.org/abs/1907.05784.	
	\bibitem{JY} Jang, J. W., Yun, S.-B.: Gain of regularity for the relativistic collision operator. Appl. Math. Lett. {\bf{90}} (2019) 162--169.
	\bibitem{Jiang2} Jiang, Z.: On the Cauchy problem for the relativistic Boltzmann equation in a periodic box: global existence. Transport Theory Statist. Phys. {\bf 28} (6) (1999) 617--628.
	\bibitem{Jiang1} Jiang, Z.: On the relativistic Boltzmann equation, Acta Math. Sci. {\bf 18} (3) (1998) 348--360.
	\bibitem{Juttner} J\"{u}ttner, F.: Das Maxwellsche Gesetz der Geschwindigkeitsverteilung in der Relativtheorie. Ann. Phys. {\bf{339}} (1911) 856--882.
	\bibitem{LL} Landau, L. D., Lifshitz, E. M.: Fluid Mechanics. Pergamon Press. (1959).
	\bibitem{LR} Lee, H., Rendall, A.: The spatially homogeneous relativistic Boltzmann equation with a hard potential, Comm. Partial Differential Equations {\bf 38} (12) (2013) 2238--2262.
	
	\bibitem{LMR}  Liu, I.-S.,  M\"uller, I.,  Ruggeri, T.:  Relativistic thermodynamics of gases.  Ann. Phys., {\bf 169}, (1986) 191--219.
	
	\bibitem{Mar3} Marle, C.: Modele cin\'{e}tique pour l\textquoteright \'{e}tablissement des lois de la conduction de la
	chaleur et de la viscosit\'{e} en th\'{e}orie de la relativit\'{e}. C. R. Acad. Sci. Paris {\bf{260}} (1965) 6539--6541.
	
	
	\bibitem{Mar2} Marle, C.: Sur l\textquoteright\'{e}tablissement des equations de l\textquoteright hydrodynamique des fluides relativistes dissipatifs. II. M\'{e}thodes de r\'{e}solution approch\'{e}e de l\textquoteright equation de Boltzmann
	relativiste. Ann. Inst. Henri Poincar\'{e} {\bf 10} (1969) 127--194.
	\bibitem{PR} Pennisi, S., Ruggeri, T.: A new BGK model for relativistic kinetic theory of monatomic and polyatomic gases. J. Phys. Conf. Ser. {\bf{1035}} (2018) 012005.
	\bibitem{PR5} Pennisi, S., Ruggeri, T.: Classical limit of relativistic moments associated with Boltzmann–Chernikov equation: Optimal choice of moments in classical theory. J. Stat. Phys. {\bf{179}} (2020) 231--246.
	
	\bibitem{PR6} Pennisi, S., Ruggeri, T.: Relativistic Eulerian rarefied gas with internal structure. J. Math. Phys. {\bf{59}} (2018) 043102.
	
	\bibitem{PR2} Pennisi, S., Ruggeri, T.: Relativistic extended thermodynamics of rarefied polyatomic gas. Ann. Phys. {\bf{377}} (2017) 414--445.
	


	\bibitem{newbook} Ruggeri, T., Sugiyama, M.: Classical and Relativistic Rational Extended Thermodynamics of Gases. Springer, Heidelberg, New York, Dordrecht, London, ISBN 978-3-030-59143-4, (2021).
	
	\bibitem{RXZ} Ruggeri, T., Xiao, Q., Zhao, H.: Nonlinear hyperbolic waves in relativistic gases of massive particles with Synge energy.   Arch. Rational Mech. Anal.  {\bf 239}, (2021) 1061-1109.
		
	
	\bibitem{SS} Speck, J., Strain, Robert M.: Hilbert expansion from the Boltzmann equation to relativistic fluids. Comm. Math. Phys. {\bf 304} (2011) 229--280.
	\bibitem{Strain Yun} Strain, Robert M., Yun, S.-B.: Spatially homogenous Boltzmann equation for relativistic particles, SIAM J. Math. Anal. {\bf 46} (1) (2014) 917--938.
	\bibitem{Strain Zhu} Strain, Robert M., Zhu, K.: Large-time decay of the soft potential relativistic Boltzmann equation in $\mathbb{R}^3_x$. Kinet. Relat. Models {\bf 5} (2) (2012) 383--415.
	
	\bibitem{Strain1} Strain, Robert M.: Asymptotic stability of the relativistic Boltzmann equation for the soft potentials. Comm. Math. Phys. {\bf 300} (2010) 529--597.
	\bibitem{Strain2} Strain, Robert M.: Global Newtonian limit for the relativistic Boltzmann equation near vacuum. SIAM J. Math. Anal. {\bf 42} (2010) 1568--1601.
	\bibitem{W} Wennberg, B.: The geometry of binary collisions and generalized Radon transforms. Arch. Rational Mech. Anal. {\bf139} (1997) no. 3, 291--302.
	
	

\end{thebibliography}

\end{document}